\documentclass{amsart} 
\usepackage[foot]{amsaddr}
\usepackage{amsthm, amsmath, mathtools}
\usepackage{amssymb}
\usepackage{amscd}
\usepackage{enumerate}
\usepackage{mathrsfs} 
\usepackage[curve,matrix,arrow,frame,tips]{xy}
\usepackage{colonequals}
\usepackage{verbatim}
\usepackage{pdfsync}
\usepackage{graphicx}
\usepackage[usenames,dvipsnames]{color}

\usepackage{url}
\usepackage{hyperref}

\usepackage{fullpage}

\usepackage[normalem]{ulem}

\RequirePackage{tikz-cd}
\RequirePackage{amssymb}
\usetikzlibrary{calc}
\usetikzlibrary{decorations.pathmorphing}

\tikzset{curve/.style={settings={#1},to path={(\tikztostart)
    .. controls ($(\tikztostart)!\pv{pos}!(\tikztotarget)!\pv{height}!270:(\tikztotarget)$)
    and ($(\tikztostart)!1-\pv{pos}!(\tikztotarget)!\pv{height}!270:(\tikztotarget)$)
    .. (\tikztotarget)\tikztonodes}},
    settings/.code={\tikzset{quiver/.cd,#1}
        \def\pv##1{\pgfkeysvalueof{/tikz/quiver/##1}}},
    quiver/.cd,pos/.initial=0.35,height/.initial=0}

\tikzset{tail reversed/.code={\pgfsetarrowsstart{tikzcd to}}}
\tikzset{2tail/.code={\pgfsetarrowsstart{Implies[reversed]}}}
\tikzset{2tail reversed/.code={\pgfsetarrowsstart{Implies}}}
\tikzset{no body/.style={/tikz/dash pattern=on 0 off 1mm}}

\numberwithin{equation}{section} 
\numberwithin{table}{section}

\theoremstyle{plain}
\newtheorem{theorem}{Theorem}[section]
\newtheorem{proposition}[theorem]{Proposition}
\newtheorem{lemma}[theorem]{Lemma}

\newtheorem{corollary}[theorem]{Corollary}

\theoremstyle{definition}
\newtheorem{definition}[theorem]{Definition}

\newtheorem{nota}[theorem]{Notation}
\newtheorem{assume}[theorem]{Assumption}

\theoremstyle{remark}
\newtheorem{remark}[theorem]{Remark}

\newcommand{\HH}{\mathrm{H}}

\begin{document} 

\title{Holomorphic differentials of alternating four covers}

\author{Frauke M. Bleher}
\address{F.B.: Department of Mathematics\\University of Iowa\\
14 MacLean Hall\\Iowa City, IA 52242-1419\\ U.S.A.}
\email{frauke-bleher@uiowa.edu}

\author{Margarita Bustos Gonzalez}
\address{M.B.G.: School of Mathematical and Statistical Sciences - 1804\\Arizona State University\\
734 W Alameda Drive\\Tempe AZ 85282\\ U.S.A.}

\thanks{Both authors were supported in part by NSF Grant No.\ DMS-1801328. The first author was also supported in part by Simons Foundation grant No.\ 960170 and NSF FRG Grant No.\ DMS-2411703. Frauke M. Bleher is the corresponding author.}

\date{October 17, 2025}


\begin{abstract}
Suppose $k$ is an algebraically closed field of characteristic two, let $A_4$ be an alternating group on four letters, and let $H$ be the unique Sylow two-subgroup of $A_4$. Let $X$ be a smooth projective irreducible curve over $k$ with a faithful $A_4$-action such that the quotient curve $X/H$ is a projective line and the $H$-cover  $X\to X/H$ is totally ramified, in the sense that it is ramified and every branch point is totally ramified. Under these assumptions, we determine the precise $kA_4$-module structure of the space of holomorphic differentials of $X$ over $k$. We show that there are infinitely many different isomorphism classes of indecomposable $kA_4$-modules that can occur as direct summands, and we give precise formulas for the multiplicities with which they occur. 
\end{abstract}

\maketitle


\section{Introduction}
\label{s:intro}

Let $k$ be an algebraically closed field, let $X$ be a smooth projective curve over $k$, and let $\HH^0(X,\Omega_X)$ be the space of holomorphic differentials of $X$ over $k$. Suppose $G$ is a finite group acting faithfully on $X$, which makes $\HH^0(X,\Omega_X)$ into a $kG$-module. It is a classical problem to determine the precise $kG$-module structure of $\HH^0(X,\Omega_X)$. This problem was first posed in 1928 by Hecke in \cite{Hecke1928}. In 1934, Chevalley and Weil solved Hecke's problem in \cite{ChevalleyWeil1934} when the characteristic of $k$ is either zero or a prime number that does not divide $\# G$. However, when the characteristic of $k$ does divide $\# G$, Hecke's problem has not been solved in general.

From now on, we will assume that the characteristic of $k$ is a prime number $p$ that divides $\#G$. Many authors have made some progress in this case, by either placing conditions on the group $G$ or on the ramification data of the cover $X\to X/G$. We refer to \cite{ValentiniMadan1981,Kani1986,Nakajima1986,RzedowskiCVillaSMadan1996,Kock2004,KaranikolopoulosKontogeorgis2013,MarquesWard2018,BleherChinburgKontogeorgis2020,Garnek2022,BleherCamacho2023,Kock2024} for a sample of previous results. 

The goal of this paper is to study the case when the characteristic of $k$ is two and $G$ is the alternating group $A_4$ on four letters. This group is a semidirect product with normal subgroup given by the unique Sylow $2$-subgroup $H\cong \mathbb{Z}/2\times\mathbb{Z}/2$ of $G$ on which a cyclic group $C\cong\mathbb{Z}/3$ faithfully acts. We assume that the quotient curve $X/H$ is the projective line over $k$, that the cover $X\to X/H$  is ramified and that all branch points are totally ramified (see Assumption \ref{ass:general}). In particular, this includes all so-called Harbater-Katz-Gabber $A_4$-covers over $k$ (see Definition \ref{def:HKGcovers} and Lemma \ref{lem:HKGcovers}). Under these assumptions, we determine the precise $kG$-module structure of $\HH^0(X,\Omega_X)$, by describing all indecomposable $kG$-modules that occur as direct summands together with their multiplicities (see Theorem \ref{thm:main} and Corollary \ref{cor:main}).

This paper, together with \cite{BleherCamacho2023}, should be seen as the first step in a program concerning the precise $kG$-module structure of $\HH^0(X,\Omega_X)$ when the group $G$ acting on $X$ has tame representation type over $k$, i.e. when the characteristic of $k$ is $2$ and the Sylow $2$-subgroups of $G$ are either dihedral, semidihedral or generalized quaternion (see \cite{BondarenkoDrozd1977, Brenner1970, Higman1954}). 
Because of the Conlon induction theorem \cite[Thm. (80.51)]{CurtisReiner}, any $kG$-module is uniquely determined by its restrictions to all so-called 2-hypo-elementary subgroups of $G$. These are precisely the subgroups of $G$ that are semidirect products with normal $2$-subgroups on which cyclic groups of odd order act. 

For this reason, we will assume now that $G$ is such a 2-hypo-elementary group with normal Sylow $2$-subgroup $H$ and corresponding quotient group $G/H=\langle \rho H\rangle$ for an element $\rho\in G$ having odd order.
As seen in \cite{BleherCamacho2023}, unlike the case of finite representation type studied in \cite{BleherChinburgKontogeorgis2020}, in the tame case the ramification data of the cover $X\to X/G$, consisting of the lower ramification groups together with the fundamental characters of the closed points of $X$, does in general not fully determine the $kG$-module structure of $\HH^0(X,\Omega_X)$. 

Therefore, we focus now on the simplest case when the quotient curve $X/H$ is a projective line, which we will assume from now on.
The first goal is to determine the $kH$-module structure of $\HH^0(X,\Omega_X)$. For example, if $H\cong \mathbb{Z}/2\times \mathbb{Z}/2$, then it was shown in \cite[Thm. 1.2]{BleherCamacho2023} that the lower ramification groups of the cover $X\to X/H=\mathbb{P}^1_k$ do in fact determine the $kH$-module structure of $\HH^0(X,\Omega_X)$.
The second goal is then to answer the question what other knowledge beyond the restriction to $H$ determines the full $kG$-module structure of $\HH^0(X,\Omega_X)$. In the case of finite representation type studied in \cite{BleherChinburgKontogeorgis2020}, the radical layers of $\HH^0(X,\Omega_X)$ with respect to the action of $H$ together with the action of $G/H$ on these layers determine the full $kG$-module structure. However, this approach does not work directly in the tame case, since the radical layers of the indecomposable $kH$-modules are not all one-dimensional over $k$. 

For this reason, it makes sense to first consider cases in which each indecomposable $kG$-module can be decomposed explicitly into a direct sum of indecomposable $kH$-modules and the action of $\rho$ on these indecomposable $kH$-module summands can be used to determine precisely the indecomposable $kG$-module restricting to their direct sum. 
The smallest example of this kind with $[G:H]>1$ is the case when $G\cong A_4$, which is the subject of this paper.

We obtain the following result (see Theorem \ref{thm:main} and Corollary \ref{cor:main} for the precise details):

\begin{theorem}
\label{thm:fakemain}
Let $G$ be an alternating four group with unique Klein four Sylow $2$-subgroup $H$.
Suppose $X\to X/G$ is a $G$-cover of smooth projective curves such that the $H$-subcover $X\to X/H$ is totally ramified and $X/H=\mathbb{P}^1_k$. Then the $kG$-module structure of $\HH^0(X,\Omega_{X/k})$ is fully determined by
\begin{itemize}
\item[(i)] the lower ramification groups of the closed points of $X$ that ramify in the cover $X\to X/H$, and
\item[(ii)] the action of an element $\rho\in G$ of order $3$ on the uniformizers of the branch points of the cover $X\to X/H$.
\end{itemize}
More precisely, the isomorphism classes of the indecomposable $kG$-modules that actually occur as direct summands of $\HH^0(X,\Omega_{X/k})$ together with their multiplicites can be described explicitly. Furthermore, the list of these isomorphism classes is infinite and contains indecomposable $kG$-modules of arbitrarily large finite $k$-dimension.
\end{theorem}

This paper is organized as follows. 

In \S~\ref{s:prelim}, we first recall some definitions and results in characteristic two on Klein four covers, using Artin-Schreier theory and \cite{BleherCamacho2023}, and on cyclic degree three covers, using Kummer theory. We then prove Theorem \ref{thm:GaloisA4}, which classifies the precise structure of alternating four covers in characteristic two. 
In \S~\ref{s:holom}, we focus on alternating four covers $X\to X/G$ over an algebraically closed field $k$ of characteristic 2, i.e. $G\cong A_4$, satisfying the above assumptions (see Assumption \ref{ass:general}). We first show in Lemmas \ref{lem:branchpoints} and \ref{lem:valuesbranchpoints} how the action of the cyclic group $C =G/H \cong\mathbb{Z}/3$ on the Sylow 2-subgroup $H\cong \mathbb{Z}/2\times\mathbb{Z}/2$ of $G$ affects the ramification data of the Klein four cover $X\to X/H$ determined in \cite[Section 3]{BleherCamacho2023}. This allows us to restate \cite[Theorem 3.7]{BleherCamacho2023} in this setting. We then use this to prove our main result, Theorem \ref{thm:main}, which determines the precise $kG$-module structure of $\HH^0(X,\Omega_X)$. As a consequence, we obtain in Corollary \ref{cor:main} the precise $kG$-module structure of $\HH^0(X,\Omega_X)$ when $X$ is a Harbater-Katz-Gabber $A_4$-cover over $k$.
In \S~\ref{s:examples}, we give in Proposition \ref{prop:list} the complete (infinite) list of isomorphism classes of indecomposable $kG$-modules that actually occur as direct summands of $\HH^0(X,\Omega_X)$. We prove this result by giving examples that show that all the modules in this list occur. 
Finally, in an appendix \S~\ref{s:repV4A4}, we describe the indecomposable modules for $kH$ and $kG$, given by so-called string and band modules. We moreover show in Lemmas \ref{lem:bandconnection} and \ref{lem:restrictA4toK4} how these modules are related via induction and restriction.

Part of this paper is the Ph.D. thesis of the second author under the supervision of the first (see \cite{MargaritaThesis2025}). 


\section{Preliminaries}
\label{s:prelim}

In this section, we first recall in \S~\ref{ss:K4} and \S~\ref{ss:Z3} some definitions and results that will be needed for the remainder of the paper. Moreover, in \S~\ref{ss:A4} we prove Theorem \ref{thm:GaloisA4} about the structure of alternating four covers  in characteristic two.

Let $A_4$ denote the alternating group on 4 letters.
We will call a field extension $L/J$ an \emph{alternating four cover} if $L$ is Galois over $J$ and its Galois group is isomorphic to $A_4$.
If $J$ and $L$ are function fields of smooth projective curves, we will call the corresponding Galois cover of these curves also an \emph{alternating four cover}.

Throughout this paper, we make the following assumptions and use the following notations.

\begin{assume}
\label{ass:general}
Let $k$ be an algebraically closed field of characteristic 2, let $G \cong A_4$ be an alternating group on four letters, let $H< G$ be the unique Sylow two-subgroup of $G$, and let $C=G/H\cong\mathbb{Z}/3$. Suppose $\pi:X\to Z$ is an alternating four cover of smooth projective curves over $k$ with Galois group $G$. We write $\pi=\pi_2\circ\pi_1$ for two covers
$$\pi_1:X\to Y \quad\mbox{and}\quad \pi_2:Y\to Z$$
where $\pi_1:X\to Y$ is an $H$-cover and $\pi_2:Y\to Z$ is a $C$-cover. 
We assume that $\pi_1$ is totally ramified, in the sense that $\pi_1$ is ramified and that every branch point of $\pi_1$ is totally ramified. Moreover, we assume that $Y=\mathbb{P}^1_k$ is a projective line over $k$.
\end{assume}

\begin{nota}
\label{not:general}
Under Assumption \ref{ass:general}, we write
$$G=\langle\sigma,\rho\;:\;\sigma^2=\rho^3=(\sigma\rho)^3=1\rangle\cong A_4.$$
Moreover, defining
$$\tau := \rho\sigma\rho^{-1}\quad\mbox{and} \quad\overline{\rho} := \rho H,$$
we write
$$H=\langle\sigma,\tau \rangle\quad\mbox{and}\quad C=\langle \overline{\rho}\rangle.$$
\end{nota}


\subsection{Totally ramified Klein four covers in characteristic two}
\label{ss:K4}

\begin{remark} 
\label{rem:K4covers}
Let $L/K$ be a Galois extension of fields of characteristic two with Galois group $H$. Then $L^{\langle\sigma\rangle}$, $L^{\langle\tau\rangle}$ and $L^{\langle\sigma\tau\rangle}$ are the three intermediate fields of degree two over $K$. By Artin-Schreier theory (see \cite[Proposition 3.7.8]{Stichtenoth2009}), there exist $u,v\in L$ such that 
$$L^{\langle\sigma\rangle}=K(u) \quad \mbox{and} \quad  L^{\langle\tau\rangle}=K(v), $$
where
$$u^2-u=\alpha\in K \quad \mbox{and} \quad v^2-v=\beta\in K, $$ and none of $\alpha,\beta$ can be written as $\xi^2-\xi$ for any $\xi\in K$. Therefore,
$$\begin{array}{rccrcc}
\sigma u&=&u, & \tau u&=&u+1, \\
\sigma v&=&v+1, & \tau v&=&v .
\end{array}$$
In particular, this implies that 
$$L=K(u,v).$$
Moreover, $K(u+v)=L^{\langle \sigma\tau\rangle}$ and $(u+v)^2-(u+v) = \alpha+\beta\in K$, where $\alpha+\beta \ne \xi^2-\xi$ for any $\xi\in K$.
\end{remark}

The next remark is a direct consequence of \cite[Remark 2.7]{BleherCamacho2023}.

\begin{remark} 
\label{rem:BC2023}
Let $\pi_1:X\to Y$ be as in  Assumption \ref{ass:general}. Defining $K:=k(Y)$ and $L:=k(X)$, we have a Galois extension $L/K$ with Galois group $H$. We let $u,v,\alpha,\beta$ be as in Remark \ref{rem:K4covers}. 

Let $y\in Y$ be a branch point of $\pi_1$ and let $x\in X$ be the point above $y$. There exist positive odd integers $m_y\leq M_y$ such that the lower ramification groups at $x$ are given by 
$$H=H_{x,0}= H_{x,1}=\dots= H_{x,m_y}>H_{x,m_y+1}=\dots=H_{x,m_y+2(M_y-m_y)}>H_{x,m_y+2(M_y-m_y)+1}=1.$$
Furthermore, for $\gamma\in\{\alpha,\beta,\alpha+\beta\}$, there exists $h_{\gamma,y}\in K$ such that either
\begin{enumerate}
\item[(i)] $m_y=M_y$ and $\operatorname{ord}_y(\gamma-(h_{\gamma,y}^2-h_{\gamma,y}))=-m_y$ for every $\gamma$,   or
\item[(ii)] $m_y<M_y$ and either
\begin{enumerate}
\item[(a)] $\operatorname{ord}_y(\alpha-(h_{\alpha,y}^2-h_{\alpha,y}))=-m_y$ and $\operatorname{ord}_y(\beta-(h_{\beta,y}^2-h_{\beta,y}))=-M_y=\operatorname{ord}_y(\alpha+\beta-(h_{\alpha+\beta,y}^2-h_{\alpha+\beta,y})),$ or
\item[(b)] $\operatorname{ord}_y(\beta-(h_{\beta,y}^2-h_{\beta,y}))=-m_y$ and $\operatorname{ord}_y(\alpha-(h_{\alpha,y}^2-h_{\alpha,y}))=-M_y=\operatorname{ord}_y(\alpha+\beta-(h_{\alpha+\beta,y}^2-h_{\alpha+\beta,y})),$ or
\item[(c)] $\operatorname{ord}_y(\alpha+\beta-(h_{\alpha+\beta,y}^2-h_{\alpha+\beta,y}))=-m_y$ and $\operatorname{ord}_y(\alpha-(h_{\alpha,y}^2-h_{\alpha,y}))=-M_y=\operatorname{ord}_y(\beta-(h_{\beta,y}^2-h_{\beta,y})).$
\end{enumerate}
\end{enumerate}
Define 
$$\left\{\begin{array}{rcccrccl}
u_y &=& u-h_{\alpha,y}&\mbox{and}& \alpha_y &=& \alpha-(h_{\alpha,y}^2-h_{\alpha,y}) & \mbox{in the cases  (i)  or  (ii)(a)},\\
u_y &=& v -h_{\beta,y} &\mbox{and}& \alpha_y &=& \beta-(h_{\beta,y}^2-h_{\beta,y}) & \mbox{in the case  (ii)(b)},\\
u_y &=& u+v - h_{\alpha+\beta,y}  &\mbox{and}& \alpha_y &=& \alpha+\beta-(h_{\alpha+\beta,y}^2-h_{\alpha+\beta,y}) & \mbox{in the case  (ii)(c)},
\end{array}\right.$$
and define
$$\left\{\begin{array}{rcccrccl}
v_y &=& v-h_{\beta,y} & \mbox{and}&\beta_y&=&\beta-(h_{\beta,y}^2-h_{\beta,y}) & \mbox{in the cases  (i)  or  (ii)(a)  or  (ii)(c)},\\
v_y &=&u-h_{\alpha,y} & \mbox{and} &\beta_y&=&\alpha-(h_{\alpha,y}^2-h_{\alpha,yx}) & \mbox{in the case  (ii)(b)}.
\end{array}\right.$$
Then
$$\operatorname{ord}_x(u_y)=-2m_y \quad \mbox{and}  \quad \operatorname{ord}_x(v_y)=-2M_y.$$
Moreover, there exist $\theta_y,\eta_y\in K$
such that $w_y=v_y+\theta_yu_y+\eta_y$ is an Artin-Schreier generator of $L$ over $K(u_y)$ with the following properties:
$$\operatorname{ord}_x(w_y)=-m_y-2(M_y-m_y),\quad \operatorname{ord}_x(\theta_y)=-2(M_y-m_y),$$
and
$$\left\{\begin{array}{rcccrccl}
(\sigma -1)w_y&=&1 & \mbox{and} &(\tau-1)w_y&=&\theta_y & \mbox{in the cases (i) or (ii)(a)}, \\
(\tau -1)w_y&=&1 & \mbox{and} &(\sigma-1)w_y&=&\theta_y & \mbox{in the case (ii)(b)}, \\
(\sigma\tau-1)w_y&=&1 & \mbox{and} & (\tau-1)w_y&=&\theta_y& \mbox{in the case (ii)(c)}.
\end{array}\right.$$
\end{remark}

\begin{nota} 
\label{not:uniform}
Assume the notation from Remark \ref{rem:BC2023}, and let $\pi_y$ be a uniformizer at $y$. We can write $\alpha_y,\beta_y,\theta_y,\eta_y$ as Laurent series of $k[[\pi_y,\pi_y^{-1}]]$ as follows, where we use that $\mathrm{ord}_y(\eta_y)\ge (-M_y+1)/2$ by \cite[Lemma 2.8]{BleherCamacho2023}:
\begin{eqnarray*}
\alpha_y &=& \pi_y^{-m_y}\sum_{i\geq0}\alpha_{y,i}\pi_y^i,\\
\beta_y &=& \pi_y^{-M_y}\sum_{i\geq0}\beta_{y,i}\pi_y^i,\\
\theta_y &=& \pi_y^{-(M_y-m_y)/2}\sum_{i\geq0}\theta_{y,i}\pi_{y}^i,\\
\eta_y &=& \pi_y^{(-M_y+1)/2}\sum_{i\geq0}\eta_{y,i}\pi_y^i.
\end{eqnarray*}
For $j\geq0$, we define
\begin{eqnarray*}
\theta_y(j) &:=& \pi_y^{-(M_y-m_y)/2}\left(\sum_{i=0}^j\theta_{y,i}\pi_y^i\right),\\[1ex]
\widetilde{\eta}_y &:= &\pi_y^{(-M_y+1)/2}\left(\sum_{i=0}^{\big\lfloor\frac{2m_y+3}{4}\big\rfloor-\big\lfloor\frac{m_y+3}{4}\big\rfloor-1}\eta_{y,i}\pi_y^i\right),\\[1ex]
w_y(j) &:=&v_y+\theta_y(j)u_y+\widetilde{\eta}_y.
\end{eqnarray*}
\end{nota}

\begin{remark}
\label{rem:doweneedthis?}
By \cite[Lemma 2.8]{BleherCamacho2023}, we have $\alpha_{y,0},\beta_{y,0},\theta_{y,0}\in k^{\times}$. 
Morever, by \cite[Equation (2.8) in Lemma 2.8]{BleherCamacho2023}, 
$$\beta_{y,2i} + \sum_{i_1+i_2=i}\alpha_{y,2i_1}\theta_{y,i_2}^2=0\quad\mbox{for $0\le i\le \lfloor \frac{m_y}{4}\rfloor$}.$$
If $m_y=M_y$ then $\theta_{y,0}\in k-\{0,1\}$.
\end{remark}

\begin{nota} 
\label{not:deltalambda}
Assuming the above notation, we define the following:
\begin{enumerate}
\item[(i)] If we are in the situation of Remark \ref{rem:BC2023}(i), we define $\lambda_y:=\theta_{y,0}$ and
$$\delta_y:=\begin{cases}
\operatorname{ord}_y(\theta_y-\lambda_y) & \mbox{if } 1\leq\operatorname{ord}_y(\theta_y-\lambda_y)\leq\lfloor\frac{m_y}{4}\rfloor,\\
0 & \mbox{otherwise}.
\end{cases}$$
By \cite[Lemma 2.8]{BleherCamacho2023},  $\lambda_y\in k-\{0,1\}$. 
\item[(ii)] If we are in the situation of Remark \ref{rem:BC2023}(ii), we define $\delta_y:=-1$ and 
$$\lambda_y:=\begin{cases}
\infty & \mbox{if we are in the situation of Remark \ref{rem:BC2023}(ii)(a)},\\ 
 0 & \mbox{if we are in the situation of Remark \ref{rem:BC2023}(ii)(b)},\\ 
1 & \mbox{if we are in the situation of Remark \ref{rem:BC2023}(ii)(c)}.
\end{cases}$$
\end{enumerate}
\end{nota}

By Assumption \ref{ass:general}, $Y=\mathbb{P}^1_k$ is the projective line over $k$.
This case was studied in \cite[Section 3]{BleherCamacho2023} and the $kH$-module structure of $\HH^0(X,\Omega_X)$ was determined.
We now recall the notation and definitions from \cite[Section 3]{BleherCamacho2023} that will be important for this paper.
We write the function field 
$$K := k(Y ) = k(\mathbb{P}^1_k) =k(s)$$ for a variable $s$.
For $\mu\in k\cup\{\infty\}$, let $y_{\mu}$ be the corresponding closed point in $Y=\mathbb{P}_k^1$ with  uniformizer
\begin{equation}
\label{eq:P1uniformizer}
\pi_{y_{\mu}}=\begin{cases}
s^{-1} & \mbox{ if } \mu=\infty,\\
s-\mu & \mbox{ if } \mu\neq\infty.
\end{cases}
\end{equation}
We will also write $\infty$ instead of $y_{\infty}$.

Because of Assumption \ref{ass:general} and Remark \ref{rem:BC2023}, we can change the variable $s$, if necessary, to be able to make the following assumptions throughout this section (see \cite[Assumption 3.1 and Remark 3.2]{BleherCamacho2023}).

The finite set of branch points is given by 
\begin{equation}
\label{eq:Ybr}
Y_{\mathrm{br}}=\{\infty,y_1,\dots,y_n\}
\end{equation}
for some $n\ge 0$, where each $y_i$ corresponds to $s-\mu_i$ for $\mu_i\in k^{\times}$. For $\gamma\in\{\alpha,\beta,\alpha+\beta\}$, there exists a rational function $h_\gamma\in K=k(s)$ and
\begin{equation}
\label{eq:standardform}
\gamma-(h_\gamma^2-h_\gamma)=s^{p_{\gamma,\infty}+p_{\gamma,y_1}+\dots+p_{\gamma,y_n}-d_\gamma}(\mu_1^{p_{\gamma,y_1}}\dots\mu_n^{p_{\gamma,y_n}})\frac{c_{\gamma,d_\gamma}s^{d_\gamma}+\dots+c_{\gamma,1}s+c_{\gamma,0}}{(s-\mu_1)^{p_{\gamma,y_1}}\dots(s-\mu_n)^{p_{\gamma,y_n}}}
\end{equation}
where $p_{\gamma,\infty}\in\{m_\infty,M_\infty\}$, $p_{\gamma,y_i}\in\{m_{y_i},M_{y_i}\}$ for $i=1,\dots,n$ such that the following conditions are satisfied:
\begin{itemize}
\item $0\leq d_{\gamma}\leq p_{\gamma,\infty}+p_{\gamma,y_1}+\dots+p_{\gamma,y_n}$,
\item $c_{\gamma,d_\gamma},c_{\gamma,0}\in k^{\times}$, and 
\item $c_{\gamma,d_\gamma}s^{d_\gamma}+\dots+c_{\gamma,0}$ and $(s-\mu_1)^{p_{\gamma,y_1}}\dots(s-\mu_n)^{p_{\gamma,y_n}}$ are relatively prime in $k[s]$.
\end{itemize}

\begin{remark} 
\label{rem:important}
Let $u,v,\alpha,\beta$ be as in Remark \ref{rem:BC2023}, and let $h_\gamma$ be as in (\ref{eq:standardform}) for $\gamma\in\{\alpha,\beta,\alpha+\beta\}$.
Define
$$\begin{array}{rclrcl}
\widetilde{u}&:=&u-h_\alpha, \quad & \widetilde{\alpha}&:=&\alpha-(h_\alpha^2-h_\alpha),\\
\widetilde{v}&:=&v-h_\beta, \quad & \widetilde{\beta}&:=&\beta-(h_\beta^2-h_\beta),\\
\widetilde{u+v}&:=&u+v-h_{\alpha+\beta}, \quad & \widetilde{\alpha+\beta}&:=&\alpha+\beta-(h_{\alpha+\beta}^2-h_{\alpha+\beta}).
\end{array}$$
Fix $y\in Y_{\mathrm{br}}$. By \cite[Remark 3.3]{BleherCamacho2023}, we obtain the following for $u_y,v_y,\alpha_y,\beta_y$ and $m_y,M_y$ from Remark \ref{rem:BC2023}:
\begin{itemize}
\item If $p_{\alpha,y}=p_{\beta,y}=p_{\alpha+\beta,y}$ or $p_{\alpha,y} < p_{\beta,y}=p_{\alpha+\beta,y}$, then $m_y=p_{\alpha,y}$ and $M_y=p_{\beta,y}=p_{\alpha+\beta,y}$, and
$$u_y = \widetilde{u},\alpha_y=\widetilde{\alpha}, \quad v_y = \widetilde{v}, \beta_y=\widetilde{\beta}.$$
\item If $p_{\beta,y}<p_{\alpha,y}=p_{\alpha+\beta,y}$, then $m_y=p_{\beta,y}$ and $M_y=p_{\alpha,y}=p_{\alpha+\beta,y}$, and
$$u_y = \widetilde{v},\alpha_y=\widetilde{\beta}, \quad v_y = \widetilde{u}, \beta_y=\widetilde{\alpha}.$$
\item If $p_{\alpha+\beta,y}<p_{\alpha,y} = p_{\beta,y}$, then $m_y=p_{\alpha+\beta,y}$ and $M_y=p_{\alpha,y}=p_{\beta,y}$, and
$$u_y = \widetilde{u+v},\alpha_y=\widetilde{\alpha+\beta}, \quad v_y = \widetilde{v}, \beta_y=\widetilde{\beta}.$$
\end{itemize}
\end{remark}


\subsection{Cyclic covers of degree three in characteristic two}
\label{ss:Z3}

\begin{remark} 
\label{rem:Z3}
Let $J$ be a field of characteristic two containing a primitive cube root of unity. Let $K/J$ be a Galois extension with Galois group $C$.
By Kummer theory, there exists an element $f \in J$ such that 
$$K=J(\sqrt[3]{f}) \mbox{ with } \sqrt[3]{f}\in K-J.$$ 
Moreover,  the generator $\overline{\rho}$ of $C=\mathrm{Gal}(K/J)$ satisfies 
$$\overline{\rho}(\sqrt[3]{f})=\zeta\sqrt[3]{f}$$  
for a primitive cube root $\zeta$ of unity in $J$. 
\end{remark} 

\begin{remark} 
\label{rem:Zbr}
Let $\pi_2:Y\to Z$ be as in  Assumption \ref{ass:general}.  Defining $J:=k(Z)$ and $K:=k(Y)$, we have a Galois extension $K/J$ with Galois group $C$. We let $f,\sqrt[3]{f},\zeta$ be as in Remark \ref{rem:Z3}. We note that $\zeta\in k$.

Since $\#C = 3$ is relatively prime to the characteristic of $k$, $\pi_2$ is tamely ramified. Let $Z_{\mathrm{br}}$ be the finite set of branch points.
By the Riemann-Hurwitz formula, we have
$$g(Y) - 1 = 3(g(Z)-1) + \#Z_{\mathrm{br}}.$$
\end{remark}

Using Remarks \ref{rem:Z3} and \ref{rem:Zbr}, we obtain the following result.

\begin{lemma} \label{nice function field when Z is P1}
Under Assumption $\ref{ass:general}$, $Z=\mathbb{P}^1_k$ and there exists $s\in K=k(Y)$, $s\not\in J=k(Z)$, such that $s$ is transcendental over $k$, 
$$K=k(Y) = k(s) \quad \mbox{and} \quad J=k(Z) = k(s^3).$$
Moreover, the generator $\overline{\rho}$ of $C=\mathrm{Gal}(K/J)$ satisfies $\overline{\rho}(s)=\zeta s$ for a primitive cube root $\zeta $ of unity in $k$. 
\end{lemma}

\begin{proof}
Since $g(Y)=0$, it follows from Remark \ref{rem:Zbr}  that $g(Z)=0$ and $\#Z_{\mathrm{br}} = 2$. In particular, $J=k(Z) = k(t)$ for a variable $t$. 
Let $f=f(t)\in J=k(t)$ and $\zeta$ be as in Remark \ref{rem:Z3}.
Writing $Z_{\mathrm{br}} = \{z_1,z_2\}$, we have 
$$\mathrm{ord}_{z_1}(f) \equiv  -\mathrm{ord}_{z_2}(f)\mod 3
\;\equiv \;\pm 1 \mod 3$$
and $\mathrm{ord}_z(f) \equiv 0\mod 3$ for all $z\in Z-Z_{\mathrm{br}}$.
Hence there exist $a,b,c,d \in k$ with $ad-bc\ne 0$ and $f_1(t)\in k(t)$ such that
$$f=\frac{at-b}{ct-d} \,(f_1(t))^3.$$
Defining 
$$s:=\sqrt[3]{\frac{at-b}{ct-d}}\in K,$$
the remaining statements of the lemma follow.
\end{proof}


\subsection{Alternating four Galois extensions in characteristic two}
\label{ss:A4}

\begin{theorem} 
\label{thm:GaloisA4}
Let $L/J$ be an extension of fields of characteristic $2$ where $J$ contains all cube roots of unity. Then $L/J$ is Galois with Galois group $G\cong A_4$ if and only if the following two conditions hold:
\begin{itemize}
\item[(1)] There exists an intermediate field $K$ of $L/J$ such that 
\begin{itemize}
\item[(a)] $K=J(\sqrt[3]{f})$ for some $f \in J$ with $\sqrt[3]{f}\not\in J$, and
\item[(b)] $K/J$ is Galois with $\mathrm{Gal}(K/J)=\langle\overline{\rho}\rangle\cong \mathbb{Z}/3$ and $\overline{\rho}(\sqrt[3]{f})=\zeta \sqrt[3]{f}$ for a primitive cube root of unity $\zeta $ in $J$; and
\end{itemize}
    
\item[(2)] $L=K(u,v)$ where $u^2-u=\alpha\in K$ such that
\begin{itemize}
\item[(a)] $\alpha\neq \xi^2-\xi$ for any $\xi \in K$,
\item[(b)] $\mathrm{Tr}_{K/J}(\alpha)=\alpha+\overline{\rho}\alpha+\overline{\rho}^2\alpha=0$, and
\item[(c)] $v^2-v=\overline{\rho}\alpha$.
\end{itemize}
\end{itemize}
\end{theorem}

\begin{proof}
Let $L/J$ be a field extension as in the first sentence of the statement of the theorem.

We first assume that conditions (1) and (2) hold. We want to show that $L/J$ is Galois with Galois group $G\cong A_4$. By condition (2), we have $L=K(u,v)$. Moreover, $u^2-u=\alpha$ such that $\alpha\neq \xi^2-\xi$ for any $\xi \in K$. Since $\overline{\rho}$ is a field automorphism of $K$, it follows that both $\overline{\rho}\alpha$ and $\overline{\rho}^2\alpha$ are also not equal to $\xi^2-\xi$ for any $\xi \in K$. This implies that $K(u)$ and $K(v)$ are both subfields of $L$ that are Galois over $K$ of degree $2$, and that $K(u)\cap K(v)=K$. By Galois theory, the composite $L=K(u) K(v)$ is Galois over $K$ and its Galois group $\mathrm{Gal}(L/K)$ is isomorphic to $$\mathrm{Gal}(K(u)/K)\times \mathrm{Gal}(K(v)/K)\cong \mathbb{Z}/2\times \mathbb{Z}/2.$$
Defining $H:=\mathrm{Gal}(L/K)$, we can write $H=\langle\sigma,\tau:\sigma^2=\tau^2=(\sigma\tau)^2=1\rangle$ such that $\sigma u = u$, $\tau u \ne u$ and $\tau v = v$, $\sigma v \ne v$.  Since $\tau(u^2-u)=\tau\alpha=\alpha$ and $\sigma(v^2-v)=\sigma\rho\alpha=\rho\alpha$, it follows that $\tau u = u+1$ and $\sigma v = v+1$.
Since $\sqrt[3]{f}\in K$, we now have, by condition (1),  that
$$L=K(u,v)=J(\sqrt[3]{f},u,v)$$
where
\begin{equation}
\label{eq:Laction}
\left\{\begin{array}{ccccccccc}
\sigma(\sqrt[3]{f})&=&\sqrt[3]{f}, \quad & \tau(\sqrt[3]{f})&=&\sqrt[3]{f}, \quad& \overline{\rho}(\sqrt[3]{f})&=&\zeta \sqrt[3]{f}, \\
\sigma u&=&u, \quad& \tau u&=&u+1, \quad&  \\
\sigma v&=&v+1, \quad& \tau v&=&v. \quad& 
\end{array}\right.
\end{equation}
Since $\overline{\rho}$ is an automorphism of $K$ over $J$ and since $L/K$ is algebraic, we can extend $\overline{\rho}$ to an embedding $\rho$ from $L$ into an algebraic closure $\overline{L}$ of $L$ that fixes the elements in $J$. By condition (2),
$$v^2-v = \rho\alpha = (\rho u)^2-\rho u$$
and
$$(u+v)^2 - (u+v) = \alpha+\rho\alpha =  \rho^2\alpha = (\rho v)^2-\rho v,$$
which implies 
\begin{equation}
\label{eq:fourrho}
\rho u\in\{v,v+1\}\subset L\quad \mbox{and} \quad \rho v\in\{u+v,u+v+1\}\subset L. 
\end{equation}
Therefore, $\rho$ is an automorphism of $L$ over $J$, and $\langle\sigma,\rho\rangle$ is a subgroup of the automorphism group of $L=J(\sqrt[3]{f},u,v)$ over $J$.

We next show that $G:=\langle\sigma,\rho\rangle \cong A_4$ by showing that the  relations 
$$\sigma^2=\rho^3=(\sigma\rho)^3=1$$
are satisfied. Since 
$$\sigma^2v=\sigma(v+1)=v+1+1=v,$$ 
it follows from (\ref{eq:Laction}) that  $\sigma$ has order $2$. 
Going through the four possibilities for the pair $(\rho u,\rho v)$ from (\ref{eq:fourrho}), we see that in all cases, $\rho^3u = (\sigma\rho)^3u = u$ and $\rho^3v = (\sigma\rho)^3v = v$. Hence it follows from (\ref{eq:Laction}), using $\overline{\rho}=\rho|_K$, that  $\rho$ and $\sigma\rho$ have order $3$, showing that $G=\langle\sigma,\rho\rangle\cong A_4$.
    
Since $\sigma\in \mathrm{Gal}(L/K)$ and $J\subset K$, it follows that the fixed field $L^G=L^{\langle \sigma,\rho\rangle}$ contains $J$. Because $[L:L^G] = \# G = 12 = [L:J]$, we obtain that $L^G = J$.  Hence, by Artin's Theorem, $L/J$ is Galois with Galois group $G=\langle \sigma,\rho\rangle\cong A_4$, which proves the backward direction of Theorem \ref{thm:GaloisA4}.
    
To prove the forward direction of Theorem \ref{thm:GaloisA4}, we now assume that $L/J$ is a Galois extension with Galois group $G\cong A_4$. As in Notation \ref{not:general}, we write $G=\langle\sigma,\rho:\sigma^2=\rho^3=(\sigma\rho)^3=1\rangle$, $\tau=\rho\sigma\rho^{-1}$ and $H=\langle\sigma,\tau\rangle$. Since $H$ is a normal subgroup of $G$, it follows by Galois theory that the fixed field $K:=L^H$ is Galois over $J$ with $\mathrm{Gal}(K/J)= G/H \cong \mathbb{Z}/3$. Since the coset $\rho H$ generates $G/H$, we can use Galois theory to see that the restriction $\overline{\rho}=\rho|_K$ is a generator of the Galois group $\mathrm{Gal}(K/J)$. Since $K/J$ is cyclic of degree 3 and $J$ contains all cube roots of unity, we can use Remark \ref{rem:Z3} to see that there exists an element $f\in J$ such that $K=J(\sqrt[3]{f})$ with $\sqrt[3]{f}\in K- J$. Moreover, the generator $\overline{\rho}=\rho|_K$ of $\mathrm{Gal}(K/J)$ satisfies $\overline{\rho}(\sqrt[3]{f})=\zeta \sqrt[3]{f}$ for a primitive cube root $\zeta $ of unity in $J$. Thus we have shown condition (1).    
    
By Galois theory, $L/K$ is Galois with $\mathrm{Gal}(L/K) = \mathrm{Gal}(L/L^H) = H = \langle\sigma,\tau\rangle$. By Remark \ref{rem:K4covers}, there exist $u^\prime\in L^{\langle\sigma\rangle}$ and $v^\prime\in L^{\langle\tau\rangle}$ such that $L^{\langle\sigma\rangle}=K(u^{\prime})$ and $L^{\langle\tau\rangle}=K(v^{\prime})$, where $(u^\prime)^2-u^\prime=\alpha^\prime\in K$ and $(v^\prime)^2-v^\prime=\beta^\prime\in K$ and none of $\alpha^\prime,\beta^\prime$ can be written as $\xi^2-\xi$ for any $\xi \in K$. Moreover, 
$$\begin{array}{cccccc}
\sigma u^\prime&=&u^\prime,\quad & \tau u^\prime&=&u^\prime+1, \\
\sigma v^\prime&=&v^\prime+1, \quad& \tau v^\prime&=&v^\prime ,
\end{array}$$
and $L=K(u^\prime,v^\prime)$. Additionally, since $u'+v' \not\in L^{\langle \sigma\rangle}\cup L^{\langle \tau\rangle}$, we obtain that $K(u^\prime+v^\prime) = L^{\langle \sigma\tau\rangle}$ and $\alpha'+\beta'\ne \xi^2-\xi$ for any $\xi \in K$.
We have now shown that
$$L=K(u^\prime,v^\prime)=J(\sqrt[3]{f},u^\prime,v^\prime)$$
where $f\in J, \sqrt[3]{f}\in K-J$ and $(u^\prime)^2-u^\prime=\alpha^\prime\in K,(v^\prime)^2-v^\prime=\beta^\prime\in K$ such that none of $\alpha^\prime$, $\beta^\prime$ and $\alpha^\prime+\beta^\prime$ can be written as $\xi^2-\xi$ for any  $\xi\in K$.

Since $\overline{\rho}=\rho|_K$, we have $\mathrm{Tr}_{K/J}(\alpha^\prime)=\alpha^\prime+\rho\alpha^\prime+\rho^2\alpha^\prime$. Because $\alpha^\prime\in K=J(\sqrt[3]{f})$, there exist $j_0,j_1,j_2\in J$ such that 
$$\alpha^\prime=j_0+j_1\sqrt[3]{f}+j_2\sqrt[3]{f^2},$$ 
which implies 
\begin{equation}
\label{eq:oy0}
\mathrm{Tr}_{K/J}(\alpha^\prime)=3j_0+(1+\zeta +\zeta ^2)j_1\sqrt[3]{f}+(1+\zeta ^2+\zeta )j_2\sqrt[3]{f^2}=j_0.
\end{equation} 
Therefore,
$$j_0=\alpha'+\rho\alpha'+\rho^2\alpha'=(u^\prime+\rho u^\prime+ \rho^2 u^\prime)^2-(u^\prime+\rho u^\prime+ \rho^2 u^\prime).$$ 
We next show that $u^\prime+\rho u^\prime+ \rho^2 u^\prime\in J$. This element is certainly in $L$, and, since $J=L^G$ and $G=\langle \sigma,\rho\rangle$, it suffices to show that this element is fixed by $\sigma$ and $\rho$. It is obvious that $\rho$ fixes $u^\prime+\rho u^\prime+ \rho^2 u^\prime$. Since $u^\prime+1=\tau u^\prime=(\rho\sigma\rho^2) u^\prime$, we have
\begin{equation}
\label{eq:oy1}
\sigma\rho^2 u^\prime =\rho^2 (u^\prime+1)=\rho^2 u^\prime+1,
\end{equation}
which implies 
$$\sigma(u^\prime+\rho u^\prime+ \rho^2 u^\prime) = u^\prime+\sigma\rho u^\prime+\rho^2 u^\prime+1.$$
Since $\rho\alpha^\prime\in K$ and $\sigma$ fixes the elements in $K$, we have $\rho\alpha^\prime=\sigma\rho\alpha^\prime=(\sigma\rho u^\prime)^2-\sigma\rho u^\prime$. Since also $\rho\alpha^\prime=(\rho u^\prime)^2-\rho u^\prime$, we obtain 
$$\sigma\rho u^\prime\in\{\rho u^\prime,\rho u^\prime+1\}.$$ 
If $\sigma\rho u^\prime=\rho u^\prime$, then, using (\ref{eq:oy1}), we see that $(\sigma\rho)^3u^\prime=u^\prime+1$, which contradicts that the order of $\sigma\rho$ is 3. Thus $\sigma\rho u^\prime=\rho u^\prime+1$, which implies that $\sigma$ fixes $u^\prime+\rho u^\prime+ \rho^2 u^\prime$, proving that this element lies in $J$.
    
Defining $j:=u^\prime+\rho u^\prime+ \rho^2 u^\prime\in J$, we obtain, using (\ref{eq:oy0}), that
$$j^2 - j = \alpha' +\rho\alpha'+\rho^2\alpha' = \mathrm{Tr}_{K/J}(\alpha') = j_0.$$
Letting $u:=u^\prime + j$ and $\alpha:=\alpha^\prime + j_0$, it follows that $L=K(u,v^\prime)$. Moreover,
$$u^2-u = (u^\prime)^2 - u' + j^2-j = \alpha'+j_0 = \alpha,$$
$\alpha\neq \xi^2-\xi$ for any $\xi\in K$, and  $\mathrm{Tr}_{K/J}(\alpha)=\mathrm{Tr}_{K/J}(\alpha') + 3j_0 = 4 j_0 = 0$. In particular, $u$ and $\alpha$ satisfy conditions (2)(a) and (2)(b). To finish the proof of Theorem \ref{thm:GaloisA4}, we now show that
$$L=K(u,v^\prime)=K(u,\rho u).$$
We have that $(\rho^2 u)^2-\rho^2u=\rho^2\alpha=\rho\alpha+\alpha$, which implies that $\rho^2u\in\{u+\rho u,u+\rho u +1\}$. Since $\rho$ is a field automorphism of $L$, $\rho u\in L=K(u,v^\prime)$. If $\rho u\in K$, then $\rho^2u\in K,$ implying $\rho u+\rho^2u\in K$, which is equivalent to $u\in K$. But this is impossible. On the other hand, using that $\tau=\rho\sigma\rho^{-1}$, we have $$\tau\rho u=\rho\sigma u=\rho u$$ implying that $\rho u\in L^{\langle\tau\rangle}-K$. Since $L^{\langle \tau\rangle}=K(v')$, this means that $K(v')=K(\rho u)$. Therefore, $L=K(u,v^\prime)=K(u, \rho u)$. Furthermore, $(\rho u)^2-\rho u=\rho\alpha=\overline{\rho}\alpha$, which gives us condition (2)(c), completing the proof of Theorem \ref{thm:GaloisA4}.
\end{proof}


\section{Holomorphic differentials of alternating four covers in characteristic two}
\label{s:holom}

Throughout this section, we suppose Assumption \ref{ass:general} holds, and we use Notation \ref{not:general}. Using Theorem \ref{thm:GaloisA4}, we introduce the following notation for the function fields.

\begin{nota}
\label{not:functionfields}
Let $\pi:X\to Z$, $\pi_1:X\to Y$ and $\pi_2:Y\to Z$ be as in Assumption \ref{ass:general}. The function fields of the curves $X$, $Y$ and $Z$ are as follows:
$$\begin{array}{ccccccc}
J&:=&k(Z)&=&k(s^3),&&\\
K&:=&k(Y) &=&J(s) &=& k(s),\\
L&:=&k(X) &=& K(u,\rho u) &=&k(s,u,\rho u) 
\end{array}$$
such that
\begin{itemize}
\item[(i)] $s\in K-J$, $s$ is transcendental over $k$,
\item[(ii)] $u\in L-K$, $u^2-u = \alpha\in K$, $\alpha\neq \xi^2-\xi$ for any $\xi\in K$, and $\mathrm{Tr}_{K/J}(\alpha)=\alpha+ \rho \alpha + \rho^2 \alpha=0$.
\end{itemize}
Moreover, there exists a primitive cube root of unity $\zeta\in k$ such that the actions of the generators $\sigma,\rho$ of $G$ on $s,u,\rho u$ are as follows:
$$\begin{array}{rclrcl}
\sigma (s) &=& s, \quad &\rho (s) &=& \zeta s,  \\
\sigma (u) &=& u, \quad &\rho (u) &=& \rho u, \\
\sigma (\rho u) &=& \rho u+1, \quad & \rho (\rho u)  &\in& \{u+\rho u ,u+\rho u + 1\}.
\end{array}$$ 
Defining
$$v:=\rho u  \quad \mbox{and} \quad \beta:=\rho \alpha$$
we have $v^2-v = \beta$ and $\beta\neq \xi^2-\xi$ for any $\xi\in K$.
Additionally, $\alpha+\beta = \rho^2 \alpha$ and $u+v\in \{\rho^2u,\rho^2 u+1\}$.
\end{nota}

Concerning the number of totally ramified points of $\pi$, we have the following result.

\begin{lemma} 
\label{lem:totram}
Under Assumption $\ref{ass:general}$ and Notations $\ref{not:general}$ and $\ref{not:functionfields}$, there are at most two totally ramified branch points of the cover $\pi:X\to Z$. Moreover, we can assume, without loss of generality, that the points in $Y_{\mathrm{br}}$ above these points are in $\{0,\infty\}$.
\end{lemma}

\begin{proof}
By Remark \ref{rem:Zbr}, the set of branch points $Z_{\mathrm{br}}$ of the cover $\pi_2:Y\to Z$ consists of precisely two points, say $Z_{\mathrm{br}}=\{Q_1,Q_2\}$. For $i\in\{1,2\}$, let $P_i$ be the unique point of $Y$ above $Q_i$. 
Let $Y_{\mathrm{br}}$ be the set of branch points of the cover $\pi_1:X\to Y$.
Since every automorphism of a function field $k(t)$ in one variable over $k$ is given by sending $t$ to a fractional linear transformation $\frac{at+b}{ct+d}$ with $ad-bc\ne 0$, we can move any 3 points on $\mathbb{P}^1_k$ to any other 3 points on $\mathbb{P}^1_k$. Hence we can move  $Q_1$ to $\infty$ and $Q_2$ to $0$. 
This means that if $k(Z)=k(t)$ then $k(Y)=k(t)(s)$ where $s^3=t$ or $s^3=t^{-1}$. 
In other words, $\{P_1,P_2\}=\{0,\infty\}$. 
Since the totally ramified points of $\pi$ on $X$, if they exist, must lie above $P_1$ or $P_2$, this completes the proof.
\end{proof}

By Lemma \ref{lem:totram}, we have to consider the following three cases, where we replace $s$ by $s^{-1}$ if necessary:
\begin{itemize}
\item[(1)] $\pi$ has precisely one totally ramified point and $\pi_1:X\to Y$ is branched over $\infty$ but not over $0$;
\item[(2)] $\pi$ has precisely two totally ramified points and $\pi_1:X\to Y$ is branched over both $\infty$ and $0$;
\item[(3)] $\pi$ has no totally ramified point and $\pi_1:X\to Y$ is branched neither over $\infty$ nor over $0$.
\end{itemize}

Case (1) is particularly important since it includes all so-called Harbater-Katz-Gabber $G$-covers, which we show next. Harbater-Katz-Gabber covers are of independent interest; see, for example, \cite{KontogeorgisTsouknidas2020,Obus2017} and their references. 
We first recall their definition (see, for example, \cite[Section 4B]{BCPS2017}).

\begin{definition}
\label{def:HKGcovers}
A $G$-cover $\tau: X\to Z$ of smooth projective curves over $k$ is called a Harbater-Katz-Gabber $G$-cover, if the following three conditions are satisfied:
\begin{itemize}
\item[(i)] $Z=\mathbb{P}^1_k$;
\item[(ii)] $\tau$ has precisely one totally ramified point $x\in X$; and 
\item[(iii)] the action of $G$ on $X-\{x\}$ is either unramified everywhere, or it is tamely and non-trivially ramified at one $G$-orbit in $X-\{x\}$ and unramified everywhere else. 
\end{itemize}
\end{definition}

\begin{lemma} 
\label{lem:HKGcovers}
Using Notation $\ref{not:general}$, every Harbater-Katz-Gabber $G$-cover of smooth projective curves over $k$ is an alternating four cover $\pi: X\to Z$ with Galois group $G$ with exactly one totally ramified point satisfying Assumption $\ref{ass:general}$.
\end{lemma}

\begin{proof}
Suppose $\tau: X\to \mathbb{P}^1_k=:Z$ is a Harbater-Katz-Gabber $G$-cover, as in Definition \ref{def:HKGcovers}. 
In particular, $1\le \#Z_{\mathrm{br}} \le 2$. 
Writing $\tau=\tau_2\circ \tau_1$ for an $H$-cover $\tau_1:X\to Y$ and a $C$-cover $\tau_2:Y\to Z$, it follows by Remark \ref{rem:Zbr} that $g(Y) - 1 = -3 + \#Z_{\mathrm{br}}$. Since $g(Y)\ge 0$, this implies $\#Z_{\mathrm{br}} = 2$ and $g(Y)=0$. Since $\tau$ is a Harbater-Katz-Gabber $G$-cover, it follows that $\tau$ has precisely one totally ramified point $x\in X$ and the action of $G$ on $X-\{x\}$ is tamely and non-trivially ramified at one $G$-orbit in $X-\{x\}$ and unramified everywhere else. But this implies that $\tau_1:X\to Y$ is totally ramified with $\#Y_{\mathrm{br}}=1$, which proves Lemma \ref{lem:HKGcovers}.
\end{proof}

\begin{nota}
\label{not:specialbranchY}
Under Assumption \ref{ass:general} and Notations \ref{not:general} and \ref{not:functionfields}, define 
$$Y_{\infty,0}:=\left\{P\in Y_{\mathrm{br}}\,:\, P\in\{\infty,0\}\right\}\quad\mbox{and}\quad r:=\# Y_{\infty,0}$$
so that $r\in\{0,1,2\}$ by Lemma $\ref{lem:totram}$. For $y\in Y_{\infty,0}$, define
$$\epsilon_y:=\left\{\begin{array}{rl}
-1 & \mbox{if } y=\infty,\\
1 & \mbox{if } y=0.
\end{array}\right.$$
Moreover, if $r> 0$, we replace $s$ by $s^{-1}$ if necessary to be able to assume, without loss of generality, that $\infty\in Y_{\infty,0}$.
\end{nota}

Since we have to deal with both the $H$-cover $\pi_1:X\to Y$ and the $C$-cover $\pi_2: Y\to Z$, we will not apply a fractional linear transformation to $K=k(Y)=k(s)$ to ensure that $\infty$ is a branch point of $\pi_1$ and $0$ is not. Instead, we will keep $Y_{\infty,0}$ as in Notation \ref{not:specialbranchY} for the remainder of this section. Because of this, we now discuss a slight variation of the results from \cite[Section 3]{BleherCamacho2023}, and in particular of the $k$-basis for $\HH^0(X,\Omega_X)$ as given in \cite[Definition 3.5 and Lemma 3.6]{BleherCamacho2023}.

Similarly to (\ref{eq:Ybr}), we have
$$Y_{\mathrm{br}} = Y_{\infty,0}\cup \{y_1,\ldots,y_n\}$$
for some $n\ge 0$, where 
each $y\in Y_{\infty,0}$ corresponds to $s^{\epsilon_y}$ for $\epsilon_y$ as in Notation \ref{not:specialbranchY},
and each $y_i$ corresponds to $s-\mu_i$ for certain $\mu_i\in k^\times$. Similarly to  (\ref{eq:standardform}), we can use the partial fractions decomposition of $\alpha$ to show that there exists $h_\alpha\in K = k(s)$ such that
\begin{equation}
\label{eq:newalpha!}
\alpha-(h_\alpha^2-h_\alpha)=s^{p_\infty+\left(\sum_{i=1}^np_i\right)-d}\;\frac{\left(\prod_{i=1}^n \mu_i^{p_i}\right)\left(c_ds^d+\dots+c_1s+c_0\right)}{(s-\mu_1)^{p_1}\dots(s-\mu_n)^{p_n}}
\end{equation}
satisfying the following five conditions:
\begin{itemize}
\item $\infty \in Y_{\infty,0}$ if and only if $p_\infty>0$;
\item $0\in Y_{\infty,0}$ if and only if $p_0:=d-(p_\infty + p_1 + \cdots + p_n)>0$;
\item if $y \in Y_{\infty,0}$ then $p_y \in\{m_y,M_y\}$; 
\item if $n\ge 1$ then $p_i\in\{m_{y_i},M_{y_i}\}$ for $1\le i \le n$;
\item $c_d,c_0\in k^{\times}$, and $c_ds^d+\cdots+c_1s+c_0$ and $(s-\mu_1)^{p_1}\dots(s-\mu_n)^{p_n}$ are relatively prime in $k[s]$.
\end{itemize}

\begin{lemma} 
\label{lem:branchpoints}
Suppose Assumption $\ref{ass:general}$ holds, and use Notations $\ref{not:general}$, $\ref{not:functionfields}$ and $\ref{not:specialbranchY}$.
\begin{itemize}
\item[(i)] We have $\#Y_{\mathrm{br}} = r+n=r+3 \ell$ for some $\ell \ge 0$.
\item[(ii)] There exist $\psi_1,\ldots,\psi_\ell\in k^\times$ such that 
$$\{\mu_1,\ldots,\mu_n\} = \{\zeta^i\psi_j\;:\; 1\le j\le \ell, 0\le i \le 2\}.$$
\item[(iii)] The element $h_\alpha \in k(s)$ in $(\ref{eq:newalpha!})$ can be chosen to satisfy $\mathrm{Tr}_{K/J}(h_\alpha) = 0 = \mathrm{Tr}_{K/J}(h_\alpha^2)$, and hence $\mathrm{Tr}_{K/J}(\alpha-(h_\alpha^2-h_\alpha))=0$.
\item[(iv)] If $\infty\in Y_{\infty,0}$ then $p_\infty \equiv \pm 1\mod 3$, and if $0\in Y_{\infty,0}$ then $p_0 \equiv \pm 1\mod 3$.
\end{itemize}
\end{lemma}

\begin{proof}
Since $k(s^{-1})=k(s)$, we can write $\alpha$ as a rational function over $k$ in the variable $s^{-1}$. Then the condition $\mathrm{Tr}_{K/J}(\alpha)=0$ means that there exist $j_1(s^{-3}),j_2(s^{-3})\in J=k(s^{-3})$ such that
\begin{equation}
\label{eq:alpha1!}
\alpha = s^{-1}\cdot j_1(s^{-3}) + s^{-2}\cdot j_2(s^{-3}).
\end{equation}
Because of (\ref{eq:newalpha!}) and since $k$ is algebraically closed, it follows that there exist polynomials $f_1(T),f_2(T)\in k[T]$ such that
\begin{equation}
\label{eq:alpha2!}
\alpha = \frac{s^{-1}\cdot f_1(s^{-3}) + s^{-2}\cdot f_2(s^{-3})}{(s^{-3})^{e_\infty} (s^{-3}-\psi^{-3}_1)^{e_1}\cdots (s^{-3}-\psi^{-3}_\ell)^{e_\ell}}
\end{equation}
for some $\ell\ge 0$, satisfying the following five conditions:
\begin{itemize}
\item[(a)] $\psi_1,\dots,\psi_\ell \in k^\times$ are distinct, and $\{\mu_1,\dots,\mu_n\} = \{\zeta^i\psi_j\;:\; 1\le j\le \ell, 0\le i \le 2\}$;
\item[(b)] $f_1(0)\ne 0$ or $f_2(0)\ne 0$, and $f_1(\psi^{-3}_j) + \zeta^i\psi^{-1}_jf_2(\psi^{-3}_j)\ne 0$ for all $1\le j\le \ell$, $0\le i\le 2$;
\item[(c)] $e_1,\dots,e_\ell >0$;
\item[(d)] $e_\infty >0$ if and only if $\infty \in Y_{\infty,0}$;
\item[(e)] $e_0:=-3 (e_\infty+e_1+\cdots+e_\ell)+1+\mathrm{max}\{3\,\mathrm{deg}(f_1),1 + 3\,\mathrm{deg}(f_2)\}> 0$ if and only if $0 \in Y_{\infty,0}$.
\end{itemize}
We note that $e_0\equiv \pm 1\mod 3$, which means that $e_0\ne 0$. Moreover, part (e) follows from
\begin{eqnarray*}
\mathrm{ord}_{s}\alpha &=& 3 (e_\infty+e_1+\cdots+e_\ell)-1 + \mathrm{min}\{\mathrm{ord}_{s}f_1(s^{-3}), -1 + \mathrm{ord}_{s}f_2(s^{-3})\}\\
&=&3 (e_\infty+e_1+\cdots+e_\ell)-1-\mathrm{max}\{3\,\mathrm{deg}(f_1),1 + 3\,\mathrm{deg}(f_2)\}.
\end{eqnarray*}
In particular, this implies parts (i) and (ii) of Lemma \ref{lem:branchpoints}.
From now on, we write 
$$\{\mu_1,\ldots,\mu_n\} = \{\zeta^i\psi_j\;:\; 1\le j\le \ell, 0\le i \le 2\}.$$
Using condition (b) above, we define
$$b:=\begin{cases}
1 & \mbox{ if } f_1(0)\neq 0,\\
2 & \mbox{ otherwise}.
\end{cases}$$

To prove parts (iii) and (iv) of Lemma \ref{lem:branchpoints}, we look at the partial fraction decomposition of (\ref{eq:alpha2!}) in the variable $s^{-1}$ over $k$. This partial fraction decomposition looks like
\begin{equation} 
\label{eq:alpha3!}
\alpha = \sum_{e=0}^{e_0} a_{0,e}(s^{-1})^e + \sum_{e=1}^{{3e_\infty-b}} \frac{a_{\infty,e}}{(s^{-1})^e} + \sum_{j=1}^{\ell} \sum_{i=0}^2 \sum_{e=1}^{3e_j} \frac{a_{j,i,e}}{(\zeta^{-i}s^{-1}-\psi^{-1}_j)^e}
\end{equation}
for appropriate $a_{0,e} (0\le e \le e_0),~a_{\infty,e} (1\le e \le 3e_\infty-b), ~a_{j,i,e} (1\le j\le \ell, 0\le i\le 2, 1\le e \le 3e_j)\in k$. We first calculate the traces from $K$ to $J$ of the summands on the right side. Fix $e\ge 0$. We have
\begin{equation}
\label{eq:trace0!}
\mathrm{Tr}_{K/J}\left((s^{-1})^e\right) = \sum_{c=0}^2 (\zeta^{-c}s^{-1})^e
= \left(1+\zeta^e+ \zeta^{2e}\right)(s^{-1})^e = 
\begin{cases}
\displaystyle(s^{-1})^e & \mbox{if }e\equiv 0 \mod 3, \\
\;\;\;0& \mbox{otherwise, and}
\end{cases}
\end{equation}
\begin{equation}
\label{eq:traceinfinity!}
\mathrm{Tr}_{K/J}\left(\frac{1}{(s^{-1})^e}\right) = \sum_{c=0}^2 \frac{1}{(\zeta^{-c}s^{-1})^e} = \frac{1+\zeta^e+ \zeta^{2e}}{(s^{-1})^e} = 
\begin{cases}
\displaystyle\frac{1}{(s^{-1})^e} & \mbox{if }e\equiv 0 \mod 3, \\
\;\;\;0& \mbox{otherwise}.
\end{cases}
\end{equation}
On the other hand, for all $1\le j\le \ell$ and all $0\le i\le 2$, we have
\begin{eqnarray}
\nonumber
\mathrm{Tr}_{K/J}\left(\frac{1}{(\zeta^{-i}s^{-1}-\psi^{-1}_j)^e}\right) &=&\sum_{c=0}^2 \frac{1}{(\zeta^{-c}s^{-1}-\psi^{-1}_j)^e}\\
\label{eq:need1!}
&=& \frac{1}{(s^{-3}-\psi^{-3}_j)^e} \sum_{c=0}^2\left(\zeta^{-2c}s^{-2} +\zeta^{-c}\psi_j^{-1}s^{-1} + \psi_j^{-2}\right)^e.
\end{eqnarray}
We note that
\begin{equation}
\label{eq:need2!}
\sum_{c=0}^2 \left(\zeta^{-2c}(\psi_j)^{-2} +\zeta^{-c}\psi_j^{-1}(\psi_j)^{-1} + \psi_j^{-2}\right)^e = \psi_j^{-2e} \ne 0.
\end{equation}
By the multinomial theorem, we have
$$\mathrm{Tr}_{K/J}\left(\frac{1}{(\zeta^{-i}s^{-1}-\psi^{-1}_j)^e}\right)
= \frac{1}{(s^{-3}-\psi^{-3}_j)^e} \sum_{{ \genfrac {} {}{0pt}{2}{0\le d_1,d_2,d_3\le e}{d_1+d_2+d_3=e}}} \frac{e!}{d_1! d_2! d_3!}
E_j(d_1,d_2,d_3)$$
where
\begin{eqnarray*}
E_j(d_1,d_2,d_3) &=& \sum_{c=0}^2\left( \zeta^{-2cd_1}s^{-2d_1} \zeta^{-cd_2}\psi_j^{-d_2}s^{-d_2}\psi_j^{-2d_3}\right)\\
&=&\begin{cases}
\psi_j^{-(d_2+2d_3)} s^{-(2d_1+d_2)} & \mbox{if }2d_1+d_2\equiv 0 \mod 3, \\
\;\;\;0& \mbox{otherwise}.
\end{cases}
\end{eqnarray*}
Using (\ref{eq:need1!}) and (\ref{eq:need2!}), this implies that, for all $1\le j\le \ell$, there exists a polynomial $f_{j,e}(T)\in k[T]$ with $f_{j,e}(\mu^{-3}_j)\ne 0$ such that, for all $0\le i\le 2$, we have
\begin{equation}
\label{eq:traceformula!}
\mathrm{Tr}_{K/J}\left(\frac{1}{(\zeta^{-i}s^{-1}-\psi^{-1}_j)^e}\right) = \frac{f_{j,e}(s^{-3})}{(s^{-3}-\psi^{-3}_j)^e}.
\end{equation}
Since $\mathrm{Tr}_{K/J}(\alpha)=0$, (\ref{eq:alpha3!}) and using that $\left\lfloor \frac{3e_\infty-b}{3}\right\rfloor=e_\infty-1$, (\ref{eq:trace0!}), (\ref{eq:traceinfinity!}) and (\ref{eq:traceformula!}) lead to the equation
\begin{equation}
\label{eq:need3!}
\sum_{e=0}^{\left\lfloor \frac{e_0}{3}\right\rfloor} a_{0,3e} (s^{-1})^{3e} + 
\sum_{e=1}^{e_\infty-1} \frac{a_{\infty,3e}}{(s^{-1})^{3e}} + 
\sum_{j=1}^{\ell} \sum_{e=1}^{3e_j} (a_{j,0,e}+a_{j,1,e}+a_{j,2,e})\frac{f_{j,e}(s^{-3})}{(s^{-3}-\psi^{-3}_j)^e}=0
\end{equation}
in $k(s^{-3})$. But this means that  (\ref{eq:need3!}) must remain true after localizing at various points. If $e_0 \le 0$ then $e_0<0$ and the first sum is zero. If $e_0 > 0$ then localizing at $s=0$, we obtain
$$\mathrm{ord}_{s} \left( \mathrm{Tr}_{K/J}(\alpha)\right) = -3 \left\lfloor \frac{e_0}{3}\right\rfloor$$
unless $a_{0,3\lfloor \frac{e_0}{3}\rfloor}=0$. Using descending induction, we conclude
$$a_{0,3e}=0 \quad \mbox{for all $e\in\{0,\ldots,\lfloor \frac{e_0}{3}\rfloor\}$.}$$
If $e_\infty \le 1$ then the second sum is zero. If $e_\infty > 1$, then, localizing at $s=\infty$, 
we similarly
conclude
$$a_{\infty,3e}=0 \quad \mbox{for all $e\in\{1,\ldots,e_\infty-1\}$.}$$
Fix $j\in\{1,\ldots,\ell\}$. Localizing at $s=\psi_j$, we 
additionally
conclude 
$$a_{j,0,e}+a_{j,1,e}+a_{j,2,e}=0 \quad \mbox{for all $j\in\{1,\ldots,\ell\}$ and all $e\in\{1,\ldots,3e_j\}$.}$$
Therefore, $\alpha$ in  (\ref{eq:alpha3!}) becomes
\begin{eqnarray}
\label{eq:alpha4!}
\alpha &=& \sum_{{ \genfrac {} {}{0pt}{2}{e\in\{1,\ldots, e_0\}}{e\not\equiv 0\mod 3}}} a_{0,e}(s^{-1})^e  \quad + \sum_{{ \genfrac {} {}{0pt}{2}{e\in\{1,\ldots, 3e_\infty-b\}}{e\not\equiv 0\mod 3}}} \frac{a_{\infty,e}}{(s^{-1})^e} \quad + \\
\nonumber
&&\sum_{j=1}^{\ell}  \sum_{e=1}^{3e_j} \left(\frac{a_{j,0,e}}{(s^{-1}-\psi^{-1}_j)^e} + \frac{a_{j,1,e}}{(\zeta^{-1}s^{-1}-\psi^{-1}_j)^e} - \frac{a_{j,0,e}+a_{j,1,e}}{(\zeta^{-2}s^{-1}-\psi^{-1}_j)^e} \right).
\end{eqnarray}
For each integer $e\ge 1$, write $e = 2^{e'}e''$ where $e'\ge 0$ and $e''$ is odd. Since $k$ is algebraically closed of characteristic $2$, for each $a\in \{a_{0,e},a_{\infty,e}\} \cup \{a_{j,0,e},a_{j,1,e}:1\le j\le \ell\}$, there exists $a'\in k$ such that $a=(a')^{2^{e'}}$. 
Define 
\begin{eqnarray*}
h_{\alpha,0,e} &:=& 
\sum_{i=1}^{e'} (a_{0,e}')^{2^{e'-i}}(s^{-1})^{2^{e'-i}e''},\\
h_{\alpha,\infty,e} &:=& 
\sum_{i=1}^{e'} \frac{(a_{\infty,e}')^{2^{e'-i}}}{(s^{-1})^{2^{e'-i}e''}},
\end{eqnarray*}
and, for $j\in\{1,\ldots,\ell\}$, define
$$h_{\alpha,j,e} := 
\sum_{i=1}^{e'} \left(\frac{(a_{j,0,e}')^{2^{e'-i}}}{(s^{-1}-\psi_j^{-1})^{2^{e'-i}e''}}+
\frac{(a_{j,1,e}')^{2^{e'-i}}}{(\zeta^{-1}s^{-1}-\psi_j^{-1})^{2^{e'-i}e''}}-
\frac{(a_{j,0,e}')^{2^{e'-i}}+ (a_{j,1,e}')^{2^{e'-i}}}{(\zeta^{-2}s^{-1}-\psi_j^{-1})^{2^{e'-i}e''}} \right).$$
Note that $h_{\alpha,0,e}=0$ if $e$ is odd or if $a_{0,e}=0$, and that $h_{\alpha,\infty,e}=0$ if $e$ is odd or if $a_{\infty,e}=0$, and that, for $j\in\{1,\ldots,\ell\}$, $h_{\alpha,j,e}=0$ if $e$ is odd or if $a_{j,0,e}=a_{j,1,e}=0$. Define
$$h_\alpha := \sum_{{ \genfrac {} {}{0pt}{2}{e\in\{1,\ldots, e_0\}}{e\not\equiv 0\mod 3}}} h_{\alpha,0,e}\quad +
\sum_{{ \genfrac {} {}{0pt}{2}{e\in\{1,\ldots, 3e_\infty-b\}}{e\not\equiv 0\mod 3}}} h_{\alpha,\infty,e} + \quad
\sum_{j=1}^\ell \sum_{e=1}^{3e_j}h_{\alpha,j,e}.$$
By our earlier calculations, leading to  (\ref{eq:trace0!}), (\ref{eq:traceinfinity!}), (\ref{eq:traceformula!}) and (\ref{eq:alpha4!}), and since the characteristic of $k$ is $2$, this implies that $\mathrm{Tr}_{K/J}(h_\alpha) =\mathrm{Tr}_{K/J}(h_\alpha^2) = 0$. We obtain
\begin{eqnarray*}
\alpha-(h_\alpha^2-h_\alpha) &=&
\sum_{{ \genfrac {} {}{0pt}{2}{e\in\{1,\ldots, e_0\}}{e\not\equiv 0\mod 3}}} a_{0,e}'(s^{-1})^{e''}+
\sum_{ { \genfrac {} {}{0pt}{2}{e\in\{1,\ldots, 3e_\infty-b\}}{e\not\equiv 0\mod 3}} }\frac{a_{\infty,e}'}{(s^{-1})^{e''}} +\\
&&
\sum_{j=1}^{\ell}  \sum_{e=1}^{3e_j} \left(\frac{a_{j,0,e}'}{(s^{-1}-\psi^{-1}_j)^{e''}} + \frac{a_{j,1,e}'}{(\zeta^{-1}s^{-1}-\psi^{-1}_j)^{e''}} - \frac{a_{j,0,e}'+a_{j,1,e}'}{(\zeta^{-2}s^{-1}-\psi^{-1}_j)^{e''}} \right).
\end{eqnarray*}
Since $e''$ is odd for all positive integers $e$, we can multiply both the numerator and the denominator of this expression with the same odd power of $s$ to obtain a right side as in  (\ref{eq:newalpha!}). Since $e = 2^{e'}e''$, where $e'\ge 0$ and $e''$ is odd, we obtain that $e\not\equiv 0\mod 3$ if and only if $e''\not\equiv 0\mod 3$. Therefore, if $\infty\in Y_{\infty,0}$ then $p_\infty \equiv \pm 1\mod 3$, and if $0\in Y_{\infty,0}$ then $p_0 \equiv \pm 1\mod 3$. This completes the proof of Lemma \ref{lem:branchpoints}.
\end{proof}

Based on Lemma \ref{lem:branchpoints}, we introduce the following notation.

\begin{nota} 
\label{not:branchA4}
Suppose Assumption \ref{ass:general} holds, and use Notations \ref{not:general}, \ref{not:functionfields} and \ref{not:specialbranchY}. 

\begin{itemize}
\item[(a)] We have
$$Y_{\mathrm{br}}=Y_{\infty,0}\cup\{y_1,y_1',y_1'',y_2,y_2',y_2'',\dots,y_\ell,y_\ell',y_\ell''\}$$
for some $\ell\ge 0$, where 
each $y\in Y_{\infty,0}$ corresponds to $s^{\epsilon_y}$ for $\epsilon_y$ as in Notation \ref{not:specialbranchY},
and, for $1\le j \le \ell$, there exists $\psi_j\in k^\times$ such that $y_j$ corresponds to $s-\psi_j$, $y_j'$ corresponds to $s-\zeta\psi_j$ and $y_j''$ corresponds to $s-\zeta^2\psi_j$. 
\item[(b)]
There exists $h_\alpha\in K=k(s)$ such that $\mathrm{Tr}_{K/J}(h_\alpha)=\mathrm{Tr}_{K/J}(h_\alpha^2)=0$ and
$$\alpha-(h_\alpha^2-h_\alpha)=s^{p_\infty+\left(\sum_{j=1}^{\ell}(p_j+p_j'+p_j'')\right)-d}\;
\frac{\zeta^{\sum_{j=1}^\ell(p_j'+2p_j'')}\prod_{j=1}^\ell\mu_j^{p_j+p_j'+p_j''}\left(\sum_{i=0}^d c_is^i\right)}{\prod_{j=1}^\ell (s-\psi_j)^{p_j}(s-\zeta\psi_j)^{p_j'}(s-\zeta^2\psi_j)^{p_j''}}$$
satisfying the following conditions: 
\begin{itemize}
\item[(b1)] $\infty \in Y_{\infty,0}$ if and only if $p_\infty>0$;
\item[(b2)] $0\in Y_{\infty,0}$ if and only if $p_0:=d-(p_\infty + \sum_{j=1}^\ell(p_j+p_j'+p_j''))>0$;
\item[(b3)] if $y \in Y_{\infty,0}$ then $p_y \in\{m_y,M_y\}$; 
\item[(b4)] if $\ell\ge 1$ then $p_j\in\{m_{y_j},M_{y_j}\}$, $p_j'\in \{m_{y_j'},M_{y_j'}\}$, $p_j''\in \{m_{y_j''},M_{y_j''}\}$ for $1\le j \le \ell$;
\item[(b5)] $c_d,c_0\in k^{\times}$, and $\sum_{i=0}^d c_is^i$ and $\prod_{j=1}^\ell (s-\psi_j)^{p_j}(s-\zeta\psi_j)^{p_j'}(s-\zeta^2\psi_j)^{p_j''}$ are relatively prime in $k[s]$.
\end{itemize}
\item[(c)] We define
$$h_{\rho \alpha} := \rho h_\alpha\quad\mbox{and}\quad h_{\rho^2 \alpha} := \rho^2 h_\alpha.$$
Applying Remark \ref{rem:important} and using that $v=\rho u$, $\beta=\rho \alpha$ and $\alpha+\beta=\rho^2\alpha$ by Notation \ref{not:functionfields}, we obtain
$$\left\{\begin{array}{rclrcl}
\widetilde{u}&=&u-h_\alpha, \quad & \widetilde{\alpha}&=&\alpha-(h_\alpha^2-h_\alpha),\\
\widetilde{\rho u}&=&\rho u -h_{\rho\alpha}, \quad & \widetilde{\rho\alpha}&=&\rho\alpha-(h_{\rho\alpha}^2-h_{\rho\alpha}),\\
\widetilde{u+\rho u}&=&u+\rho u-h_{\rho^2\alpha}, \quad & \widetilde{\rho^2\alpha}&=&\rho^2\alpha-(h_{\rho^2\alpha}^2-h_{\rho^2\alpha}).
\end{array}\right.$$
Moreover, for $y\in Y_{\mathrm{br}}$, we obtain
$$p_{\alpha,y}=-\operatorname{ord}_{y}(\widetilde{\alpha}), \quad p_{\rho\alpha,y}=-\operatorname{ord}_{y}(\widetilde{\rho\alpha})\quad \mbox{and}\quad p_{\rho^2\alpha,y}=-\operatorname{ord}_{y}(\widetilde{\rho^2\alpha}).$$
\end{itemize}
\end{nota}

The following result provides the values introduced in Remark \ref{rem:BC2023} and Notation \ref{not:deltalambda}.

\begin{lemma}
\label{lem:valuesbranchpoints}
Suppose Assumption $\ref{ass:general}$ holds, and use Notations $\ref{not:general}$, $\ref{not:functionfields}$, $\ref{not:specialbranchY}$ and $\ref{not:branchA4}$.
Fix $y\in Y_{\mathrm{br}}$. The values of $m_y,M_y$ from Remark $\ref{rem:BC2023}$ and of $\lambda_y, \delta_y$ from Notation $\ref{not:deltalambda}$ are as follows:
\begin{itemize}
\item[(a)] If $y\in Y_{\infty,0}$, then $m_y=M_y=p_y$, $\lambda_y=\zeta^{\epsilon_yp_y}\in\{\zeta,\zeta^2\}$, and $\delta_y\in \{0,1,\cdots,\lfloor \frac{m_y}{4}\rfloor\}$. Moreover,
$$u_y=\widetilde{u}, \alpha_y = \widetilde{\alpha},\quad v_y=\widetilde{\rho u}, \beta_y = \widetilde{\rho\alpha}.$$
\item[(b)] Let $j\in\{1,\ldots,\ell\}$, define $\psi:=\psi_j$, and define $y:=y_j$, $y':=y_j'$ and $y'':=y_j''$, so that the uniformizers are $\pi_y=(s-\psi)$, $\pi_{y'}=(s-\zeta\psi)$ and $\pi_{y''}=(s-\zeta^2\psi)$.
\begin{itemize}
\item[(i)] If $m_y=M_y$ then $m_y=m_{y'}=m_{y''}=M_y=M_{y'}=M_{y''}=p_j=p_j'=p_j''$, $$\lambda_y\in k-\{0,1\},\quad \lambda_{y'}=\frac{1+\lambda_y}{\lambda_y}\quad\mbox{and}\quad\lambda_{y''} = \frac{1}{1+\lambda_y},$$
and $\delta_y=\delta_{y'}=\delta_{y''}\in \{0,1,\cdots,\lfloor \frac{m_y}{4}\rfloor\}$.
Moreover,
$$u_y=u_{y'}=u_{y''}=\widetilde{u}, \alpha_y=\alpha_{y'}=\alpha_{y''} = \widetilde{\alpha},\quad v_y=v_{y'}=v_{y''}=\widetilde{\rho u}, \beta_y=\beta_{y'}=\beta_{y''} = \widetilde{\rho\alpha}.$$
\item[(ii)] If $m_y<M_y$, then $m_y=m_{y'}=m_{y''}<M_y=M_{y'}=M_{y''}$,
$\{\lambda_y,\lambda_{y'},\lambda_{y''}\}=\{0,1,\infty\}$,
and $\delta_y=\delta_{y'}=\delta_{y''}=-1$. 
Moreover:
\begin{itemize}
\item[(1)] If $m_y=p_{\alpha,y}$, then $m_y=p_j$, $M_y=p_j'=p_j''$, and
$$\begin{array}{ll}
u_y=\widetilde{u}, \alpha_y = \widetilde{\alpha}, & v_y=\widetilde{\rho u}, \beta_y = \widetilde{\rho\alpha},\\
u_{y'}= \widetilde{u+\rho u}, \alpha_{y'}=\widetilde{\rho^2\alpha}, &v_{y'}=\widetilde{\rho u}, \beta_{y'} = \widetilde{\rho\alpha},\\
u_{y''}=\widetilde{\rho u}, \alpha_{y''} = \widetilde{\rho\alpha}, & v_{y''}= \widetilde{u}, \beta_{y''}=\widetilde{\alpha}.
\end{array}$$
\item[(2)] If $m_y=p_{\rho\alpha,y}$, then $m_y=p_j'$, $M_y=p_j=p_j''$, and
$$\begin{array}{ll}
u_y=\widetilde{\rho u}, \alpha_y = \widetilde{\rho\alpha}, & v_y= \widetilde{u}, \beta_y=\widetilde{\alpha},\\
u_{y'}=\widetilde{u}, \alpha_{y'} = \widetilde{\alpha}, & v_{y'}=\widetilde{\rho u}, \beta_{y'} = \widetilde{\rho\alpha},\\
u_{y''}= \widetilde{u+\rho u}, \alpha_{y''}=\widetilde{\rho^2\alpha}, &v_{y''}=\widetilde{\rho u}, \beta_{y''} = \widetilde{\rho\alpha}.
\end{array}$$
\item[(3)] If $m_y=p_{\rho^2\alpha,y}$, then $m_y=p_j''$, $M_y=p_j=p_j'$,  and
$$\begin{array}{ll}
 u_y= \widetilde{u+\rho u}, \alpha_y=\widetilde{\rho^2\alpha}, &v_y=\widetilde{\rho u}, \beta_y = \widetilde{\rho\alpha},\\
u_{y'}=\widetilde{\rho u}, \alpha_{y'} = \widetilde{\rho\alpha}, & v_{y'}= \widetilde{u}, \beta_{y'}=\widetilde{\alpha},\\
u_{y''}=\widetilde{u}, \alpha_{y''} = \widetilde{\alpha}, & v_{y''}=\widetilde{\rho u}, \beta_{y''} = \widetilde{\rho\alpha}.
\end{array}$$
\end{itemize}
\end{itemize}
\end{itemize}
\end{lemma}

\begin{proof}
We first prove part (a). Suppose $y\in Y_{\infty,0}$. Since $\infty$ corresponds to $s^{-1}$ and $0$ corresponds to $s$, we have $p_y=p_{\alpha,y}=p_{\rho\alpha,y}=p_{\rho^2\alpha,y}$.
Therefore, we obtain from Remark \ref{rem:important} and Notation \ref{not:deltalambda} that $m_y=M_y=p_y$, and $\delta_y\in \{0,1,\cdots,\lfloor \frac{m_y}{4}\rfloor\}$. In particular, we obtain from Remark \ref{rem:important}, using Notation \ref{not:branchA4}, that
\begin{equation}
\label{eq:blimy}
u_y=\widetilde{u}, \alpha_y = \widetilde{\alpha},\quad v_y=\widetilde{\rho u}, \beta_y = \widetilde{\rho\alpha}.
\end{equation}
It remains to show $\lambda_y=\zeta^{\epsilon_yp_y}\in\{\zeta,\zeta^2\}$.
By Remark \ref{rem:doweneedthis?},
$$\beta_{y,0} + \alpha_{y,0}\theta_{y,0}^2 = 0,$$
where, by Notation \ref{not:deltalambda}, $\lambda_y=\theta_{y,0}$. Using Notation \ref{not:branchA4}, parts (b) and (c), and (\ref{eq:blimy}), we obtain
$$\lambda_y^2=\theta_{y,0}^2=\zeta^{-\epsilon_yp_y}$$
for $\epsilon_y$ as in Notation \ref{not:specialbranchY}, which means $\lambda_y=\zeta^{\epsilon_yp_y}$.
By Lemma \ref{lem:branchpoints}(iv), we get $\lambda_y\in\{\zeta,\zeta^2\}$, which finishes the proof of part (a).

For the proof of part (b), we use that by Notation \ref{not:branchA4},
\begin{equation}
\label{eq:needed?}
\left\{ \begin{array}{l@{\,=\;}rl@{\,=\;}rl@{\,=\;}r}
p_{\alpha,y} & p_j, & p_{\rho\alpha,y} & p_j', & p_{\rho^2\alpha,y} & p_j'',\\
p_{\alpha,y'} & p_j', & p_{\rho\alpha,y'} & p_j'', & p_{\rho^2\alpha,y'} & p_j,\\
p_{\alpha,y''} & p_j'', & p_{\rho\alpha,y''} & p_j, & p_{\rho^2\alpha,y''} & p_j'.
\end{array}\right.
\end{equation}

We first consider part (b)(i), i.e. $m_y=M_y$. By (\ref{eq:needed?}), this implies 
$$m_y=m_{y'}=m_{y''}=M_y=M_{y'}=M_{y''}=p_j=p_j'=p_j''.$$
Hence we obtain from Remark \ref{rem:important} and Notation \ref{not:branchA4}(c) that
$$u_y=u_{y'}=u_{y''}=\widetilde{u}, \alpha_y=\alpha_{y'}=\alpha_{y''} = \widetilde{\alpha},\quad v_y=v_{y'}=v_{y''}=\widetilde{\rho u}, \beta_y=\beta_{y'}=\beta_{y''} = \widetilde{\rho\alpha}.$$
In particular, it follows from Notation \ref{not:deltalambda} that $\lambda_y\in k-\{0,1\}$ and that $\delta_y\in \{0,1,\cdots,\lfloor \frac{m_y}{4}\rfloor\}$. To show the remaining statements in part (b)(i), we write
\begin{eqnarray}
\label{eq:oyy1}
\widetilde{\alpha} (s-\psi)^{m_y}&=& \sum_{i\ge 0} \alpha_i (s-\psi)^i,\\
\label{eq:oyy2}
\widetilde{\rho\alpha} (s-\psi)^{m_y}&=& \sum_{i\ge 0} \beta_i (s-\psi)^i,
\end{eqnarray}
which gives us expressions for $\alpha_y$ and $\beta_y$ in terms of $\pi_y$.
We note that $\alpha_0\ne 0$, $\beta_0\ne 0$ and $\alpha_0+\beta_0\ne 0$ by Remark \ref{rem:doweneedthis?}.
Since $\widetilde{\alpha}+\widetilde{\rho\alpha}+\widetilde{\rho^2\alpha} = \mathrm{Tr}_{K/J}(\widetilde{\alpha})=0$ by Lemma \ref{lem:branchpoints}, we can apply $\rho$ and $\rho^2$ to (\ref{eq:oyy1}) and (\ref{eq:oyy2}) to 
get similar expressions for $\alpha_{y'},~\alpha_{y''}$ and $\beta_{y'},~\beta_{y''}$.
We thus obtain
$$\begin{array}{ll}
\alpha_y=\pi_y^{-m_y} \sum_{i\ge 0} \alpha_i \pi_y^i, &
\beta_y= \pi_y^{-m_y} \sum_{i\ge 0} \beta_i \pi_y^i,\\
\alpha_{y'}=\pi_{y'}^{-m_y} \sum_{i\ge 0} \beta_i\zeta^{-2(m_y-i)} \pi_{y'}^i, &
\beta_{y'}= \pi_{y'}^{-m_y} \sum_{i\ge 0} (\alpha_i+\beta_i)\zeta^{-2(m_y-i)} \pi_{y'}^i,\\
\alpha_{y''}=\pi_{y''}^{-m_y} \sum_{i\ge 0} (\alpha_i+\beta_i)\zeta^{-(m_y-i)} \pi_{y''}^i, &
\beta_{y''}= \pi_{y''}^{-m_y} \sum_{i\ge 0} \alpha_i\zeta^{-(m_y-i)} \pi_{y''}^i.
\end{array}$$
Using Remark \ref{rem:doweneedthis?}, we can use these expressions to prove
$$\lambda_y^2=\frac{\beta_0}{\alpha_0},\quad \lambda_{y'}^2=\frac{\alpha_0+\beta_0}{\beta_0}\quad\mbox{and}\quad
\lambda_{y''}^2=\frac{\alpha_0}{\alpha_0+\beta_0}.$$
In particular, this implies $\lambda_{y'}=\frac{1+\lambda_y}{\lambda_y}$ and $\lambda_{y''} = \frac{1}{1+\lambda_y}$.
Using again Remark \ref{rem:doweneedthis?} together with an inductive argument, we moreover obtain $\delta_y=\delta_{y'}=\delta_{y''}$, which completes the proof of part (b)(i).

Finally, we consider part (b)(ii), i.e. $m_y<M_y$. By  Remark \ref{rem:important} and (\ref{eq:needed?}), this implies 
$$m_y=m_{y'}=m_{y''}<M_y=M_{y'}=M_{y''}.$$
In other words, $\lambda_y,\lambda_{y'},\lambda_{y''}\in\{0,1,\infty\}$, which immediately implies $\delta_y=\delta_{y'}=\delta_{y''}=-1$. 
Using Remark \ref{rem:important} and Notation \ref{not:branchA4} together with (\ref{eq:needed?}), the remaining statements in part (b)(ii) follow, finishing the proof of Lemma \ref{lem:valuesbranchpoints}.
\end{proof}

Similarly to \cite[Definition 3.5 and Lemma 3.6]{BleherCamacho2023}, we now construct a $k$-basis of $\HH^0(X,\Omega_X)$ using our above notation. In particular, if $r=\#Y_{\infty,0}=1$ then our new basis is the same as in \cite{BleherCamacho2023}.

\begin{definition} 
\label{def:uniformky}
For $y\in Y$, define the uniformizer $\pi_y$ as in (\ref{eq:P1uniformizer}), and define
$$k_y:=\left\{\begin{array}{rl}
0 & \mbox{ if } y\neq\infty,\\
-2 & \mbox{ if } y=\infty.
\end{array}\right.$$
\end{definition}

The next definition is a variation of \cite[Definition 3.5]{BleherCamacho2023}.

\begin{definition} 
\label{def:3.5}
Suppose Assumption $\ref{ass:general}$ holds, and use Notations $\ref{not:general}$, $\ref{not:functionfields}$, $\ref{not:specialbranchY}$ and $\ref{not:branchA4}$. Write
$$Y_{\mathrm{br}} =Y_{\infty,0}\cup \{y_1,y_1',y_1'',\ldots,y_\ell,y_\ell',y_\ell''\}$$
as in Notation $\ref{not:branchA4}$(a), and let $r=\# Y_{\infty,0}$. 
Fix $y\in Y_{\mathrm{br}}$.
Let $m_y,~M_y$, $\lambda_y,~\delta_y$ and $u_y, ~ v_y$ be as in Lemma \ref{lem:valuesbranchpoints}.
Moreover, 
let $\theta_y(j),w_y(j)$ ($j\ge 0$) be as in Notation \ref{not:uniform}, 
and let $\pi_y,~k_y$ be as in Definition \ref{def:uniformky}.
\begin{enumerate}
\item[(a)] Define $\mu_{y,1}:=\bigg\lfloor\frac{m_y+3}{4}\bigg\rfloor,~\mu_{y,2}:=\bigg\lfloor\frac{2m_y+3}{4}\bigg\rfloor,~\mu_{y,3}:=\bigg\lfloor\frac{3m_y+3}{4}\bigg\rfloor, \mbox{ and } \nu_y:=\frac{M_y-m_y}{2}.$ Moreover, define 
\begin{eqnarray*}
a(y)&:=&\begin{cases}
0 & \mbox{if } r\ge 1\mbox{ and }y=\infty , \\
2 & \mbox{if } r = 0 \mbox{ and } y=y_1,\\
1 & \mbox{otherwise},
\end{cases}\\
b(y)&:=&\begin{cases}
0 & \mbox{if $r\ge 1$ or if $r=0$ and $y=y_1$},\\
1 & \mbox{otherwise}.
\end{cases}
\end{eqnarray*}
In other words, $b(y)=1$ if and only if $r=0$ and 
$y\in Y_{\mathrm{br}}-\{y_1\}$.
\item[(b)] Define
$$\begin{array}{rclrcl}
{f_{y,1,i_1}} & {:=} & {\pi_y^{-i_1},} & a(y)+b(y)\leq i_1 &\leq& \mu_{y,3}+\nu_y+k_y, \\
{f_{y,2,i_2}} & {:=} & {\pi_y^{-i_2}u_y,} & a(y)+b(y)\leq i_2 &\leq& \mu_{y,1}+\nu_y+k_y, \\
{f_{y,3,i_3}} & {:=} & {\pi_y^{-i_3}v_y,} & a(y)+b(y)\leq i_3 &\leq& \mu_{y,1}+k_y, \\
{f_{y,3,i_3}} & {:=} & {\pi_y^{-i_3}w_y(i_3-\mu_{y,1}-k_y-1),} \quad& \mu_{y,1}+k_y+1\leq i_3 &\leq &\mu_{y,2}+k_y.
\end{array}$$
If $r=0$, define additionally, for 
$y\in Y_{\mathrm{br}}-\{y_1\}$,
$$\begin{array}{rcl}
{f_{y,1,1}} & {:=} & {\left\{\begin{array}{rl}
\pi_{y_1}^{-1}\pi_{y_1'}^{-1} & \mbox{if } y=y_1',\\
\pi_{y_1'}^{-1}\pi_{y_1''}^{-1} & \mbox{if } y=y_1'',\\
\pi_{y_1}^{-1}\pi_{y_j}^{-1} & \mbox{if } y=y_j\mbox{ for } j\ge 2,\\
\pi_{y_1'}^{-1}\pi_{y_j'}^{-1} & \mbox{if } y=y_j'\mbox{ for } j\ge 2,\\
\pi_{y_1''}^{-1}\pi_{y_j''}^{-1} & \mbox{if } y=y_j''\mbox{ for } j\ge 2,
\end{array}\right.}\\
{f_{y,2,1}} & {:=} & {f_{y,1,1}\,u_y},\\
{f_{y,3,1}} & {:=} & {f_{y,1,1}\,v_y}.
\end{array}$$
Define 
\begin{eqnarray*}
\mathcal{B}_{y,1} &:=& \{f_{y,1,i_1}\; : \; a(y)\le i_1 \le \mu_{y,3}  + \nu_y + k_y \} ,\\
\mathcal{B}_{y,2} &:=& \{f_{y,2,i_2}\; : \; a(y)\le i_2 \le \mu_{y,1} + \nu_y + k_y \},\\
\mathcal{B}_{y,3} &:=& \{f_{y,3,i_3}\; : \; a(y)\le i_3 \le \mu_{y,2}  + k_y \}.
\end{eqnarray*}
\item[(c)] For each $y\in Y_{\mathrm{br}}$, define $\mathcal{B}_y:=\mathcal{B}_{y,1}\cup\mathcal{B}_{y,2}\cup\mathcal{B}_{y,3}$, and 
$$\mathcal{B}:=\bigcup_{y\in Y_{\mathrm{br}}}\mathcal{B}_y.$$
\end{enumerate}
\end{definition}

\begin{remark}
\label{rem:cuteargument}
Let $Y_{\mathrm{br}} =Y_{\infty,0}\cup \{y_1,y_1',y_1'',\ldots,y_\ell,y_\ell',y_\ell''\}$ be as in Definition \ref{def:3.5} and suppose $\ell \ge 1$. 
We order $Y_{\mathrm{br}} -Y_{\infty,0}$ as follows:
$$y_1<y_1'<y_1''<y_2<y_2'<y_2''<\cdots < y_\ell < y_\ell' < y_\ell''.$$
Define
$$\pi_{y,\widetilde{y}}:= \pi_{y}^{-1} \pi_{\widetilde{y}}^{-1}\quad \mbox{for $y<\widetilde{y}$ in $Y_{\mathrm{br}}-Y_{\infty,0}$,}$$
and let $V_\ell$ be the $k$-span of all $\pi_{y,\widetilde{y}}$. We claim that $\mathrm{dim}_k V_\ell = 3\ell -1$.

To see this, consider the Weil divisor 
$$D:= -2\infty + \sum_{j=1}^{\ell} (y_j+y_j'+y_j'') $$
on $\mathbb{P}^1_k$. It follows that $V_\ell\subseteq \HH^0(\mathbb{P}^1_k,\mathcal{O}_{\mathbb{P}^1_k}(D))$. Since the genus of $\mathbb{P}^1_k$ is zero,  the Riemann-Roch theorem implies 
$$\mathrm{dim}_k \HH^0(\mathbb{P}^1_k,\mathcal{O}_{\mathbb{P}^1_k}(D)) - \mathrm{dim}_k \HH^1(\mathbb{P}^1_k,\mathcal{O}_{\mathbb{P}^1_k}(D)) = \mathrm{deg}(D) +1.$$
By Serre duality, we have
$$\HH^1(\mathbb{P}^1_k,\mathcal{O}_{\mathbb{P}^1_k}(D)) = 
\mathrm{Hom}_k(\HH^0(\mathbb{P}^1_k,\Omega_{\mathbb{P}^1_k}\otimes\mathcal{O}_{\mathbb{P}^1_k}(-D)),k)=
\mathrm{Hom}_k(\HH^0(\mathbb{P}^1_k,\mathcal{O}_{\mathbb{P}^1_k}(-2\infty-D)),k)=0.$$
Since $\mathrm{deg}(D) = 3\ell - 2$, we obtain that $\mathrm{dim}_k V_\ell \le 3\ell -1$.
To show equality, we first note that if $\ell=1$ then $\mathrm{dim}_k V_\ell = 2 = 3\ell -1$, since $\pi_{y_1,y_1'},\pi_{y_1',y_1''}$ are $k$-linearly independent. If $\ell >1$, then we realize that $V_\ell$ is the direct sum of $V_{\ell -1}$ and the subspace $V_\ell'$ spanned by all elements of the form $\pi_{y,\widetilde{y}}$, for $y<\widetilde{y}$ and $\widetilde{y}\in\{y_\ell,y_\ell',y_\ell''\}$. By induction, $\mathrm{dim}_k V_{\ell-1} = 3(\ell-1) -1$. On the other hand, since $\pi_{y_1,y_\ell},\pi_{y_1,y_\ell'},\pi_{y_1,y_\ell''}$ are $k$-linearly independent and in $V_\ell'$, it follows that $\mathrm{dim}_k V_\ell\ge 3\ell-1$, which proves the claim.

This has the following consequence in Definition \ref{def:3.5}: If $r=0$ then $\{f_{y,1,1}\,:\, y\in Y_{\mathrm{br}}-\{y_1\}\}$ is a $k$-basis of $V_\ell$. More precisely, applying $\mathrm{ord}_y$ to a $k$-linear combination of these elements, for descending $y=y_\ell'', y_\ell', y_\ell, \ldots,y_2'',y_2',y_2,y_1'',y_1'$, we see that these elements are $k$-linearly independent. 
\end{remark}

The next result is a variation of \cite[Lemma 3.6]{BleherCamacho2023}.

\begin{lemma} 
\label{lem:Hbasis}
Suppose Assumption $\ref{ass:general}$ holds, and use Notations $\ref{not:general}$, $\ref{not:functionfields}$, $\ref{not:specialbranchY}$, $\ref{not:branchA4}$ and Definitions $\ref{def:uniformky}$, $\ref{def:3.5}$.
A $k$-basis of $\HH^0(X,\Omega_X)$ is given by 
$$\{fds:f\in\mathcal{B}\}.$$ 
\end{lemma}

\begin{proof}
If $r\ge 1$, this is proved using the same arguments as in the proof of \cite[Lemma 3.6]{BleherCamacho2023}.

Suppose now that $r=0$. Then $\infty\not\in Y_{\mathrm{br}}$ and we have, for all $x_\infty\in X$ above $\infty$,
$$-e_{x_\infty/\infty}k_\infty-d_{x_\infty/\infty}=2,$$
since the ramification index $e_{x_\infty/\infty}=1$ and the different exponent $d_{x_\infty/\infty}=0$. Since, for all $y\in Y_{\mathrm{br}}$, $\mathrm{ord}_{x_\infty}(\pi_y^{-a})=a$ for all integers $a$ and since $\mathrm{ord}_{x_\infty}(u_y)$ and $\mathrm{ord}_{x_\infty}(v_y)$ are greater than or equal to 0, it follows that, for all $f\in\mathcal{B}$,
$$\mathrm{ord}_{x_\infty}(f)\ge 2 = -e_{x_\infty/\infty}k_\infty-d_{x_\infty/\infty}.$$
Suppose next that $\widetilde{y}\in Y-Y_{\mathrm{br}}$, $\widetilde{y}\ne \infty$, and let $x_{\widetilde{y}}\in X$ be above $\widetilde{y}$. Since, for all $y\in Y_{\mathrm{br}}$, $\mathrm{ord}_{x_{\widetilde{y}}}(\pi_y^{-a})= 0$ for all integers $a$ and  $\mathrm{ord}_{x_{\widetilde{y}}}(u_y)$ and $\mathrm{ord}_{x_{\widetilde{y}}}(v_y)$ are greater than or equal to 0, it follows that, for all $f\in\mathcal{B}$,
$$\mathrm{ord}_{x_{\widetilde{y}}}(f)\ge 0 = -e_{x_{\widetilde{y}}/\widetilde{y}}k_{\widetilde{y}}-d_{x_{\widetilde{y}}/\widetilde{y}}.$$
Finally, suppose $\widetilde{y}\in Y_{\mathrm{br}}$. Using similar arguments to the ones used in the proof of \cite[Lemma 3.6]{BleherCamacho2023}, it follows that
$$\mathrm{ord}_{x_{\widetilde{y}}}(f)\ge  -e_{x_{\widetilde{y}}/\widetilde{y}}k_{\widetilde{y}}-d_{x_{\widetilde{y}}/\widetilde{y}}$$
for all $f\in \mathcal{B} - \{f_{y_i,j,1}\,:\, 2\le i \le 3\ell,1\le j \le 3\}$. Suppose $2\le i \le 3\ell$ and $1\le j \le 3$. Since $\mathrm{ord}_{x_{\widetilde{y}}}(f_{y_i,1,1})\ge -4$ and since $\mathrm{ord}_{x_{\widetilde{y}}}(u_{y_i})$ and $\mathrm{ord}_{x_{\widetilde{y}}}(v_{y_i})$ are at least as big as $-2M_{\widetilde{y}}$, we conclude that, for all $2\le i \le 3\ell$ and $1\le j \le 3$, 
$$\mathrm{ord}_{x_{\widetilde{y}}}(f_{y_i,j,1}) \ge -4-2 M_{\widetilde{y}} = -3(m_{\widetilde{y}}+1)-2(M_{\widetilde{y}}-m_{\widetilde{y}})+m_{\widetilde{y}}-1 \ge -e_{x_{\widetilde{y}}/\widetilde{y}}k_{\widetilde{y}}-d_{x_{\widetilde{y}}/\widetilde{y}},$$
where the last inequality follows since $k_{\widetilde{y}}=0$ and $d_{x_{\widetilde{y}}/\widetilde{y}}=3(m_{\widetilde{y}}+1)+2(M_{\widetilde{y}}-m_{\tilde{y}})$ by \cite[Equation (3.11)]{BleherCamacho2023}.
Therefore, $f ds\in \HH^0(X,\Omega_X)$ for all $f\in\mathcal{B}$. 

Similarly to the proof of \cite[Lemma 3.6]{BleherCamacho2023}, we obtain that the elements in $\mathcal{B}$ are $k$-linearly independent, and that, for all $y\in Y_{\mathrm{br}}$,
$$\mu_{y,3}+\nu_y + \mu_{y,1}+\nu_y + \mu_{y,2} = \textstyle \frac{1}{2}\,d_{x_y/y}.$$
We obtain 
$$\#\mathcal{B}_y = \left\{ \begin{array}{cl}
\frac{1}{2}\,d_{x_y/y} - 3&\mbox{if $y=y_1$},\\
\frac{1}{2}\,d_{x_y/y}&\mbox{if $y\ne y_1$}.
\end{array}\right. $$
Since the Riemann-Hurwitz formula implies
$$g(X) = 1+4(g(Y)-1) + \frac{1}{2} \sum_{y\in Y_{\mathrm{br}}} d_{x_y/y}= -3 + \frac{1}{2} \sum_{y\in Y_{\mathrm{br}}} d_{x_y/y},$$
we obtain $\#\mathcal{B}=g(X)$, which completes the proof of Lemma \ref{lem:Hbasis}.
\end{proof}

\begin{remark}
\label{rem:really??}
We make the same assumptions and use the same notation as in Lemma \ref{lem:Hbasis}.
For each $y\in Y_{\mathrm{br}}$, define
$$\sigma_y-1:=\begin{cases}
\sigma-1, & \mbox{in the situation of Remark \ref{rem:BC2023} (i) or (ii)(a)},\\
\tau-1, & \mbox{in the situation of Remark \ref{rem:BC2023} (ii)(b)},\\
\sigma\tau-1, & \mbox{in the situation of Remark \ref{rem:BC2023} (ii)(c)},
\end{cases}$$
and 
$$\tau_y-1:=\begin{cases}
\tau-1, & \mbox{in the situation of Remark \ref{rem:BC2023} (i) or (ii)(a)},\\
\sigma -1, & \mbox{in the situation of Remark \ref{rem:BC2023} (ii)(b)},\\
\tau - 1, & \mbox{in the situation of Remark \ref{rem:BC2023} (ii)(c)}.
\end{cases}$$
Similarly to the proof of \cite[Theorem 3.7]{BleherCamacho2023}, it follows that the $H$-action on $\mathcal{B}_y$ is given by 
$$\begin{array}{|c||c|c|l}
     \cline{1-3}
     & \sigma_y-1 & \tau_y-1 & \\
     \cline{1-3}\\[-2.75ex]\cline{1-3}
    f_{y,1,i_1} & 0 & 0 & \mbox{ for } a(y)\leq i_1\leq \mu_{y,3}+\nu_y+k_y,\\
    \cline{1-3}
    f_{y,2,i_2} & 0 & f_{y,1,i_2} & \mbox{ for } a(y)\leq i_2\leq \mu_{y,1}+\nu_y+k_y,\\
    \cline{1-3}
    f_{y,3,i_3} & f_{y,1,i_3} & 0 & \mbox{ for } a(y)\leq i_3\leq \mu_{y,1}+k_y,\\
    \cline{1-3}
    f_{y,3,i_3} & f_{y,1,i_3} & \displaystyle\sum_{i=0}^{i_3-\mu_{y,1}-k_y-1}\theta_{y,i}f_{y,1,i_3+\nu_y-i} & \mbox{ for } \mu_{y,1}+k_y+1\leq i_3\leq \mu_{y,2}+k_y.\\
    \cline{1-3}
\end{array}$$
\end{remark}

\medskip

Using Definition \ref{def:3.5} and Lemma \ref{lem:Hbasis} instead of \cite[Definition 3.5 and Lemma 3.6]{BleherCamacho2023} and using Remark \ref{rem:really??}, we can now apply similar arguments as in the proof of \cite[Theorem 3.7]{BleherCamacho2023} to obtain the following result, where we use Notation \ref{not:K4indecomposables} for the indecomposable $kH$-modules.

\begin{theorem} 
\label{thm:BC2023}
Suppose Assumption $\ref{ass:general}$ holds, and use Notations $\ref{not:general}$, $\ref{not:deltalambda}$, $\ref{not:functionfields}$, $\ref{not:specialbranchY}$, $\ref{not:branchA4}$ and Definitions $\ref{def:uniformky}$, $\ref{def:3.5}$. There is an isomorphism of $kH$-modules 
$$\HH^0(X,\Omega_X)\cong\bigoplus_{y\in Y_{\mathrm{br}}}\bigg(N_{2l_y,\lambda_y}^{\oplus a_{y,1}}\oplus N_{2(l_y-1),\lambda_y}^{\oplus a_{y,2}}\bigg)\oplus M_{3,1}^{\oplus b}\oplus k^{\oplus c}$$
where 
$b=-1+\sum_{y\in Y_{\mathrm{br}}} \mu_{y,1}$, $c=\sum_{y\in Y_{\mathrm{br}}} (\mu_{y,3}-\mu_{y,2})$, and $l_y,~a_{y,1},~a_{y,2}$ are given as follows for each $y\in Y_{\mathrm{br}}:$
\begin{itemize}
\item[(i)] If $\delta_y=0$, then $l_y=1,~a_{y,1}=\mu_{y,2}-\mu_{y,1}$, and $a_{y,2}=0$.
 \item[(ii)] If $\delta_y\ge 1$, then $l_y\geq1$ and $1\leq a_{y,1}\leq\delta_y$ are uniquely determined by the equation
 $$\mu_{y,2}-\mu_{y,1}=(l_y-1)\delta_y+a_{y,1},$$
and $a_{y,2}=\delta_y-a_{y,1}$.
\item[(iii)] If $\delta_y=-1$, then $l_y\ge 1$ and $1\leq a_{y,1}\leq\frac{M_y-m_y}{2}$ are uniquely determined by the equation 
$$\mu_{y,2}-\mu_{y,1} +\frac{M_y-m_y}{2}=(l_y-1)\frac{M_y-m_y}{2}+a_{y,1},$$
and $a_{y,2}=\frac{M_y-m_y}{2}-a_{y,1}$.
\end{itemize}
\end{theorem}

Our goal is now to determine the precise $kG$-module structure of $\HH^0(X,\Omega_X)$. We use Notation \ref{not:A4indecomposables} for the indecomposable $kG$-modules. Applying Theorem \ref{thm:BC2023} together with the restrictions of all indecomposable $kG$-modules from $G$ to $H$, as provided by Lemmas \ref{lem:bandconnection} and \ref{lem:restrictA4toK4}, we see which indecomposable $kG$-modules can occur as direct summands of $\HH^0(X,\Omega_X)$. However, determining the precise multiplicities of these indecomposable $kG$-module summands is more subtle. To accomplish this, we use the following notation to organize the branch points $Y_{\mathrm{br}}$.

\begin{nota} 
\label{not:organizebranch}
Suppose Assumption \ref{ass:general} holds, and use Notations \ref{not:general}, \ref{not:functionfields} and \ref{not:branchA4}.
We order 
$$Y_{\mathrm{br}}=Y_{\infty,0}\cup\{y_j,y_j',y_j''\,:\, 1\le j\le \ell\}$$
as in Notation \ref{not:branchA4}, i.e. each $y\in Y_{\infty,0}$ corresponds to $s^{\epsilon_y}$ for $\epsilon_y$ as in Notation \ref{not:specialbranchY}, and, for $1\le j\le \ell$, there exists $\psi_j\in k^\times$  such that $y_j$ corresponds to $s-\psi_j$, $y_j'$ corresponds to $s-\zeta\psi_j$ and $y_j''$ corresponds to $s-\zeta^2\psi_j$. For $1\le j\le \ell$, we define
$$\lambda_j:=\lambda_{y_j} \quad\mbox{and}\quad
\phi_j :=  \frac{\zeta+\lambda_j}{\zeta^2+\lambda_j}.$$
Moreover, we order $\{y_1,\ldots,y_\ell\}$ such that there are  integers
$0\le \ell_1\le \ell_2\le \ell_3\le\ell$ satisfying
$$\lambda_1,\ldots,\lambda_{\ell_1} = \zeta,\;
\lambda_{\ell_1+1},\ldots,\lambda_{\ell_2} = \zeta^2,\;
\lambda_{\ell_2+1},\ldots,\lambda_{\ell_3} \in k^\times-\{1,\zeta,\zeta^2\},\mbox{ and }
\lambda_{\ell_3+1},\ldots,\lambda_{\ell} \in \{0,1,\infty\}.$$
Additionally, for $y\in Y_{\infty,0}$, we let 
$$\ast_y = \begin{cases}
0 & \mbox{if } \lambda_y=\zeta, \\
\infty &  \mbox{if } \lambda_y=\zeta^2.
\end{cases}$$
\end{nota}

We can now state our main result, which gives the precise $kG$-module structure of $\HH^0(X,\Omega_X)$, where we use Notation \ref{not:A4indecomposables} for the indecomposable $kG$-modules.

\begin{theorem}
\label{thm:main}
Suppose Assumption $\ref{ass:general}$ holds, and use Notations $\ref{not:general}$, $\ref{not:functionfields}$, $\ref{not:specialbranchY}$, $\ref{not:branchA4}$, $\ref{not:organizebranch}$ and Definitions $\ref{def:uniformky}$, $\ref{def:3.5}$,  in addition to the notation used in the statement of Theorem $\ref{thm:BC2023}$.
Then there is a $kG$-module isomorphism between $\HH^0(X,\Omega_X)$ and the direct sum 
$$\bigoplus_{y\in Y_{\infty,0}} \;\bigoplus_{i=0}^2 \left(N_{2l_y,\ast_y,i}^{\oplus a_{y,1,i}} \oplus N_{2(l_y-1),\ast_y,i}^{\oplus a_{y,2,i}}\right)\oplus$$
$$\bigoplus_{j=1}^{\ell_1}\bigoplus_{i=0}^2 \left(N_{2l_{y_j},0,i}^{\oplus a_{y_j,1}}\oplus N_{2(l_{y_j-1)},0,i}^{\oplus a_{y_j,2}}\right)\oplus
\bigoplus_{j=\ell_1+1}^{\ell_2}\bigoplus_{i=0}^2 \left(N_{2l_{y_j},\infty,i}^{\oplus a_{y_j,1}}\oplus N_{2(l_{y_j-1)},\infty,i}^{\oplus a_{y_j,2}}\right)\oplus$$
$$\bigoplus_{j=\ell_2+1}^{\ell_3} \left(B_{6l_{y_j},\phi_j^3}^{\oplus a_{y_j,1}}\oplus
B_{6(l_{y_j-1)},\phi_j^3}^{\oplus a_{y_j,2}}\right)\oplus
\bigoplus_{j=\ell_3+1}^{\ell} \left(B_{6l_{y_j},1}^{\oplus a_{y_j,1}}\oplus
B_{6(l_{y_j-1)},1}^{\oplus a_{y_j,2}}\right)\oplus$$
$$\bigoplus_{i=0}^2 M_{3,1,i}^{\oplus b_i} \oplus \bigoplus_{i=0}^2 S_i^{\oplus c_i}$$
where $\ell_1$, $\ell_2$, $\ell_3$, and $\phi_j$ $($for $1\le j \le \ell$$)$ are as in Notation $\ref{not:organizebranch}$, and $l_y$ $($for $y\in Y_{\infty,0}$$)$, and $l_{y_j}$, $a_{y_j,1}$, $a_{y_j,2}$ $($for $1\le j\le \ell$$)$ are as in the statement of Theorem $\ref{thm:BC2023}$, and the remaining parameters are as follows, for $i\in \{0,1,2\}$:
\begin{itemize}
\item[(a)] 
For $y\in Y_{\infty,0}$, write
$\mu_{y,1}+k_y-a(y)+1 = 3q_{y,1}+r_{y,1}$ for appropriate integers $q_{y,1}\ge 0$ and $r_{y,1}\in\{0,1,2\}$, and define
$$b_{y,i}=\begin{cases}
q_{y,1}+1 & \mbox{if } r_{y,1}\ne 0 \mbox{ and } i\equiv 1-\epsilon_ya(y)\mod 3,\mbox{ or}\\
& \mbox{if } r_{y,1}= 2 \mbox{ and } i\equiv 1-\epsilon_y(a(y)+1)\mod 3,\\
q_{y,1} & \mbox{otherwise}.
\end{cases}$$
Then
$$b_i=\sum_{y\in Y_{\infty,0}} b_{y,i} -\delta(r,i)+  \sum_{j=1}^\ell \mu_{y_j,1}$$
where $\delta(r,i)$ equals $1$ if $(r,i)=(0,0)$, and $0$ otherwise.
\item[(b)] 
For $y\in Y_{\infty,0}$, write 
$\mu_{y,3}-\mu_{y,2} = 3q_{y,2}+r_{y,2}$ for appropriate integers $q_{y,2}\ge 0$ and $r_{y,2}\in\{0,1,2\}$, and define
$$c_{y,i}=\begin{cases}
q_{y,2}+1 & \mbox{if } r_{y,2}\ne 0 \mbox{ and }i\equiv 1-\epsilon_y(\mu_{y,2}+k_y+1)\mbox{ mod }3,\mbox{ or}\\
&\mbox{if } r_{y,2}= 2  \mbox{ and } i\equiv 1-\epsilon_y(\mu_{y,2}+k_y+2)\mod 3, \\
q_{y,2} & \mbox{otherwise}.
\end{cases}$$
Then
$$c_i=\sum_{y\in Y_{\infty,0}} c_{y,i} + \sum_{j=1}^\ell \left(\mu_{y_j,3}-\mu_{y_j,2} \right).$$
\item[(c)] 
For $y\in Y_{\infty,0}$, let $a_{y,1}$ and $a_{y,2}$ be as defined in the statement of Theorem $\ref{thm:BC2023}$. For $x\in\{1,2\}$,  write 
$a_{y,x} = 3q_{y,3,x}+r_{y,3,x}$  for appropriate integers $q_{y,3,x}\ge 0$ and $r_{y,3,x}\in\{0,1,2\}$.
Then
$$a_{y,1,i} = \begin{cases}
q_{y,3,1}+1 & \mbox{if } r_{y,3,1}\ne 0 \mbox{ and } i\equiv 1-\epsilon_y(\mu_{y,1}+k_y+1)\mod 3,\mbox{ or}\\
& \mbox{if } r_{y,3,1}= 2 \mbox{ and }i\equiv 1-\epsilon_y(\mu_{y,1}+k_y+2)\mod 3, \\
q_{y,3,1} & \mbox{otherwise},
\end{cases}$$
and
$$a_{y,2,i} = \begin{cases}
q_{y,3,2}+1 & \mbox{if } r_{y,3,2}\ne 0 \mbox{ and } i\equiv 1-\epsilon_y(\mu_{y,1}+k_y+a_{y,1}+1)\mod 3, \mbox{ or}\\
&\mbox{if } r_{y,3,2}= 2 \mbox{ and } i\equiv 1-\epsilon_y(\mu_{y,1}+k_y+a_{y,1}+2)\mod 3, \\
q_{y,3,2} & \mbox{otherwise}.
\end{cases}$$
\end{itemize}
Moreover, the parameters in $(\mathrm{a})$, $(\mathrm{b})$ and $(\mathrm{c})$ are determined by the action of $\rho$ on the uniformizers $\pi_y$ of the branch points $y\in Y_{\mathrm{br}}$.
\end{theorem}

\begin{proof}
It follows from the proof of Theorem \ref{thm:BC2023}, which is proved using similar arguments as in the proof of \cite[Theorem 3.7]{BleherCamacho2023}, that 
\begin{equation}
\label{eq:H0decomposition}
\mathrm{Res}^G_H\HH^0(X,\Omega_X)=\bigoplus_{y\in Y_{\mathrm{br}}} U_y,
\end{equation}
where $U_y$ is the $kH$-module that is the $k$-span of
$$\mathcal{B}_y\,ds:=\{f\,ds\,:\, f\in\mathcal{B}_y\}$$ 
and $\mathcal{B}_y$ is as in Definition \ref{def:3.5}(c). 

Lemmas \ref{lem:bandconnection} and \ref{lem:restrictA4toK4} show the following. Suppose $M$ is a $kH$-submodule and a direct summand of $\HH^0(X,\Omega_X)$, suppose $M$ is isomorphic to a direct sum of copies of $M_{3,1}=M_{3,1}^{(A,B)}$ (resp. a direct sum of copies of the simple $kH$-module $k$), and suppose the action of $\rho$ on $M$ preserves the socle of $M$ as a subspace. Then $M$ is a $kG$-submodule of $\HH^0(X,\Omega_X)$ that is a direct summand. Moreover, for each $i\in\{0,1,2\}$, the multiplicity with which $M_{3,1,i}$ (resp. $S_i$) occurs as a direct summand of $M$ equals the $k$-dimension of the eigenspace with eigenvalue $\zeta^i$ of the action of $\rho$ on the socle of $M$.

Suppose first that $y\in Y_{\infty,0}$, i.e. suppose $Y_{\infty,0}\ne \emptyset$. We consider $U_y$, which will give us the parameters $a_{y,1,i},~a_{y,2,i}$ and $b_{y,i},~c_{y,i}$, for $i=0,1,2$. By  Notation \ref{not:specialbranchY} and Definitions \ref{def:uniformky} and \ref{def:3.5}, we have
$$\epsilon_\infty=-1,~k_\infty=-2,~a(\infty)=0\quad\mbox{and}\quad \epsilon_0=1,~k_0=0,~a(0)=1.$$
Similar to the proof of \cite[Theorem 3.7]{BleherCamacho2023}, we define the following, pairwise disjoint subsets of $\mathcal{B}_y\,ds$:
$$\begin{array}{rcll}
\mathcal{B}_{y,3,1,s_1}\,ds &:=& \{f_{y,1,s_1}\, ds, f_{y,2,s_1}\,ds,f_{y,3,s_1}\,ds\} & \mbox{for } a(y)\le s_1\le \mu_{y,1}+k_y,\\
\mathcal{B}_{y,1,s_2} \,ds&:=& \{f_{y,1,s_2}\,ds\} & \mbox{for } \mu_{y,2} + k_y + 1 \le s_2 \le \mu_{y,3}  +k_y,
\end{array}$$
where we use that $\nu_y=\frac{M_y-m_y}{2}=0$ by Lemma \ref{lem:valuesbranchpoints}(a). As in the proof of \cite[Theorem 3.7]{BleherCamacho2023}, it follows that the $k$-span of each $\mathcal{B}_{y,3,1,s_1}\,ds$ gives a $kH$-module that is isomorphic to $M_{3,1}=M_{3,1}^{(A,B)}$, and the $k$-span of each $\mathcal{B}_{y,1,s_2}\,ds$ gives the simple $kH$-module $k$. We have
\begin{equation}
\label{eq:rhoactioninty0}
\rho (\pi_y^{-a}\, ds) = \rho (s^{-\epsilon_ya}\,ds) = (\zeta s)^{-\epsilon_ya} d(\zeta s) = \zeta^{1-\epsilon_y a} \,\pi_y^{-a}\, ds
\end{equation}
for all non-negative integers $a$.
By Definition \ref{def:3.5} and Remark \ref{rem:really??}, it follows that $\rho$ preserves the socle of each of the $kH$-modules spanned by $\mathcal{B}_{y,3,1,s_1}\,ds$ (resp. by $\mathcal{B}_{y,1,s_2}\,ds$). Therefore, as discussed above, by Lemmas \ref{lem:bandconnection} and \ref{lem:restrictA4toK4}, it follows that each of these $kH$-modules must in fact be a $kG$-module, where the $k$-span of each $\mathcal{B}_{y,3,1,s_1}\,ds$ gives a $kG$-module that is isomorphic to $M_{3,1,i}$, and the $k$-span of each $\mathcal{B}_{y,1,s_2}\,ds$ gives a simple $kG$-module $S_i$, for various $i\in\{0,1,2\}$. To figure out which $i$ occur with which multiplicity, we use (\ref{eq:rhoactioninty0}) to obtain the following simple $kG$-modules for the socles of the $k$-spans of $\mathcal{B}_{y,3,1,s_1}\,ds$, for $a(y)\le s_1\le \mu_{y,1}+k_y$:
$$S_{1-\epsilon_ya(y)\text{ mod }3},S_{1-\epsilon_y(a(y)+1)\text{ mod 3}}, S_{1-\epsilon_y(a(y)+2)\text{ mod 3}}, \ldots, S_{1-\epsilon_y(\mu_{y,1}+k_y)\text{ mod }3}$$
with $\mu_{y,1}+k_y-a(y)+1$ simple modules.
Writing $\mu_{y,1}+k_y-a(y)+1 = 3q_{y,1}+r_{y,1}$ for appropriate integers $q_{y,1}\ge 0$ and $r_{y,1}\in\{0,1,2\}$, we obtain
$$b_{y,i}=\begin{cases}
q_{y,1}+1 & \mbox{if } r_{y,1}\ne 0 \mbox{ and } i\equiv 1-\epsilon_ya(y)\mod 3,\mbox{ or}\\
& \mbox{if } r_{y,1}= 2 \mbox{ and } i\equiv 1-\epsilon_y(a(y)+1)\mod 3,\\
q_{y,1} & \mbox{otherwise}.
\end{cases}$$
On the other hand, the simple $kG$-modules corresponding to the $k$-spans of $\mathcal{B}_{y,1,s_2}\,ds$, for $\mu_{y,2}  +k_y+1 \le s_2 \le \mu_{y,3} +k_y$ are:
$$S_{1-\epsilon_y(\mu_{y,2}+k_y+1)\text{ mod }3}, S_{1-\epsilon_y(\mu_{y,2}+k_y+2)\text{ mod }3},\ldots, S_{1-\epsilon_y(\mu_{y,3}+k_y)\text{ mod }3}$$
with $\mu_{y,3}-\mu_{y,2}$ simple modules. Writing $\mu_{y,3}-\mu_{y,2} = 3q_{y,2}+r_{y,2}$ for appropriate integers $q_{y,2}\ge 0$ and $r_{y,2}\in\{0,1,2\}$, we obtain
$$c_{y,i}=\begin{cases}
q_{y,2}+1 & \mbox{if } r_{y,2}\ne 0 \mbox{ and }i\equiv 1-\epsilon_y(\mu_{y,2}+k_y+1)\mbox{ mod }3,\mbox{ or}\\
&\mbox{if } r_{y,2}= 2  \mbox{ and } i\equiv 1-\epsilon_y(\mu_{y,2}+k_y+2)\mod 3, \\
q_{y,2} & \mbox{otherwise}.
\end{cases}$$

Next, we consider the remaining elements of $\mathcal{B}_y\, ds$:
\begin{eqnarray*}
\mathcal{B}_{\lambda_y} \,ds&:=& \{f_{y,1,s_3}\,ds \; : \; \mu_{y,1}+k_y+1 \le s_3 \le \mu_{y,2} + k_y \}\\
&  \cup & \{f_{y,3,s_3}\,ds \; : \; \mu_{y,1}+k_y+1 \le s_3 \le \mu_{y,2} + k_y\},
\end{eqnarray*}
where we use again that $\nu_y=0$. 
   
Define $n_y=\mu_{y,2}-\mu_{y,1}$. If $n_y=0$ then $a_{y,1}=a_{y,2}=0$, and hence $a_{y,1,i}=a_{y,2,i}=0$ for all $i\in\{0,1,2\}$. 
   
Suppose now that $n_y\ge 1$.
Since $\lambda_y\in\{\zeta,\zeta^2\}$ by Lemma \ref{lem:valuesbranchpoints}, it follows by applying similar arguments as in the proof of \cite[Theorem 3.7]{BleherCamacho2023} and by Lemmas \ref{lem:bandconnection} and \ref{lem:restrictA4toK4} that the $k$-span of $\mathcal{B}_{\lambda_y}\,ds$ is a $kG$-module, which we call $U_{\lambda_y}$. We obtain the following simple $kG$-modules for the socle of $U_{\lambda_y}$: 
\begin{equation}
\label{eq:soclesinfinitystrings}
S_{1-\epsilon_y(\mu_{y,1}+k_y+1)\text{ mod }3}, S_{1-\epsilon_y(\mu_{y,1}+k_y+2)\text{ mod }3},\ldots, S_{1-\epsilon_y(\mu_{y,2} + k_y)\text{ mod }3}
\end{equation}
with $n_y$ simple modules. Moreover, by applying similar arguments as in the proof of \cite[Theorem 3.7]{BleherCamacho2023} and by Lemmas \ref{lem:valuesbranchpoints} and  \ref{lem:bandconnection}, $\mathrm{Res}^G_H\,U_{\lambda_y}$ is a $kH$-module that is isomorphic to 
$$\left(N_{2l_y,\lambda_y}^{(A,B)}\right)^{\oplus a_{y,1}}\;\oplus\; \left(N_{2(l_y-1),\lambda_y}^{(A,B)}\right)^{\oplus a_{y,2}} \cong
\left(N_{2l_y,\ast_y}^{(C,D)}\right)^{\oplus a_{y,1}}\;\oplus\; \left(N_{2(l_y-1),\ast_y}^{(C,D)}\right)^{\oplus a_{y,2}}$$
where the modules are as Notations \ref{not:K4indecomposables} and \ref{not:kCDmodules}, $\ast_y$ is as in Notation \ref{not:organizebranch}, and $l_y,~a_{y,1},~a_{y,2}$ are as defined in Theorem \ref{thm:BC2023}. 
Letting $\delta_y$ be as in Notation \ref{not:deltalambda}, we define the following $n_y\times n_y$ matrix $\Theta$:

If $\delta_y\in\{0,n_y\}$ then $\Theta$ is defined to be the zero matrix, and if $\delta_y\not\in\{0,n_y\}$ then
$$\Theta:=
\left(\begin{array}{ccccccc}
0&\cdots&0&\theta_{y,\delta_y}&\theta_{y,\delta_y+1}&\cdots&\theta_{y,n_y-1}\\
\vdots&&&0&\theta_{y,\delta_y}&\ddots&\vdots\\
\vdots&&&&\ddots&\ddots&\theta_{y,\delta_y+1}\\
\vdots&&&&&0&\theta_{y,\delta_y}\\
\vdots&&&&&&0\\
\vdots&&&&&&\vdots\\
0&\cdots&\cdots&\cdots&\cdots&\cdots&0
\end{array}\right).$$ 
In particular, if $\delta_y\not\in\{0,n_y\}$ then the first $\delta_y$ columns of $\Theta$ are zero and $\theta_{y,\delta_y}\ne 0$. 
Applying similar arguments as in the proof of \cite[Theorem 3.7]{BleherCamacho2023} and using Remark \ref{rem:AB2}, we obtain that $C$ and $D$ from Definition \ref{def:CD} act as the following matrices on $U_{\lambda_y}$ with respect to the $k$-basis given by $\mathcal{B}_{\lambda_y} \,ds$:
$$C\longleftrightarrow
	\begin{pmatrix}
	0_{n_y} & \zeta^2\, \Theta_1\\
	0_{n_y} & 0_{n_y}
        \end{pmatrix},\quad 
D\longleftrightarrow
	\begin{pmatrix}
	0_{n_y} & \zeta^2\, \Theta_2\\
	0_{n_y} & 0_{n_y}
        \end{pmatrix}$$
where 
$$(\Theta_1,\Theta_2) =\begin{cases}
(I_{n_y}+\Theta,\Theta) & \mbox{if } \lambda_y = \zeta,\\
(\Theta, I_{n_y}+\Theta) & \mbox{if } \lambda_y = \zeta^2
\end{cases}$$
and $0_{n_y}$ (resp. $I_{n_y}$) is the $n_y\times n_y$ zero matrix (resp. identity matrix).

If $\delta_y\in\{0,n_y\}$, then $l_y=1$, $a_{y,1}=\mu_{y,2}-\mu_{y,1}$ and $a_{y,2}=0$, which means $a_{y,2,i}=0$ for $0\le i \le 2$. In this case, writing $a_{y,1}=\mu_{y,2}-\mu_{y,1} = 3q_{y,3,1}+r_{y,3,1}$ for appropriate integers $q_{y,3,1}\ge 0$ and $r_{y,3,1}\in\{0,1,2\}$, we obtain
$$a_{y,1,i} = \begin{cases}
q_{y,3,1}+1 & \mbox{if } r_{y,3,1}\ne 0 \mbox{ and } i\equiv 1-\epsilon_y(\mu_{y,1}+k_y+1)\mod 3, \mbox{ or}\\
&\mbox{if } r_{y,3,1}= 2 \mbox{ and }i\equiv 1-\epsilon_y(\mu_{y,1}+k_y+2)\mod 3,  \\
   q_{y,3,1} & \mbox{otherwise}.
\end{cases}$$

Next suppose $0<\delta_y<n_y$. We perform two changes of bases using the matrices $E_1$ and $E_2$, where
$$E_1=\begin{pmatrix}
I_{n_y} & 0_{n_y}\\
0_{n_y} &  \zeta(I_{n_y}+\Theta)^{-1}
\end{pmatrix}
\quad\mbox{and}\quad
E_2 =\begin{pmatrix}
T_{n_y} & 0_{n_y}\\
0_{n_y} & T_{n_y}
\end{pmatrix},$$
and $T_{n_y}$ is an $n_y\times n_y$ matrix such that
$$T_{n_y}^{-1} \left( \Theta(I_{n_y}+\Theta)^{-1}\right) T_{n_y}$$
is a block diagonal matrix with $a_{y,1}$ blocks of the form $J_{l_y}(0)$ and $a_{y,2}=\delta_y - a_{y,1}$ blocks of the form $J_{l_y-1}(0)$. More precisely, suppose $\{b_1,\ldots, b_{n_y}\}$ is the standard basis of $k^{n_y}$ on which $\Theta(I_{n_y}+\Theta)^{-1}$ acts. Considering $\left(\Theta(I_{n_y}+\Theta)^{-1}\right)^{l_y-1}$ we see that the first $a_{y,1}$ vectors $b_1,\ldots, b_{a_{y,1}}$ are eigenvectors of $a_{y,1}$ Jordan blocks of $\Theta(I_{n_y}+\Theta)^{-1}$ of size $l_y$. On the other hand, the next $a_{y,2}=\delta_y - a_{y,1}$ vectors $b_{a_{y,1}+1}, \ldots,b_{\delta_y}$  are eigenvectors of $a_{y,2}$ Jordan blocks of $\Theta(I_{n_y}+\Theta)^{-1}$ of size $l_y-1$. Define 
$$t_{y,j}:=\begin{cases} l_y & \text{ if } 1\le j \le a_{y,1},\\
l_y-1 &  \text{ if } a_{y,1}+1\le j \le \delta_y.\end{cases}$$
We now find vectors $v_1,\ldots,v_{\delta_y}$ in $k^{n_y}$ such that
\begin{equation}
\label{eq:subtle!}
\left(\Theta(I_{n_y}+\Theta)^{-1}\right)^{t_{y,j}-1} v_j = b_j\quad\text{ for } 1\le j\le \delta_y.
\end{equation}
We therefore obtain a new basis of $k^{n_y}$ given by
$$\{c_1,\ldots, c_{n_y}\} 
= \bigcup_{j=1}^{\delta_y}\{\left(\Theta(I_{n_y} \Theta)^{-1}\right)^{t_{y,j}-1}v_j,\ldots, \Theta(I_{n_y}+\Theta)^{-1} v_j,v_j\}.$$
The matrix $T_{n_y}$ is then defined to be the change of basis matrix expressing the vectors $c_1,\ldots, c_{n_y}$ in terms of the standard basis $\{b_1,\ldots, b_{n_y}\}$ of $k^{n_y}$. By (\ref{eq:subtle!}), there exist $1=i_1<i_2<\ldots< i_{\delta_y}$ in $\{1,\ldots, n_y\}$ such that $c_{i_j}=b_j$ for $1\le j\le \delta_y$. Note that $i_{j+1}-i_j=t_{y,j}$ for $1\le j\le \delta_y-1$. Considering the transition matrix $E_1\cdot E_2$, we see that, for $1\le j\le \delta_y$, the basis element at position $i_j$ in the new basis for $U_{\lambda_y}$ equals the element at position $j$ in the basis $\mathcal{B}_{\lambda_y} \,ds$. 

It follows that, for $1\le j\le a_{y,1}$, the element at position $j$ in the basis $\mathcal{B}_{\lambda_y} \,ds$ corresponds, up to a non-zero scalar multiple, to the second basis vector of the $k$-basis for $N_{2l_y,\ast,i}$, for some $i\in\{0,1,2\}$, according to Remark \ref{rem:stringbandmodules}(ii) applied to the appropriate string in Notation \ref{not:A4indecomposables}(3). Similarly, for $a_{y,1}+1\le j\le \delta_y$, the element at position $j$ in the basis $\mathcal{B}_{\lambda_y} \,ds$ corresponds, up to a non-zero scalar multiple, to the second basis vector of the $k$-basis for $N_{2(l_y-1),\ast,i}$, for some $i\in\{0,1,2\}$.
Moreover, using the strings in Notation \ref{not:A4indecomposables}(3), we see that the action of $\rho$ on the second basis vector for $N_{2n,\ast,i}$ determines $i$ for all positive integers $n$. Therefore, to count the number of $i\in\{0,1,2\}$ that occur, we write, for $x=1,2$, $a_{y,x} = 3q_{y,3,x}+r_{y,3,x}$  for appropriate integers $q_{y,3,x}\ge 0$ and $r_{y,3,x}\in\{0,1,2\}$. Using (\ref{eq:soclesinfinitystrings}),
we conclude that
$$a_{y,1,i} = \begin{cases}
q_{y,3,1}+1 & \mbox{if } r_{y,3,1}\ne 0 \mbox{ and } i\equiv 1-\epsilon_y(\mu_{y,1}+k_y+1)\mod 3,\mbox{ or}\\
& \mbox{if } r_{y,3,1}= 2 \mbox{ and }i\equiv 1-\epsilon_y(\mu_{y,1}+k_y+2)\mod 3, \\
q_{y,3,1} & \mbox{otherwise},
\end{cases}$$
and
$$a_{y,2,i} = \begin{cases}
q_{y,3,2}+1 & \mbox{if } r_{y,3,2}\ne 0 \mbox{ and } i\equiv 1-\epsilon_y(\mu_{y,1}+k_y+a_{y,1}+1)\mod 3, \mbox{ or}\\
&\mbox{if } r_{y,3,2}= 2 \mbox{ and } i\equiv 1-\epsilon_y(\mu_{y,1}+k_y+a_{y,1}+2)\mod 3, \\
q_{y,3,2} & \mbox{otherwise}.
\end{cases}$$

We have now established that if $y\in Y_{\infty,0}$, then the $kH$-module $U_y$ from (\ref{eq:H0decomposition}) is in fact a $kG$-module, and we have determined its decomposition into a direct sum of indecomposable $kG$-modules.

Suppose next that $y\in Y_{\mathrm{br}}-Y_{\infty,0}=\{y_j,y_j',y_j''\,:\,1\le j\le \ell\}$, i.e. suppose $\ell \ge 1$. In particular, $k_y=0$ and $$a(y)=\begin{cases} 2 &\mbox{if } r=0 \mbox{ and } y=y_1,\\
1& \mbox{otherwise}. \end{cases}$$ 
We define
$$\mathcal{B}^{\dagger}_y:=\begin{cases}
\{f_{y,1,1},f_{y,2,1},f_{y,3,1}\} & \mbox{if }r=0\mbox{ and } y\in\{y_1',y_1''\},\\
\emptyset & \mbox{otherwise},
\end{cases}$$
and we define $\mathcal{B}^{\ddagger}_y:=\mathcal{B}_y-\mathcal{B}^{\dagger}_y$. Moreover, we define $U^{\dagger}_y$ to be the $k$-span of $\mathcal{B}^{\dagger}_y\,ds = \{f\,ds\,:\, f\in \mathcal{B}^{\dagger}_y\}$, and we define $U^{\ddagger}_y$ to be the $k$-span of $\mathcal{B}^{\ddagger}_y\,ds = \{f\,ds\,:\, f\in \mathcal{B}^{\ddagger}_y\}$.
Applying similar arguments as in the proof of \cite[Theorem 3.7]{BleherCamacho2023}, we see that $U^{\dagger}_y$ and $U^{\ddagger}_y$ are $kH$-modules and that, as $kH$-modules,
$$U_y=U^{\dagger}_y \oplus U^{\ddagger}_y.$$
We define
$$U^{\dagger}:=\bigoplus_{y\in Y_{\mathrm{br}}-Y_{\infty,0}} U^{\dagger}_y=\begin{cases}
U^{\dagger}_{y_1'}\oplus U^{\dagger}_{y_1''} & \mbox{if }r=0,\\
0 & \mbox{otherwise}.
\end{cases}$$

If $U^{\dagger}$ is not the zero module then $r=0$. In this case, $U^{\dagger}$ is the $k$-span of 
$$\{f_{y_1',1,1}\, ds,  f_{y_1',2,1}\, ds, f_{y_1',3,1}\, ds, f_{y_1'',1,1}\, ds, f_{y_1'',2,1}\, ds, f_{y_1'',3,1}\, ds\}.$$
As in the proof of \cite[Theorem 3.7]{BleherCamacho2023}, it follows that $U^{\dagger}$ is a $kH$-module that is isomorphic to $M_{3,1}\oplus M_{3,1}$. 
Moreover, the action of $\rho$ on the socle of $U^{\dagger}$ is given as follows:
\begin{eqnarray*}
\rho(f_{y_1',1,1}\, ds) &=& \rho((s-\psi_1)^{-1}(s-\zeta\psi_1)^{-1})\,ds) =  (\zeta s-\psi_1)^{-1}(\zeta s-\zeta\psi_1)^{-1}\,d(\zeta s) \\
&=& 
\zeta^2\,(s-\zeta^2\psi_1)^{-1}(s-\psi_1)^{-1}\,ds= (f_{y_1',1,1}+\zeta f_{y_1'',1,1})\,ds
\end{eqnarray*}
and
\begin{eqnarray*}
\rho(f_{y_1'',1,1}\, ds) &=& \rho((s-\zeta\psi_1)^{-1}(s-\zeta^2\psi_1)^{-1})\,ds) = (\zeta s-\zeta\psi_1)^{-1}(\zeta s-\zeta^2\psi_1)^{-1}\,d(\zeta s) \\
&=&
\zeta^2\,(s-\psi_1)^{-1}(s-\zeta\psi_1)^{-1}\,ds=\zeta^2\, f_{y_1',1,1}\,ds.
\end{eqnarray*}
This means that $\rho$ preserves the $k$-span of $\{f_{y_1',1,1}\, ds, f_{y_1'',1,1}\, ds\}$, which is the socle of $U^{\dagger}$. By Lemmas \ref{lem:bandconnection} and \ref{lem:restrictA4toK4}, it follows that $U^{\dagger}$ is a $kG$-module that is isomorphic to $M_{3,1,i_1'}\oplus M_{3,1,i_1''}$ for certain $i_1',i_1''\in\{0,1,2\}$. To determine $i_1'$ and $i_1''$, we compute the eigenspaces of the action of $\rho$ on the socle of $U^{\dagger}$. Since
\begin{eqnarray*}
\rho((f_{y_1',1,1}+ f_{y_1'',1,1})\, ds) &=& \zeta(f_{y_1',1,1}+ f_{y_1'',1,1})\, ds,\quad\mbox{and}\\
\rho((\zeta f_{y_1',1,1}+ f_{y_1'',1,1})\, ds) &=& \zeta^2(\zeta f_{y_1',1,1}+ f_{y_1'',1,1})\, ds,
\end{eqnarray*}
we obtain 
$$U^{\dagger} \cong \begin{cases} M_{3,1,1 }\oplus M_{3,1,2} & \mbox{if }r=0,\\
0 & \mbox{otherwise}.
\end{cases}$$

We next look at the $kH$-modules $U^{\ddagger}_y$. Fix $j\in\{1,\ldots,\ell\}$, and consider the three branch points $y:= y_j$, $y' := y_j'$ and $y'' := y_j''$. 
Using Definition \ref{def:3.5}(b), we have
$$\begin{array}{r@{\;=\;}c@{\;=\;}c@{\;=\;}l}
\rho^{-1}(f_{y,1,a}\, ds) & (\zeta^{-1}s-\psi_j)^{-a}\,d(\zeta^{-1}s) & \zeta^{a-1} \,(s-\zeta\psi_j)^{-a}\,ds & \zeta^{a-1} \, f_{y',1,a}\, ds,\\
\rho^{-2}(f_{y,1,a}\, ds) & (\zeta^{-2}s-\psi_j)^{-a}\,d(\zeta^{-2}s) & \zeta^{2(a-1)} \,(s-\zeta^2\psi_j)^{-a}\,ds & \zeta^{2(a-1)} \, f_{y'',1,a}\, ds,
\end{array}$$
for $2\le a \le \mu_{y,3}+\nu_y$. Moreover,
\begin{eqnarray*}
\rho^{-1}(f_{y,1,1}\, ds)&=&\left\{
\begin{array}{r@{\;=\;}rl}
(\zeta^{-1}s-\psi_j)^{-1}\,d(\zeta^{-1}s) &f_{y',1,1}\, ds &\mbox{if } r\ge 1,\\
(\zeta^{-1}s-\psi_1)^{-1}(\zeta^{-1}s-\psi_j)^{-1}\,d(\zeta^{-1}s) & \zeta f_{y',1,1}\, ds &\mbox{if } r=0\mbox{ and } j\ge 2,
\end{array}\right.\\
\rho^{-2}(f_{y,1,1}\, ds)&=&\left\{
\begin{array}{r@{\;=\;}rl}
(\zeta^{-2}s-\psi_j)^{-1}\,d(\zeta^{-2}s) &f_{y'',1,1}\, ds &\mbox{if } r\ge 1,\\
(\zeta^{-2}s-\psi_1)^{-1}(\zeta^{-2}s-\psi_j)^{-1}\,d(\zeta^{-2}s) & \zeta^2 f_{y'',1,1}\, ds &\mbox{if } r=0\mbox{ and } j\ge 2.
\end{array}\right.
\end{eqnarray*}
This implies that the action of $\rho^{-1}$ on $\HH^0(X,\Omega_X)$ maps the socle of $U^{\ddagger}_y$ bijectively onto the socle of $U^{\ddagger}_{y'}$, and the socle of $U^{\ddagger}_{y'}$ bijectively onto the socle of $U^{\ddagger}_{y''}$. Moreover, we see from Lemma \ref{lem:valuesbranchpoints}(b) that all the parameters needed to write down Theorem \ref{thm:BC2023} are the same for the branch points $y$, $y'$ and $y''$. 
Since, by Lemma \ref{lem:Hbasis}, $\{f\, ds :f\in\mathcal{B}\}$ is a $k$-basis of $\HH^0(X,\Omega_X)$, it follows that there is a unique $kH$-submodule of $\HH^0(X,\Omega_X)$ whose socle is the same as the socle of $U^{\ddagger}_{y'}$ (resp. of $U^{\ddagger}_{y''}$) and this has to be $U^{\ddagger}_{y'}$ (resp. $U^{\ddagger}_{y''}$). This means that the $kG$-submodule of $\HH^0(X,\Omega_X)$ generated by $U^{\ddagger}_y$ equals
$$\mathrm{Ind}_H^G U^{\ddagger}_y = U^{\ddagger}_y\oplus \rho^{-1} U^{\ddagger}_y \oplus \rho^{-2} U^{\ddagger}_y = U^{\ddagger}_y\oplus U^{\ddagger}_{y'} \oplus U^{\ddagger}_{y''}.$$
Therefore, the remainder of Theorem \ref{thm:main} follows now directly from Lemmas \ref{lem:bandconnection} and \ref{lem:restrictA4toK4} and using similar arguments as in the proof of \cite[Theorem 3.7]{BleherCamacho2023}.
\end{proof}  

As a consequence of Lemma \ref{lem:HKGcovers} and Theorem \ref{thm:main}, we obtain the precise structure of the holomorphic differentials of Harbater-Katz-Gabber $G$-covers:

\begin{corollary}
\label{cor:main}
Assume that $\pi:X\to \mathbb{P}^1_k=:Z$ is a Harbater-Katz-Gabber $G$-cover, as defined in Definition $\ref{def:HKGcovers}$.
Then $\pi$ is an alternating four cover satisfying Assumption $\ref{ass:general}$, and we can assume, without loss of generality, that $Y_{\mathrm{br}}=\{\infty\}$. There is a $kG$-module isomorphism between $\HH^0(X,\Omega_X)$ and the direct sum 
$$\bigoplus_{i=0}^2 \left(N_{2l_\infty,\ast_\infty,i}^{\oplus a_{\infty,1,i}} \oplus N_{2(l_\infty-1),\ast_\infty,i}^{\oplus a_{\infty,2,i}}\right)\oplus
\bigoplus_{i=0}^2 M_{3,1,i}^{\oplus b_{\infty,i}} \oplus \bigoplus_{i=0}^2 S_i^{\oplus c_{\infty,i}}$$
where $l_\infty$ is as in the statement of Theorem $\ref{thm:BC2023}$, and the remaining parameters are as follows:
\begin{itemize}
\item[(a)] 
Write $\mu_{\infty,1}-1 = 3q_{\infty,1}+r_{\infty,1}$ for appropriate integers $q_{\infty,1}\ge 0$ and $r_{\infty,1}\in\{0,1,2\}$. Then
$$b_{\infty,i}=\begin{cases}
q_{\infty,1}+1 & \mbox{if } r_{\infty,1}\ne 0 \mbox{ and } i\equiv 1\mod 3,\mbox{ or}\\
& \mbox{if } r_{\infty,1}= 2 \mbox{ and } i\equiv 2\mod 3,\\
q_{\infty,1} & \mbox{otherwise}.
\end{cases}$$
\item[(b)] 
Write $\mu_{\infty,3}-\mu_{\infty,2} = 3q_{\infty,2}+r_{\infty,2}$ for appropriate integers $q_{\infty,2}\ge 0$ and $r_{\infty,2}\in\{0,1,2\}$. Then
$$c_{\infty,i}=\begin{cases}
q_{\infty,2}+1 & \mbox{if } r_{\infty,2}\ne 0 \mbox{ and }i\equiv \mu_{\infty,2}\mbox{ mod }3,\mbox{ or}\\
&\mbox{if } r_{\infty,2}= 2  \mbox{ and } i\equiv \mu_{\infty,2}+1\mod 3, \\
q_{\infty,2} & \mbox{otherwise}.
\end{cases}$$
\item[(c)] 
Let $a_{\infty,1}$ and $a_{\infty,2}$ be as defined in the statement of Theorem $\ref{thm:BC2023}$. For $x\in\{1,2\}$, write
$a_{\infty,x} = 3q_{\infty,3,x}+r_{\infty,3,x}$  for appropriate integers $q_{\infty,3,x}\ge 0$ and $r_{\infty,3,x}\in\{0,1,2\}$. Then
$$a_{\infty,1,i} = \begin{cases}
q_{\infty,3,1}+1 & \mbox{if } r_{\infty,3,1}\ne 0 \mbox{ and } i\equiv \mu_{\infty,1}\mod 3,\mbox{ or}\\
& \mbox{if } r_{\infty,3,1}= 2 \mbox{ and }i\equiv \mu_{\infty,1}+1\mod 3, \\
q_{\infty,3,1} & \mbox{otherwise},
\end{cases}$$
and
$$a_{\infty,2,i} = \begin{cases}
q_{\infty,3,2}+1 & \mbox{if } r_{\infty,3,2}\ne 0 \mbox{ and } i\equiv \mu_{\infty,1}+a_{\infty,1}\mod 3, \mbox{ or}\\
&\mbox{if } r_{\infty,3,2}= 2 \mbox{ and } i\equiv \mu_{\infty,1}+a_{\infty,1}+1\mod 3, \\
q_{\infty,3,2} & \mbox{otherwise}.
\end{cases}$$
\end{itemize}
\end{corollary}


\section{Examples}
\label{s:examples}

Throughout this section, we make the same assumptions and use the same notations as in the statement of Theorem \ref{thm:main}. The main goal of this section is to show the following result.

\begin{proposition}
\label{prop:list}
The list of isomorphism classes of indecomposable $kG$-modules that actually occur as direct summands of $\HH^0(X,\Omega_X)$ in the situation of Theorem $\ref{thm:main}$ for various $X$ is infinite and given as follows:
\begin{equation}
\label{eq:themodules}
\left\{N_{2n,0,i},N_{2n,\infty,i}, B_{6n,\mu}, M_{3,1,i}, S_i\;:\; n\in\mathbb{Z}^+,\mu\in k^\times, i\in \{0,1,2\}\right\}.
\end{equation}
Moreover, if we restrict ourselves to Harbater-Katz-Gabber $G$-covers $X$, as in Corollary $\ref{cor:main}$, then the list in $(\ref{eq:themodules})$ is replaced by
\begin{equation}
\label{eq:themodulesHKG}
\left\{N_{2n,0,i},N_{2n,\infty,i}, M_{3,1,i}, S_i\;:\; n\in\mathbb{Z}^+, i\in \{0,1,2\}\right\}.
\end{equation}
In particular, both lists $(\ref{eq:themodules})$ and $(\ref{eq:themodulesHKG})$ contain indecomposable $kG$-modules of arbitrarily large $k$-dimension.
\end{proposition}

We prove this proposition by giving appropriate examples that show that all modules in (\ref{eq:themodules}) (resp. in (\ref{eq:themodulesHKG})) actually occur. 

\subsection{Example 1}
\label{ss:ex1}

In the first example, we concentrate on Harbater-Katz-Gabber covers and show that all indecomposable modules in (\ref{eq:themodulesHKG}) occur. 

Let $n\in\mathbb{Z}^+$, and suppose $n\equiv r_n\mod 3$, where $r_n\in\{0,1,2\}$. Let $x\in\{1,2\}$, and define
$$\alpha:= s^{(32n-2r_n+4)x-3} (s^{16x}+1).$$
Then 
$$\rho\alpha = \zeta^{2x}\, s^{(32n-2r_n+4)x-3} (s^{16x}+\zeta^{2x}).$$
In particular, it follows immediately that $\mathrm{Tr}_{K/J}(\alpha)=0$.
We see that for $h_\alpha=0$, $\alpha$ has the same form as in Notation \ref{not:branchA4}(b), where 
$$\ell=0, \quad d=16x, \quad  p_\infty-d=(32n-2r_n+4)x-3.$$
In other words, $p_\infty=(32n-2r_n+20)x-3$. Moreover, since $\mathrm{ord}_{s^{-1}}(\alpha) = -(32n-2r_n+20)x+3$ is negative and odd, it follows that $\alpha\neq \xi^2-\xi$ for any $\xi\in k(s)$.
We have $Y_\mathrm{br}=\{\infty\}$. Moreover, if  $x_\infty\in X$ lies above $\infty\in Y$, then
$$\mathrm{ord}_{x_\infty}(u)=\mathrm{ord}_{x_\infty}(\rho u)=-(64n-4r_n+40)x+6.$$
Let 
$$w_\infty = \rho u + \left(\zeta^x+\sum_{j=1}^n s^{-8jx}\right)u.$$
A straightforward calculation shows that 
\begin{eqnarray*}
w_\infty^2-w_\infty&=&s^{(16n-2r_n+4)x-3} + 
\left(1+s^{-8x}+s^{-16x} +\sum_{j=2}^n\left(s^{-8jx}+s^{-16jx}\right)\right)u.
\end{eqnarray*}
Since 
$$\mathrm{ord}_{x_\infty}\left(s^{(16n-2r_n+4)x-3}\right)=-(64n-8r_n+16)x+12
> -(64n-4r_n+40)x+6,$$
we obtain 
$$\mathrm{ord}_{x_\infty}(w_\infty^2-w_\infty)=-(64n-4r_n+40)x+6$$
which implies $\mathrm{ord}_{x_\infty}w_\infty=-(32n-2r_n+20)x+3$.
By Notations \ref{not:deltalambda} and \ref{not:organizebranch}, we have
$$\delta_\infty= 8x, \quad \lambda_\infty=\zeta^x,
\quad\mbox{and}\quad 
\ast_\infty = \begin{cases}
0 & \mbox{if } x=1, \\
\infty &  \mbox{if } x=2.
\end{cases}$$

We next show that, for $x\in\{1,2\}$,  $q_{\infty,1},~q_{\infty,2},~q_{\infty,3,1},~q_{\infty,3,1}$ from Theorem \ref{thm:main} (and also from Corollary \ref{cor:main}) are at least one and that $l_\infty=n+1$, which will show that all indecomposable modules in (\ref{eq:themodulesHKG}) occur. 

We have $m_\infty=M_\infty=(32n-2r_n+20)x-3$, which implies
\begin{eqnarray*}
\mu_{\infty,1} &=& (8n+5)x-\left\lceil \frac{r_nx}{2}\right\rceil,\\
\mu_{\infty,2} &=& (16n+10-r_n)x-1,\\
\mu_{\infty,3} &=& (24n+15-r_n)x-1 -\left\lceil \frac{r_nx+1}{2}\right\rceil.
\end{eqnarray*}
Since
$$\mu_{\infty,1}-1 = (8n+5)x-1-\left\lceil \frac{r_nx}{2}\right\rceil
\ge 8+5-1-2 = 10> 3,$$
$q_{\infty,1}\ge 1$. This shows that, for all $i\in\{0,1,2\}$, $M_{3,1,i}$ is a direct summand of  $\HH^0(X,\Omega_X)$. Since 
$$\mu_{\infty,3}-\mu_{\infty,2} = (8n+5)x-\left\lceil \frac{r_nx+1}{2}\right\rceil
\ge 8+5-3 = 10 > 3,$$
$q_{\infty,2}\ge 1$. This shows that, for all $i\in\{0,1,2\}$, $S_i$ is a direct summand of  $\HH^0(X,\Omega_X)$. Finally, we consider
$$\mu_{\infty,2}-\mu_{\infty,1} = (8n+5-r_n)x-1+\left\lceil \frac{r_nx}{2}\right\rceil.$$
Since 
$$0< (5-r_n)x-1+\left\lceil \frac{r_nx}{2}\right\rceil < 8x$$
and since $\delta_\infty=8x$, we obtain
$$l_\infty = n+1\quad \mbox{and}\quad
a_{\infty,1} = (5-r_n)x-1+\left\lceil \frac{r_nx}{2}\right\rceil\ge 3,$$
which implies $q_{\infty,3,1}\ge 1$.
On the other hand,
$$a_{\infty,2} = 8x-a_{\infty,1} = (3+r_n)x+1-\left\lceil \frac{r_nx}{2}\right\rceil\ge 3,$$
which implies $q_{\infty,3,2}\ge 1$.
This shows that, for all $i\in\{0,1,2\}$, both $N_{2(n+1),\ast_\infty, i}$ and $N_{2n,\ast_\infty, i}$ are direct summands of $\HH^0(X,\Omega_X)$.


\subsection{Example 2}
\label{ss:ex2}

Let $n\in\mathbb{Z}^+$, and define
$$\alpha:=\frac{\zeta\, s^{12n+2}}{(s-1)^{4n-1}(s-\zeta)^{4n-3}(s-\zeta^2)^{4n-1}}.$$ 
Then
$$\rho\alpha=
\frac{\zeta^2\,s^{12n+2}}{(s-1)^{4n-3}(s-\zeta)^{4n-1}(s-\zeta^2)^{4n-1}}.$$
Since
$$\alpha=\frac{\zeta \,s^{12n+2}(s^2-\zeta^2)}{(s^3-1)^{4n-1}},$$
it follows that $\mathrm{Tr}_{K/J}(\alpha)=0$.
We see that for $h_\alpha=0$, $\alpha$ has the same form as in Notation \ref{not:branchA4}(b), where 
$$\ell=1, \quad \psi_1=1,\quad d=0, \quad p_1=p_1''=4n-1,p_1'=4n-3,\quad p_\infty+(p_1+p_1'+p_1'')-d=12n+2.$$
In other words, $p_\infty=7$. Moreover, since $\mathrm{ord}_{s^{-1}}(\alpha) = -7$
is negative and odd, it follows that $\alpha\neq \xi^2-\xi$ for any $\xi\in k(s)$.
We have $Y_\mathrm{br}=\{\infty,y_1,y_1',y_1''\}$, where $y_1$ corresponds to $s-\psi_1=s-1$ and $m_{y_1}=4n-3<4n-1=M_{y_1}$. By Notations \ref{not:deltalambda} and \ref{not:organizebranch}, we have
$$\delta_{y_1}=-1,\quad \lambda_1\in \{0,1,\infty\}\quad \mbox{and}\quad \phi_1\in \{1,\zeta,\zeta^2\}.$$
Moreover,
$$\mu_{y_1,2}-\mu_{y_1,1} +\frac{M_{y_1}-m_{y_1}}{2}= 
(2n-1)-n+1=n$$
which implies, by Theorem \ref{thm:BC2023}, that 
$$l_{y_1}=n,\quad a_{y_1,1}=1,\quad\mbox{and}\quad
a_{y_1,2}=0.$$
By Theorem \ref{thm:main}, we obtain that $B_{6n,1}$ occurs as a direct summand of $\HH^0(X,\Omega_x)$ with multiplicity one.


\subsection{Example 3}
\label{ss:ex3}

Let $n\in\mathbb{Z}^+$, let $\mu\in k^\times-\{1\}$, and let $\psi \in k^\times-\{1,\zeta,\zeta^2\}$ be such that $\psi^3=\mu$. Define 
$$\alpha:=\frac{\psi^{3(4n+1)}s^{12n+2}(s^2+1)}{(s^3-\psi^3)^{4n+1}},$$ 
which implies
$$\rho\alpha=
\frac{\psi^{3(4n+1)}\zeta^2s^{12n+2}(\zeta^2s^2+1)}{(s^3-\psi^3)^{4n+1}}$$
and $\mathrm{Tr}_{K/J}(\alpha)=0$. 
We see that for $h_\alpha=0$, $\alpha$ has the same form as in Notation \ref{not:branchA4}(b), where 
$$\ell=1, \quad \psi_1=\psi,\quad d=2, \quad p_1=p_1'=p_1''=4n+1,\quad p_\infty+(p_1+p_1'+p_1'')-d=12n+2.$$
In other words, $p_\infty=1$. Moreover, since $\mathrm{ord}_{s^{-1}}(\alpha) = -1$
is negative and odd, it follows that $\alpha\neq \xi^2-\xi$ for any $\xi\in k(s)$.
We have $Y_\mathrm{br}=\{\infty,y_1,y_1',y_1''\}$, where $y_1$ corresponds to $s-\psi_1=s-\psi$ and $m_{y_1}=M_{y_1}=4n+1$. Moreover, if  $x_1\in X$ lies above $y_1$, then
$$\mathrm{ord}_{x_1}(u)=\mathrm{ord}_{x_1}(\rho u)=-(8n+2).$$
Let 
$$w_{y_1} = \rho u + \left(\frac{\zeta+\zeta^2\psi}{1+\psi}+\sum_{j=1}^n
\frac{(s-\psi)^j}{(1+\psi)^{j+1}}\right)u.$$
A straightforward calculation shows that then
\begin{eqnarray*}
w_{y_1}^2-w_{y_1}&=&\frac{\psi^{3(4n+1)}s^{12n+2}}{(1+\psi^2)^{n+1}(s-\zeta\psi)^{4n+1}(s-\zeta^2\psi)^{4n+1}}\,(s-\psi)^{-(2n-1)}\\
&+& \left(\frac{1+\psi+\psi^2}{(1+\psi)^2}+\frac{(s-\psi)}{(1+\psi)^2}+\frac{(s-\psi)^2}{(1+\psi)^4}+\sum_{j=2}^n\left(\frac{(s-\psi)^j}{(1+\psi)^{j+1}}+\frac{(s-\psi)^{2j}}{(1+\psi)^{2(j+1)}}\right)\right)u.
\end{eqnarray*}
Since $1+\psi+\psi^2\ne 0$ and since $\mathrm{ord}_{x_1}((s-\psi)^{-(2n-1)})=-(8n-4)>-(8n+2)$, we obtain 
$$\mathrm{ord}_{x_1}(w_{y_1}^2-w_{y_1})=-(8n+2)$$
which implies $\mathrm{ord}_{x_1}w_{y_1}=-(4n+1)$.
By Notations \ref{not:deltalambda} and \ref{not:organizebranch}, we have
$$\delta_{y_1}= 1, \quad \lambda_1=\lambda_{y_1}=\frac{\zeta+\zeta^2\psi}{1+\psi}, \quad\mbox{and}\quad
\phi_1 =  \frac{\zeta+\lambda_1}{\zeta^2+\lambda_1}=\psi.$$
Since $m_{y_1}=M_{y_1}=4n+1$, we have
$$\mu_{y_1,2}-\mu_{y_1,1} = 
(2n+1)-(n+1)=n$$
which implies, by Theorem \ref{thm:BC2023}, that 
$$l_{y_1}=n,\quad a_{y_1,1}=1,\quad\mbox{and}\quad
a_{y_1,2}=0.$$
By Theorem \ref{thm:main}, we obtain that
$$B_{6n,\mu} = B_{6l_{y_1},\phi_1^3}$$
occurs as a direct summand of $\HH^0(X,\Omega_x)$ with multiplicity one.
In particular, Example 3 finishes the proof of Proposition \ref{prop:list}.


\section{Appendix: Indecomposable modules in characteristic two}
\label{s:repV4A4}

Let $k$ be an algebraically closed field of characteristic 2, and let 
$G=\langle \sigma,\rho\rangle\cong A_4$, $H=\langle \sigma,\tau\rangle\cong\mathbb{Z}/2\times\mathbb{Z}/2$ and $C=G/H=\langle \overline{\rho}\rangle \cong \mathbb{Z}/3$ be as in Notation \ref{not:general}.

In this appendix, we describe the finitely generated indecomposable modules for $kH$ and $kG$, and we also describe the relationship between them, given by induction and restriction.


\subsection{Basic string algebras over $k$}
\label{ss:stringalg}

Let $\Omega$ be a basic string algebra over $k$, in the sense of \cite[Section 3]{ButlerRingel1987}. This means that $\Omega\cong k\mathcal{Q}/\mathcal{I}$ for a finite quiver $\mathcal{Q}$ and an admissible ideal $\mathcal{I}$ of $k\mathcal{Q}$ satisfying the following three conditions:
\begin{enumerate}
\item[(i)] Every vertex of $\mathcal{Q}$ is the starting point of at most two arrows, and every vertex is the end point of at most two arrows.
\item[(ii)] Given an arrow $\beta$ in $\mathcal{Q}$, there is at most one arrow $\gamma$ with $s(\beta)=e(\gamma)$ and $\beta\gamma\not\in \mathcal{I}$, and there is at most one arrow $\alpha$ with $s(\alpha)=e(\beta)$ and $\alpha\beta\not\in \mathcal{I}$.
\item[(iii)] The ideal $\mathcal{I}$ is generated by zero relations, i.e. by paths in $k\mathcal{Q}$.
\end{enumerate}
By \cite[Section 3]{ButlerRingel1987}, it follows that the isomorphism classes of all finitely generated indecomposable $\Omega$-modules are given by so-called string and band modules, which we describe in the next definition and remark.

\begin{definition}
\label{def:stringsbands}
Let $\Omega = k\mathcal{Q}/\mathcal{I}$ be a basic string algebra.
\begin{enumerate}
\item[(a)] A word $w$ of length $l\ge 1$ is a sequence $w=w_1 w_2\cdots w_l$ where each $w_i$ is either an arrow $\alpha$ of $\mathcal{Q}$ or its formal inverse $\alpha^{-1}$ such that $s(w_i)=e(w_{i+1})$ for $1\leq i\leq l-1$. Moreover, to each vertex $v$ of $\mathcal{Q}$, we associate a word $e_v$ of length $0$. We also denote by $e_v$ the primitive central idempotent of $\Omega$ associated with the vertex $v$.
\item[(b)] We define two different equivalence relations on the set of words:
\begin{enumerate}
\item[(i)] $\sim$ identifies each word with its inverse, and 
\item[(ii)] $\sim_r$ identifies each word with its rotations and their inverses, if they are defined.
\end{enumerate}
\item[(c)] A string is either of the form $w=e_v$ for a vertex $v$ of $\mathcal{Q}$, or
a representative of words $w=w_1\dots w_l$ of length $l\ge 1$ under the relation $\sim$ such that $w_i\neq w_{i+1}^{-1}$ for $1\leq i\leq l-1$ and no subword of $w$ or its inverse belongs to $\mathcal{I}$.
\item[(d)] A band is a representative of words $w=w_1\dots w_l$ of length $l\ge 1$ under the relation $\sim_r$ such that $w_i\neq w_{i+1}^{-1}$ for $1\le i \le l-1$,  $w_l \neq w_1^{-1}$, $w$ is not a proper power of a subword, and for all $a\ge 1$, the power $w^a$ is defined and no subword of $w^a$ or its inverse belongs to $\mathcal{I}$.
\end{enumerate}
\end{definition}

\begin{remark}
\label{rem:stringbandmodules}
Let $\Omega = k\mathcal{Q}/\mathcal{I}$ be a basic string algebra. By \cite[Section 3]{ButlerRingel1987} (see also \cite[Chapter II]{Erdmann1990}), the following list of $\Omega$-modules provides a complete set of representatives for the isomorphism classes of finitely generated indecomposable $\Omega$-modules.
\begin{enumerate}
\item[(i)] For each vertex $v$ of $\mathcal{Q}$, there is precisely one isomorphism class of simple $\Omega$-modules, corresponding to the string $e_v$ of length $0$. 
\item[(ii)] For each string $w=w_1 w_2\cdots w_l$ of length $l\ge 1$, there is precisely one isomorphism class of string modules $M(w)$ for $\Omega$ whose $k$-dimension is $l+1$. More precisely, the $\Omega$-action on $M(w)$ is given by a $k$-linear map
$$\varphi_w:\Omega \to \mathrm{Mat}_{l+1}(k)$$
such that for each vertex $v$ and each arrow $\alpha$ in $\mathcal{Q}$, the $(i,j)$-entry of $\varphi_w(v)$ and $\varphi_w(\alpha)$, respectively, is as follows:
\begin{eqnarray*}
\varphi_w(v)_{i,j} &=& \begin{cases}
1, & \mbox{if $1\le i=j\le l$ and $e(w_j)=v$, or }\\
& \mbox{if $i=j=l+1$ and $s(w_l)=v$,}\\
0, & \mbox{otherwise},
\end{cases}\\[2ex]
\varphi_w(\alpha)_{i,j} &=& \begin{cases}
1, & \mbox{if $1\le i=j-1\le l$ and $w_{j-1}=\alpha$, or }\\
& \mbox{if $2\le i=j+1\le l+1$ and $w_j=\alpha^{-1}$,}\\
0, & \mbox{otherwise}.
\end{cases}
\end{eqnarray*}
\item[(iii)] For each band $w=w_1 w_2\cdots w_l$ and each $n\in\mathbb{Z}^+$ and each $\mu\in k^\times$, there is precisely one isomorphism class of band modules $M(w,n,\mu)$ for $\Omega$ whose $k$-dimension is $l\cdot n$. More precisely, we can assume, without loss of generality, that $w_1$ is an arrow. Then the $\Omega$-action on $M(w,n,\mu)$ is given by a $k$-linear map
$$\varphi_{w,n,\mu}:\Omega \to \mathrm{Mat}_{n\cdot l}(k)$$
such that for each vertex $v$ and each arrow $\alpha$ in $\mathcal{Q}$, $\varphi_{w,n,\mu}(v)$ and $\varphi_{w,n,\mu}(\alpha)$ are $l\times l$ block matrices with blocks of size $n\times n$, where the $(i,j)$-block is as follows:
\begin{eqnarray*}
\varphi_{w,n,\mu}(v)_{i,j} &=& \begin{cases}
I_n, & \mbox{if $1\le i=j\le l$ and $e(w_j)=v$,}\\
0_n, & \mbox{otherwise},
\end{cases}\\[2ex]
\varphi_{w,n,\mu}(\alpha)_{i,j}&=& \begin{cases}
J_n(\mu), & \mbox{if $1=i=j-1$ and $w_1=\alpha$,}\\
I_n, & \mbox{if $2\le i=j-1\le l-1$ and $w_{j-1}=\alpha$, or}\\
& \mbox{if $i=l$ and $j=1$ and $w_l=\alpha$, or} \\
& \mbox{if $2\le i=j+1\le l$ and  $w_j=\alpha^{-1}$, or} \\
& \mbox{if $i=1$ and $j=l$ and $w_nl=\alpha^{-1}$,}\\
0_n, & \mbox{otherwise},
\end{cases}
\end{eqnarray*}
where $I_n$ denotes the $n\times n$ identity matrix, $0_n$ denotes the $n\times n$ zero matrix, and  $J_n(\mu)$ is the upper triangular Jordan block of size $n\times n$ with eigenvalue $\mu$.
\end{enumerate}
\end{remark}


\subsection{Indecomposable $kH$-modules}
\label{ss:K4char2}

\begin{definition}
\label{def:AB}
Define 
$$A:=\sigma-1 \quad \mbox{and} \quad B:=\tau-1.$$
\end{definition}

\begin{remark}
\label{rem:kH1}
We have
$$kH\cong k[x,y]/\langle x^2,y^2,xy-yx\rangle.$$
Since $A$ and $B$ from Definition \ref{def:AB} satisfy the relations
$$A^2=0, \quad B^2=0 \quad \mbox{and} \quad AB-BA = 0,$$
it follows that $kH=k[A,B]$. Since $\mbox{socle}(kH) = \langle AB-BA\rangle$, we obtain that 
$$\overline{kH}:=kH/\text{socle}(kH) \cong k[x,y]/\langle x^2,y^2,xy,yx \rangle$$
is a string algebra. By \cite[Section II.1.3]{Erdmann1990}, it follows that the finitely generated non-projective indecomposable $kH$-modules coincide with the finitely generated indecomposable $\overline{kH}$-modules, viewed as $kH$-modules by inflation. 
\end{remark}

The next notation is a consequence of Remark \ref{rem:kH1}.

\begin{nota}
\label{not:K4indecomposables}
The following infinite list of $kH$-modules provides a complete set of representatives for the isomorphism classes of finitely generated non-projective indecomposable $kH$-modules.
\begin{enumerate}[(1)]
\item There is precisely one isomorphism class of simple $kH$-modules corresponding to the unique string of length zero, which we will denote by $k$.
\item For each positive integer $n$, there are precisely two isomorphism classes of $kH$-modules of $k$-dimension $2n+1$ corresponding to strings of length $2n$, which we will denote as follows:
\begin{itemize}
\item[(i)] $M_{2n+1,1}=M_{2n+1,1}^{(A,B)}$ corresponding to the string $(B^{-1}A)^n$,
\item[(ii)] $M_{2n+1,2}=M_{2n+1,2}^{(A,B)}$ corresponding to the string $(AB^{-1})^n$.
\end{itemize}
\item For each positive integer $n$, there are precisely two isomorphism classes of indecomposable $kH$-modules of $k$-dimension $2n$ corresponding to strings of length $2n-1$, which we will denote as follows:
\begin{itemize}
\item[(i)] $N_{2n,0}=N_{2n,0}^{(A,B)}$ corresponding to the string $(A^{-1}B)^{n-1}A^{-1}$,
\item[(ii)] $N_{2n,\infty}=N_{2n,\infty}^{(A,B)}$ corresponding to the string $(B^{-1}A)^{n-1}B^{-1}$.
\end{itemize}
\item For each positive integer $n$ and each $\mu \in k^\times$, there is precisely one isomorphism class of indecomposable $kH$-modules of dimension $2n$ corresponding to the unique band $BA^{-1}$, which we will denote by $N_{2n,\mu}=N_{2n,\mu}^{(A,B)}$.
\end{enumerate}
\end{nota}


\subsection{Indecomposable $kG$-modules}
\label{ss:A4char2}

\begin{definition} 
\label{def:A4quiver}
Define $\Lambda=kQ/I$, where  
\begin{eqnarray*}
Q&=&\begin{tikzcd}
0 &&&& 1 \\
\\
&& 2
\arrow["\gamma_{10}", curve={height=-6pt}, from=1-1, to=1-5]
\arrow["\gamma_{21}", curve={height=-6pt}, from=1-5, to=3-3]
\arrow["\gamma_{02}", curve={height=-6pt}, from=3-3, to=1-1]
\arrow["\delta_{01}", curve={height=-6pt}, from=1-5, to=1-1]
\arrow["\delta_{20}", curve={height=-6pt}, from=1-1, to=3-3]
\arrow["\delta_{12}", curve={height=-6pt}, from=3-3, to=1-5]
\end{tikzcd}\quad \mbox{and}\\
I&=&\langle \delta_{01}\gamma_{10}-\gamma_{02}\delta_{20}, \delta_{12}\gamma_{21}-\gamma_{10}\delta_{01}, \delta_{20}\gamma_{02}-\gamma_{21}\delta_{12},\\
&&\gamma_{21}\gamma_{10}, \gamma_{10}\gamma_{02}, \gamma_{02}\gamma_{21},
\delta_{12}\delta_{20}, \delta_{20}\delta_{01},\delta_{01}\delta_{12}\rangle.
\end{eqnarray*}
\end{definition}

\begin{remark} 
\label{rem:kA4}
Let $\Lambda=kQ/I$ be as in Definition \ref{def:A4quiver}.
By \cite[Lemma V.2.4 and Corollary V.2.4.1]{Erdmann1990}, $\Lambda$ is isomorphic to $kG\cong kA_4$. In particular, $\Lambda$ is a symmetric algebra.
Since 
$$\text{socle}(\Lambda)=\langle \delta_{01}\gamma_{10}-\gamma_{02}\delta_{20}, \delta_{12}\gamma_{21}-\gamma_{10}\delta_{01},
\delta_{20}\gamma_{02}-\gamma_{21}\delta_{12}\rangle,$$
we have that 
$$\overline{\Lambda}:=\Lambda/\text{socle}(\Lambda) \cong kQ/(kQ^+)^2$$
is a string algebra, where $(kQ^+)^2$ is the ideal of $kQ$ generated by all paths of length $2$. By \cite[Section II.1.3]{Erdmann1990}, it follows that the finitely generated non-projective indecomposable $\Lambda$-modules coincide with finitely generated the indecomposable $\overline{\Lambda}$-modules, viewed as $\Lambda$-modules by inflation.
\end{remark}

The next notation is a consequence of Remark \ref{rem:kA4}.

\begin{nota}
\label{not:A4indecomposables}
The following infinite list of $\Lambda$-modules provides a complete set of representatives for the isomorphism classes of finitely generated non-projective indecomposable $\Lambda$-modules. In particular, since $\Lambda\cong kG$, this provides a complete set of representatives for the isomorphism classes of finitely generated non-projective indecomposable $kG$-modules.

\begin{enumerate}[(1)]
\item There are precisely three isomorphism classes of simple $\Lambda$-modules, 
which we will denote by $S_0$, $S_1$ and $S_2$. These correspond to the three strings of length zero $e_0$, $e_1$ and $e_2$, respectively, that are associated to the three vertices of $Q$.

\item For each positive integer $n$, there are precisely six isomorphism classes of $\Lambda$-modules of $k$-dimension $2n+1$ corresponding to strings of length $2n$, which we will denote as follows:
\begin{itemize}
\item[(i)]  $M_{2n+1,1,0}$ corresponding to the string 
$\delta_{01}^{-1}\gamma_{02}\delta_{12}^{-1}\gamma_{10}\delta_{20}^{-1}\gamma_{21}\delta_{01}^{-1}\cdots$,
\item[(ii)] $M_{2n+1,1,1}$ corresponding to the string 
$\delta_{12}^{-1}\gamma_{10}\delta_{20}^{-1}\gamma_{21}\delta_{01}^{-1}\gamma_{02}\delta_{12}^{-1}\cdots$,
\item[(iii)] $M_{2n+1,1,2}$ corresponding to the string 
$\delta_{20}^{-1}\gamma_{21}\delta_{01}^{-1}\gamma_{02}\delta_{12}^{-1}\gamma_{10}\delta_{20}^{-1}\cdots$,
\item[(iv)] $M_{2n+1,2,0}$ corresponding to the string 
$\gamma_{02}\delta_{12}^{-1}\gamma_{10}\delta_{20}^{-1}\gamma_{21}\delta_{01}^{-1}\gamma_{02}\cdots$,
\item[(v)] $M_{2n+1,2,1}$ corresponding to the string 
$\gamma_{10}\delta_{20}^{-1}\gamma_{21}\delta_{01}^{-1}\gamma_{02}\delta_{12}^{-1}\gamma_{10}\cdots$,
\item[(vi)] $M_{2n+1,2,2}$ corresponding to the string 
$\gamma_{21}\delta_{01}^{-1}\gamma_{02}\delta_{12}^{-1}\gamma_{10}\delta_{20}^{-1}\gamma_{21}\cdots$.
\end{itemize}

\item For each positive integer $n$, there are precisely six distinct isomorphism classes of $\Lambda$-modules of $k$-dimension $2n$ corresponding to strings of length $2n-1$, which we will denote as follows:
\begin{itemize}
\item[(i)] $N_{2n,0,0}$ corresponding to the string 
$\gamma_{02}^{-1}\delta_{01}\gamma_{21}^{-1}\delta_{20}\gamma_{10}^{-1}\delta_{12}\cdots$,
\item[(ii)] $N_{2n,0,1}$ corresponding to the string 
$\gamma_{10}^{-1}\delta_{12}\gamma_{02}^{-1}\delta_{01}\gamma_{21}^{-1}\delta_{20}\cdots$,
\item[(iii)] $N_{2n,0,2}$ corresponding to the string 
$\gamma_{21}^{-1}\delta_{20}\gamma_{10}^{-1}\delta_{12}\gamma_{02}^{-1}\delta_{01}\cdots$,
\item[(iv)] $N_{2n,\infty,0}$ corresponding to the string 
$\delta_{01}^{-1}\gamma_{02}\delta_{12}^{-1}\gamma_{10}\delta_{20}^{-1}\gamma_{21}\cdots$,
\item[(v)] $N_{2n,\infty,1}$ corresponding to the string 
$\delta_{12}^{-1}\gamma_{10}\delta_{20}^{-1}\gamma_{21}\delta_{01}^{-1}\gamma_{02}\cdots$,
\item[(vi)] $N_{2n,\infty,2}$ corresponding to the string 
$\delta_{20}^{-1}\gamma_{21}\delta_{01}^{-1}\gamma_{02}\delta_{12}^{-1}\gamma_{10}\cdots$
\end{itemize}

\item For each positive integer $n$ and each $\mu \in k^\times$, there is precisely one isomorphism class of $\Lambda$-modules of dimension $6n$ corresponding to the unique band 
$\delta_{01}\gamma_{21}^{-1}\delta_{20}\gamma_{10}^{-1}\delta_{12}\gamma_{02}^{-1}$, which we will denote by $B_{6n,\mu}$.
\end{enumerate}
\end{nota}


\subsection{Induction and restriction of modules for $kH$ and $kG$}
\label{ss:inducerestrict}

\begin{definition} 
\label{def:CD}
Define
$$C:=\gamma_{10}+\gamma_{21}+\gamma_{02} \quad \text{and} \quad D:=\delta_{01}+\delta_{12}+\delta_{20}$$
viewed as elements of $\Lambda$. 
\end{definition}

\begin{remark}
\label{rem:CD}
Let $C$ and $D$ be as in Definition \ref{def:CD}. Then
\begin{equation}
\label{eq:seemtoneed1}
\left\{
\begin{array}{rclrclrcl}
\gamma_{10}&=&e_1Ce_0, \quad &\gamma_{21}&=&e_2Ce_1, \quad &\gamma_{02}&=&e_0Ce_2,\\
\delta_{01}&=&e_0De_1, \quad &\delta_{12}&=&e_1De_2, \quad &\delta_{20}&=&e_2De_0.
\end{array}
\right.
\end{equation}
Moreover, we have
$$C^2=0,\quad D^2=0,\quad\mbox{and}\quad CD-DC=0.$$
In particular, the subalgebra $k[C,D]$ of $\Lambda$ is isomorphic to $kH$.
\end{remark}

The next notation is the analogue of Notation \ref{not:K4indecomposables} for finitely generated non-projective indecomposable $k[C,D]$-modules.

\begin{nota}
\label{not:kCDmodules}
The following infinite list of $k[C,D]$-modules provides a complete set of representatives for the isomorphism classes of finitely generated non-projective indecomposable $k[C,D]$-modules:
\begin{enumerate}[(1)]
\item $k$ corresponding to the unique string of length zero;
\item for each positive integer $n$, 
\begin{itemize}
\item[(i)] $M_{2n+1,1}^{(C,D)}$ corresponding to the string $(D^{-1}C)^n$,
\item[(ii)] $M_{2n+1,2}^{(C,D)}$ corresponding to the string $(CD^{-1})^n$;
\end{itemize}
\item for each positive integer $n$, 
\begin{itemize}
\item[(i)] $N_{2n,0}^{(C,D)}$ corresponding to the string $(C^{-1}D)^{n-1}C^{-1}$,
\item[(ii)] $N_{2n,\infty}^{(C,D)}$ corresponding to the string $(D^{-1}C)^{n-1}D^{-1}$;
\end{itemize}
\item for each positive integer $n$ and each $\mu \in k^\times$, $N_{2n,\mu}^{(C,D)}$ corresponding to the unique band $DC^{-1}$.
\end{enumerate}
\end{nota}

The next remark gives the relationship between $C,~D$ and $A,~B$ from Definition \ref{def:AB}.

\begin{remark} 
\label{rem:AB2}
The element $\rho$ of $G$ acts (on the left) by conjugation on the radical $\mathrm{rad}(kH)$, viewed as a $k$-subalgebra of $kG$. Identifying $kG=\Lambda$, it follows from the proof of \cite[Lemma V.2.4]{Erdmann1990} that $C$ and $D$ lie in $\mathrm{rad}(kH)$ and that they are eigenvectors of this action of $\rho$ with eigenvalues $\zeta$ and $\zeta^2$, respectively, where $\zeta\in k$ is a primitive cube root of unity. Moreover, $\{C,D,CD\}$ is a $k$-basis of $\mathrm{rad}(kH)$.

On the other hand, if $A,~B$ are as in Definition \ref{def:AB}, then $\{A,B,AB\}$ is also $k$-basis of $\text{rad}(kH)$. Using $\tau=\rho\sigma\rho^{-1}$ and $(\sigma\rho)^3=1$, we see that 
$$\rho\tau\rho^{-1}=\rho^2\sigma\rho^{-2}=(\rho^{-1}\sigma\rho^{-1})\rho^{-1} = \sigma\rho\sigma\rho^{-1}=\sigma\tau.$$
We obtain
$$\begin{array}{rcccccl}
\rho A \rho^{-1} &=& \rho\sigma\rho^{-1} -1 &=& \tau-1 &=& B,\\
\rho B \rho^{-1} &=& \rho\tau\rho^{-1}-1 &=& \sigma\tau-1 &=& A+B+AB ,\\
\rho AB \rho^{-1} &=&\multicolumn{3}{c}{B(A+B+AB)} &=& AB.
\end{array}$$
This means that the action of $\rho$ on $\text{rad}(kH)$, with respect to the basis $\{A,B,AB\}$, is given by the matrix
$$\begin{pmatrix}
    0 & 1 & 0\\
    1 & 1 & 0\\
    0 & 1 & 1\\
\end{pmatrix}$$
which has eigenvalues $\zeta$, $\zeta^2$ and $1$.  
Computing the corresponding eigenvectors and noticing that $C$ and $D$ are only unique up to non-zero scalar multiples, we make the following choices from now on:
\begin{eqnarray*}
C&=&\zeta A +\zeta^2 B + AB \quad\mbox{and}\\
D&=&A +\zeta^2 B +\zeta AB,
\end{eqnarray*}
which leads to
$$CD=\zeta AB.$$
We obtain the following change-of-basis matrices:
\begin{eqnarray}
\label{eq:changeAB1}
\begin{pmatrix}
        C\\
        D\\
        CD
    \end{pmatrix}&=&
    \begin{pmatrix}
        \zeta & \zeta^2 & 1\\
        1 & \zeta^2 & \zeta\\
        0 & 0 & \zeta
    \end{pmatrix}
    \begin{pmatrix}
        A\\
        B\\
        AB\\
    \end{pmatrix}\quad\mbox{and}\\
\label{eq:changeAB2}
    \begin{pmatrix}
        A\\
        B\\
        AB
    \end{pmatrix}&=&
    \begin{pmatrix}
        \zeta & \zeta & \zeta^2\\
        \zeta^2 & 1 & \zeta^2\\
        0 & 0 & \zeta^2
    \end{pmatrix}
    \begin{pmatrix}
        C\\
        D\\
        CD
    \end{pmatrix}.
 \end{eqnarray}
\end{remark}

\begin{nota} 
\label{not:arrgh}
Using Notation \ref{not:general} and Definitions \ref{def:A4quiver} and \ref{def:AB}, 
we identify $kG$ with $\Lambda$ and we identify $kH=k[A,B]$ with $k[C,D]$, where we use Equations (\ref{eq:changeAB1}) and (\ref{eq:changeAB2}) to go back and forth between $A,~B$ and $C,~D$.
\end{nota}

Straightforward matrix calculations show the following connection between the finitely generated non-projective indecomposable $kH$-modules with respect to the action of $A,~B$ versus the action of $C,~D$:

\begin{lemma} 
\label{lem:bandconnection}
We use Notation $\ref{not:arrgh}$ together with Notations $\ref{not:K4indecomposables}$ and $\ref{not:kCDmodules}$.
\begin{itemize}
\item[(i)] For all positive integers $n$ and $x\in\{1,2\}$, we have
$$M_{2n+1,x}^{(C,D)} \cong M_{2n+1,x}^{(A,B)}.$$
\item[(ii)] For all positive integers $n$ and $\phi\in k\cup\{\infty\}$, we have
$$N_{2n,\phi}^{(C,D)} \cong N_{2n,\zeta(1+\zeta\phi)(1+\phi)^{-1}}^{(A,B)}.$$
In particular, $N_{2n,0}^{(C,D)} \cong N_{2n,\zeta}^{(A,B)}$ and $N_{2n,\infty}^{(C,D)} \cong N_{2n,\zeta^2}^{(A,B)}$, and
$N_{2n,\zeta^2}^{(C,D)} \cong N_{2n,0}^{(A,B)}$ and
$N_{2n,1}^{(C,D)} \cong N_{2n,\infty}^{(A,B)}$.
\end{itemize}
\end{lemma}

\begin{remark}
\label{rem:morebandconnection}
Suppose $n$ is a positive integer and $\lambda\in k\cup\{\infty\}$. Define
$$\phi = \phi(\lambda):=\frac{\zeta+\lambda}{\zeta^2+\lambda}.$$
Then $\phi$ is the unique element in $k\cup\{\infty\}$ such that $\lambda = \zeta(1+\zeta\phi)(1+\phi)^{-1}$. It follows that
$$\frac{1+\lambda}{\lambda} = \frac{\zeta(1+\zeta^2\phi)}{1+\zeta\phi}\quad\mbox{and}\quad
\frac{1}{1+\lambda} = \frac{\zeta(1+\phi)}{1+\zeta^2\phi}.$$
Therefore,
if $N_{2n,\phi}^{(C,D)} \cong N_{2n,\lambda}^{(A,B)}$, then
$$N_{2n,\zeta\phi}^{(C,D)} \cong N_{2n,(1+\lambda)\lambda^{-1}}^{(A,B)}\quad\mbox{and}\quad
N_{2n,\zeta^2\phi}^{(C,D)} \cong N_{2n,(1+\lambda)^{-1}}^{(A,B)}.$$
Notice that if $\lambda\in\{\zeta,\zeta^2\}$, then $\lambda=\frac{1+\lambda}{\lambda}=\frac{1}{1+\lambda}$. 
\end{remark}

The next result gives the restrictions of the non-projective indecomposable $kG$-modules to $kH$ and the inductions of the non-projective indecomposable $kH$-modules to $kG$.

\begin{lemma} 
\label{lem:restrictA4toK4}
Suppose Assumption $\ref{ass:general}$ holds, and use Notations $\ref{not:general}$ and $\ref{not:arrgh}$. Consider the indecomposable $kG$-modules listed in Notation $\ref{not:A4indecomposables}$ and the indecomposable $kH$-modules listed in 
Notation $\ref{not:kCDmodules}$. Let $n\in\mathbb{Z}^+$, $x\in \{1,2\}$, $\ast\in\{0,\infty\}$, and $\phi\in k^\times$.
\begin{itemize}
\item[(a)] For $i\in\{0,1,2\}$, we have
\begin{itemize}
\item[(1)] $\mathrm{Res}^G_H S_i \cong k \quad (\mbox{the trivial simple $kH$-module})$,
\item[(2)] $\mathrm{Res}^G_H M_{2n+1,x,i} \cong M_{2n+1,x}^{(C,D)}$,
\item[(3)] $\mathrm{Res}^G_H N_{2n,\ast,i} \cong N_{2n,\ast}^{(C,D)}$, 
\item[(4)] $\mathrm{Res}^G_H B_{6n,\phi^3} \cong N_{2n,\phi}^{(C,D)} \oplus N_{2n,\zeta\phi}^{(C,D)} \oplus N_{2n,\zeta^2\phi}^{(C,D)}$.
\end{itemize}
\item[(b)] We have
\begin{itemize}
\item[(1)] $\mathrm{Ind}_H^G k \cong S_0\oplus S_1\oplus S_2$,
\item[(2)] $\mathrm{Ind}^G_H M_{2n+1,x}^{(C,D)} \cong M_{2n+1,x,0} \oplus M_{2n+1,x,1} \oplus M_{2n+1,x,2}$,
\item[(3)] $\mathrm{Ind}^G_H N_{2n,\ast}^{(C,D)} \cong N_{2n,\ast,0} \oplus N_{2n,\ast,1} \oplus N_{2n,\ast,2}$,
\item[(4)] $\mathrm{Ind}^G_H N_{2n,\phi}^{(C,D)} \cong B_{6n,\phi^3}$.
\end{itemize}
\end{itemize}
\end{lemma}

\begin{proof}
The isomorphisms in parts (a)(1) through (a)(3) follow almost immediately from the observation that the action of each of $\gamma_{10},~\gamma_{21},~\gamma_{02}$ restricts to the action of $C$ and the action of each of $\delta_{01},~\delta_{12},~\delta_{20}$ restricts to the action of $D$. 

For the isomorphism in part (a)(4), we first prove the isomorphism in part (b)(4). By Remark \ref{rem:AB2}, we know that 
\begin{equation}
\label{eq:rhoaction}
\rho C \rho^{-1} = \zeta C\quad \mbox{and}\quad \rho D \rho^{-1} = \zeta^2 D.
\end{equation}
Let $\{b_{1,1},\ldots,b_{1,n},b_{2,1},\ldots,b_{2,n}\}$ be a $k$-basis of $N_{2n,\phi}^{(C,D)}$ on which $C$ and $D$ act as:
$$C\longleftrightarrow\begin{pmatrix}
            0_n & I_n\\
            0_n & 0_n
        \end{pmatrix}\quad \mbox{and} \quad
     D\longleftrightarrow\begin{pmatrix}
            0_n & J_n(\phi)\\
            0_n & 0_n
        \end{pmatrix}.$$
By the proof of \cite[Lemma V.2.4]{Erdmann1990}, the primitive central idempotents of $\Lambda=kG$ are $e_0,~e_1,~e_2$, where
\begin{equation}
\label{eq:A4idempotents}
e_i = 1 + \zeta^{-i}\rho+ \zeta^{-2i}\rho^2,
\end{equation}
which implies $\rho e_i=\zeta^ie_i$, for $i\in\{0,1,2\}$. 
For $j_1\in\{1,2\}$ and $j_2\in\{1,\ldots,n\}$, we have
\begin{eqnarray*}
e_0b_{j_1,j_2} &=& b_{j_1,j_2} + \rho b_{j_1,j_2} + \rho^2 b_{j_1,j_2},\\
e_1b_{j_1,j_2} &=& b_{j_1,j_2} + \zeta^2\rho b_{j_1,j_2} + \zeta\rho^2 b_{j_1,j_2},\\
e_2b_{j_1,j_2} &=& b_{j_1,j_2} + \zeta\rho b_{j_1,j_2} + \zeta^2\rho^2 b_{j_1,j_2}.
\end{eqnarray*}
By (\ref{eq:rhoaction}), we have
$$C\rho = \zeta^2 \rho C \quad\mbox{and}\quad D\rho =\zeta \rho D .$$
This implies, for all $1\le j\le n$,
\begin{equation}
\label{eq:ABactionband}
\left\{ \begin{array}{rcl@{\quad}rcl@{\quad}rcl}
C e_0 b_{2,j} &=& e_1 b_{1,j},&Ce_1 b_{2,j} &=& e_2 b_{1,j}, & C e_2 b_{2,j} &=& e_0 b_{2,j},\\
D e_0 b_{2,j} &=& J_n(\phi)e_2 b_{1,j}, &D e_1 b_{2,j} &=& J_n(\phi)e_0 b_{1,j}, &D e_2 b_{2,j} &=& J_n(\phi)e_1 b_{1,j}.
\end{array}\right.
\end{equation}
Since $\phi\in k^\times$ and since
$$\mathrm{Ind}^G_H N_{2n,\phi}^{(C,D)} = N_{2n,\phi}^{(C,D)} \oplus \rho\, N_{2n,\phi}^{(C,D)} \oplus \rho^2\,N_{2n,\phi}^{(C,D)},$$
we obtain the following ordered $k$-basis for $\mathrm{Ind}^G_H N_{2n,\phi}^{(C,D)}$:
$$\{e_0 b_{1,j}\}_{j=1}^n \cup 
\{J_n(\phi)^2e_1 b_{2,j}\}_{j=1}^n \cup 
\{J_n(\phi)^2e_2 b_{1,j}\}_{j=1}^n \cup 
\{J_n(\phi)e_0 b_{2,j}\}_{j=1}^n \cup 
\{J_n(\phi)e_1 b_{1,j}\}_{j=1}^n \cup
\{e_2 b_{2,j}\}_{j=1}^n.$$
Because $k$ has characteristic two, it is straightforward to show that the matrix $J_n(\phi)^3$ is similar to the matrix $J_n(\phi^3)$.
Using Remark \ref{rem:CD}, and in particular (\ref{eq:seemtoneed1}), this implies the isomorphism in part (b)(4).

By Mackey's Theorem (see, for example, \cite[Lemma III.8.7]{Alperin1986}), every $kH$-module $M$ is a direct summand of $\mathrm{Res}^G_H\,\mathrm{Ind}^G_H M$.
Since $\phi\in k^\times$ and since $\phi^3=(\zeta\phi)^3=(\zeta^2\phi)^3$, it follows that 
$$N_{2n,\phi}^{(C,D)}\oplus N_{2n,\zeta \phi}^{(C,D)} \oplus N_{2n,\zeta^2 \phi}^{(C,D)}$$
must be isomorphic to a direct summand of $\mathrm{Res}^G_H B_{6n,\phi^3}$. Comparing $k$-dimensions, we obtain the isomorphism in part (a)(4).

To show the remaining isomorphisms in parts (b)(1) through (b)(3), let $U$ be one of the indecomposable $kH$-modules $k$ or $M_{2n+1,x}^{(C,D)}$ or $N_{2n,\ast}^{(C,D)}$, respectively, and let $d$ be the $k$-dimension of $U$. Let $\{b_1,\ldots, b_d\}$ be a $k$-basis of $U$ on which $C$ and $D$ act according to Remark \ref{rem:stringbandmodules} using the string $w_U=w_1\cdots w_{d-1}$ associated to $U$, as given in Notation \ref{not:kCDmodules}. In particular, if $U=k$, then $w_U$ is the unique string of length zero. Utilizing the idempotents $e_0$, $e_1$ and $e_2$ from (\ref{eq:A4idempotents}), we see that
$$\{e_i b_j\;:\; i\in\{0,1,2\}, j\in\{1,\ldots,d\} \}$$
is a $k$-basis of $\mathrm{Ind}^G_H U$. 
Similarly to (\ref{eq:ABactionband}), we see that, for $1\le j\le d$,
\begin{equation}
\label{eq:ABactionstring}
\left\{ \begin{array}{rcl@{\quad}rcl@{\quad}rcl}
C e_0 b_j &=& e_1 Cb_j,&Ce_1 b_j &=& e_2 Cb_j, & C e_2 b_j &=& e_0 Cb_j,\\
D e_0 b_j &=& e_2 Cb_j, &D e_1 b_j &=& e_0 Cb_j, &D e_2 b_j &=& e_1 Cb_j,
\end{array}\right.
\end{equation}
where, for $E\in\{C,D\}$,
$$E b_j=\begin{cases}
b_{j-1}&\mbox{ if } j\ge 2 \mbox{ and } w_{j-1}=E,\\
b_{j+1}&\mbox{ if } j\le d-1 \mbox{ and } w_j=E^{-1},\\
0&\mbox{ otherwise}.
\end{cases}$$
Fix $a\in\{0,1,2\}$. Then it follows from (\ref{eq:ABactionstring}) and from (\ref{eq:seemtoneed1}) that, for each $j\in\{2,\ldots,d\}$, there exists a unique $i(a,j)\in\{0,1,2\}$ such that the $k$-span of
$$\{e_ab_1, e_{i(a,2)}b_2, \ldots, e_{i(a,d)}b_d\}$$
is a $kG$-module $U_a$ with $\mathrm{Res}^G_H U_a = U$. Moreover, defining $i(a,1):=a$, we see that, for all $1\le j\le d$, 
$$\{i(0,j), i(1,j), i(2,j)\} = \{0,1,2\}.$$
We obtain
$$\mathrm{Ind}^G_H U = U_0\oplus U_1\oplus U_2.$$
Using parts (a)(1) through (a)(3) together with Notation \ref{not:A4indecomposables}, this proves the isomorphisms in parts (b)(1) through (b)(3).
\end{proof}


\bibliographystyle{plain}

\begin{thebibliography}{10}
\bibitem{Alperin1986}
J.~L. Alperin.
\newblock Local representation theory.
\newblock Cambridge Studies in Advanced Mathematics, 11. Cambridge University Press, Cambridge, 1986.

\bibitem{BleherCamacho2023}
F.~M.~Bleher, and N.~Camacho.
\newblock Holomorphic differentials of Klein four covers.
\newblock {\em J. Pure Appl. Algebra} 227 (2023), Paper No.107384. 

\bibitem{BleherChinburgKontogeorgis2020}
F.~M. Bleher, T.~Chinburg, and A.~Kontogeorgis.
\newblock Galois structure of the holomorphic differentials of curves.
\newblock {\em J. Number Theory}, 216 (2020), 1--68.

\bibitem{BCPS2017}
F.~M.~Bleher, T.~Chinburg, B.~Poonen and P.~Symonds.
\newblock Automorphisms of Harbater-Katz-Gabber curves.
\newblock {\em Math. Ann.} 368 (2017), 811--836.

\bibitem{BondarenkoDrozd1977}
V.~M.~Bondarenko and J.~A.~Drozd.
\newblock The representation type of finite groups. (Russian)
Modules and representations.
\newblock Zap. Nau\v{c}n. Sem. Leningrad. Otdel. Mat. Inst. Steklov. (LOMI) 71 (1977), 24--41, 282.

\bibitem{Brenner1970}
S.~Brenner.
\newblock Modular representations of $p$groups.
\newblock  {\em J. Algebra} 15 (1970), 89--102.

\bibitem{MargaritaThesis2025}
M.~Bustos Gonzalez.
\newblock {\em Holomorphic {D}ifferentials of {A}lternating {F}our {C}overs}.
\newblock ProQuest LLC, Ann Arbor, MI, 2025.
\newblock Thesis (Ph.D.)--The University of Iowa.

\bibitem{ButlerRingel1987}
M.~C.~R.~Butler and C.~M.~Ringel.
\newblock Auslander-Reiten sequences with few middle terms and applications to string algebras.
\newblock {\em Comm. Algebra} 15 (1987), 145--179.

\bibitem{ChevalleyWeil1934}
C.~Chevalley, A.~Weil, and E.~Hecke.
\newblock \"{U}ber das {V}erhalten der {I}ntegrale 1. {G}attung bei
  {A}utomorphismen des {F}unktionenk\"orpers.
\newblock {\em Abh. Math. Sem. Univ. Hamburg}, 10 (1934), 358--361.

\bibitem{CurtisReiner}
C.~W.~Curtis and I.~Reiner. 
\newblock Methods of representation theory. With applications to finite groups and orders. Vol. II. 
\newblock Pure and Applied Mathematics. John Wiley \& Sons, Inc., New York, 1987. 

\bibitem{Erdmann1990}
K.~Erdmann.
\newblock Blocks of tame representation type and related algebras. 
\newblock Lecture Notes in Mathematics, 1428. Springer-Verlag, Berlin, 1990

\bibitem{Garnek2022}
J.~Garnek.
\newblock $p$-group {G}alois covers of curves in characteristic $p$.
\newblock {\em Doc. Math.} 30 (2025), 347--377.

\bibitem{HasseCrelle1935}
H.~Hasse.
\newblock Theorie der relativ-zyklischen algebraischen {F}unktionenk\"{o}rper,
  insbesondere bei endlichem {K}onstantenk\"{o}rper.
\newblock {\em J. Reine Angew. Math.} 172 (1935), 37--54.

\bibitem{Hecke1928}
E.~Hecke.
\newblock \"{U}ber ein {F}undamentalproblem aus der {T}heorie der elliptischen 
  {M}odulfunktionen.
\newblock {\em Abh. Math. Sem. Univ. Hamburg} 6 (1928), 235--257.

\bibitem{Higman1954}
D.~G.~Higman.
\newblock Indecomposable representations at characteristic $p$.
\newblock {\em Duke Math. J.} 21 (1954), 377--381.

\bibitem{Kani1986}
E.~Kani.
\newblock The {G}alois-module structure of the space of holomorphic
  differentials of a curve.
\newblock {\em J. Reine Angew. Math.} 367 (1986), 187--206.

\bibitem{KaranikolopoulosKontogeorgis2013}
S.~Karanikolopoulos and A.~Kontogeorgis.
\newblock Representation of cyclic groups in positive characteristic and
  {W}eierstrass semigroups.
\newblock {\em J. Number Theory} 133 (2013), 158--175.

\bibitem{Kock2004}
B.~K\"{o}ck.
\newblock Galois structure of {Z}ariski cohomology for weakly ramified covers
  of curves.
\newblock {\em Amer. J. Math.} 126 (2004), 1085--1107.

\bibitem{Kock2024}
B.~K\"{o}ck and L.~Laurent.
\newblock The canonical representation of the Drinfeld curve.
\newblock {\em Math. Nachr.} 297 (2024), 4115--4120.

\bibitem{KontogeorgisTsouknidas2020}
A.~Kontogeorgis and I.~Tsouknidas.
\newblock A cohomological treatise of HKG-covers with applications to the Nottingham group.
\newblock {\em J. Algebra} 555 (2020), 325--345.

\bibitem{MarquesWard2018}
S.~Marques and K.~Ward.
\newblock Holomorphic differentials of certain solvable covers of the
  projective line over a perfect field.
\newblock {\em Math. Nachr.} 291 (2018), 2057--2083.

\bibitem{Nakajima1986}
S.~Nakajima.
\newblock Galois module structure of cohomology groups for tamely ramified
  coverings of algebraic varieties.
\newblock {\em J. Number Theory} 22 (1986), 115--123.

\bibitem{Obus2017}
A.~Obus.
\newblock A generalization of the Oort conjecture,
\newblock {\em Comment. Math. Helv.} 92 (2017), 551--620.

\bibitem{RzedowskiCVillaSMadan1996}
M.~Rzedowski-Calder{\'o}n, G.~Villa-Salvador, and M.~L. Madan.
\newblock Galois module structure of holomorphic differentials in
  characteristic {$p$}.
\newblock {\em Arch. Math. (Basel)} 66(1996), 150--156.

\bibitem{Stichtenoth2009}
H.~Stichtenoth.
\newblock {\em Algebraic function fields and codes}.
\newblock Graduate Texts in Mathematics, 254. Springer-Verlag, Berlin, 2009.

\bibitem{ValentiniMadan1981}
R.~C. Valentini and M.~L. Madan.
\newblock Automorphisms and holomorphic differentials in characteristic {$p$}.
\newblock {\em J. Number Theory}, 13(1981), 106--115.

\end{thebibliography}

\end{document}